\documentclass[a4paper, 10pt]{amsart}
\usepackage{amsthm}
\usepackage[]{amsmath}
\usepackage{amssymb}
\usepackage{enumerate}
\usepackage{tabularx}
\usepackage[]{color}
\usepackage[left=3cm, right=3cm]{geometry}
\usepackage[colorlinks]{hyperref}
\usepackage{tikz}
\usepackage{multirow}
\usepackage{subcaption}
\usepackage{algorithm}
\usepackage{algorithmic}

\newcommand{\bm}[1]{\boldsymbol{#1}}
\newcommand{\bmr}[1]{\bm{\mr{#1}}}
\newcommand{\lj}{[ \hspace{-2pt} [}
\newcommand{\rj}{] \hspace{-2pt} ]}
\newcommand{\mb}[1]{\mathbb{#1}}
\newcommand{\mc}[1]{\mathcal{#1}}
\newcommand{\mr}[1]{\mathrm{#1}}
\newcommand{\jump}[1]{\lj #1 \rj}
\newcommand{\aver}[1]{ \{#1\}  }
\newcommand{\wt}[1]{ \widetilde{ #1}}
\newcommand{\wh}[1]{ \widehat{ #1}}
\newcommand{\tr}[1]{\ifmmode \mathrm{tr}\left( #1 \right) \else 
\text{tr} \left( #1 \right) \fi }
\newcommand{\vect}[2]{ \begin{bmatrix} #1 \\ #2 \\ \end{bmatrix}}
\newcommand{\vech}[3]{ \begin{bmatrix} #1 \\ #2 \\ #3 \\ \end{bmatrix}}
\newcommand{\matt}[4]{ \begin{bmatrix} #1 & #3 \\ #2 & #4 \\
\end{bmatrix}}
\newcommand\curl{\ifmmode \mathrm{curl} \else \text{curl}\fi}
\renewcommand\det{\ifmmode \mathrm{det} \else \text{det} \fi}

\def\Re{\ifmmode \mathrm{Re} \else \text{Re} \fi}

\newcommand\MTh{\mc{T}_h}
\newcommand\MEh{\mc{E}_h}
\newcommand\un{\bm{\mr n}}
\renewcommand{\d}[1]{\mathrm d \boldsymbol{#1}}
\newcommand\pnorm[1]{\| #1 \|_{\bm{\mathrm{p}}}}
\newcommand\unorm[1]{\| #1 \|_{\bm{\mathrm{u}}}}
\newcommand\Unorm[1]{\| #1 \|_{\bm{\mathrm{U}}}}

\newtheorem{assumption}{Assumption}
\newtheorem{theorem}{Theorem}
\newtheorem{lemma}{Lemma}

\numberwithin{equation}{section}
\numberwithin{lemma}{section}
\numberwithin{theorem}{section}

\title[Reconstructed Discontinuous Approximation to Stokes Equation]{
  Reconstructed Discontinuous Approximation to Stokes Equation in A
  Sequential Least Squares Formulation}

\author[R. Li]{Ruo Li} \address{CAPT, LMAM and School of Mathematical
  Sciences, Peking University, Beijing 100871, P.R. China}
\email{rli@math.pku.edu.cn}

\author[F.-Y. Yang]{Fanyi Yang} \address{School of Mathematical
  Sciences, Peking University, Beijing 100871, P.R. China}
\email{yangfanyi@pku.edu.cn}

\begin{document}

% vim:spell:tw=70:fo+=Mn:cc=70
\begin{abstract}
  We propose a new least squares finite element method to solve the
  Stokes problem with two sequential steps. The approximation spaces
  are constructed by patch reconstruction with one unknown per
  element. For the first step, we reconstruct an approximation space
  consisting of piecewise curl-free polynomials with zero trace. By
  this space, we minimize a least squares functional to obtain the
  numerical approximations to the gradient of the velocity and the
  pressure. In the second step, we minimize another least squares
  functional to give the solution to the velocity in the reconstructed
  piecewise divergence-free space.  We derive error estimates for all
  unknowns under $L^2$ norms and energy norms.  Numerical results in
  two dimensions and three dimensions verify the convergence rates and
  demonstrate the great flexibility of our method.


\end{abstract}

%%% Local Variables:
%%% mode: latex
%%% TeX-master: "article"
%%% End:

\maketitle

% vim:spell:tw=70:fo+=Mn:cc=70
\section{Introduction}
\label{sec_introduction}

The Stokes problem, which models a viscous and incompressible fluid
flow, is a linearized version of the full Navier-Stokes equation
neglecting the nonlinear convective term. Consequently, the Stokes
problem has a large number of applications especially for the time
discretization to the Naiver-Stokes problem. Reliable and efficient
numerical methods for the Stokes problem have been extensively studied
in the references. Among these methods, there were many efforts
devoted to develop mixed finite element methods based on the weak
formulation of the Stokes problem. A key issue of classical mixed
finite element methods is the choice of element types. The pair of
finite element spaces are required to satisfy the stability condition,
such as the inf-sup condition. We refer the readers to
\cite{boffi2013mixed, brenner2007mathematical, girault1986finite} for
some examples in classical mixed finite element methods.

The least squares finite element methods for the Stokes problem have
been developed in \cite{Cai1997first, Bochev2012locally,
Jiang1990least, Bochev1994analysis, Shin2016least,
Bochev2013nonconforming, Liu2013hybrid, Bertrand2019least}. For these
methods, least squares principle together with finite element methods
can offer the advantage of circumventing the inf-sup condition arising
in mixed methods. Bochev and Gunzburger developed a least squares
approach based on rewriting the velocity-vorticity-pressure
formulation as a first-order elliptic system
\cite{Bochev1994analysis}. Cai and his coworkers developed the least
squares finite element method based on the $L^2$ norm residual and
$C^0$ spaces for the Stokes problem, we refer to \cite{Cai1997first,
Bertrand2019least, Cai2000least, Cai2000first} for more details. Liu
et al. developed a hybrid least squares finite element method based on
continuous finite element spaces. This method attempts to combine the
advantages of FOSLS and FOSLL* \cite{Liu2013hybrid}. The works
introduced above are based on conforming finite element spaces and
such continuous least squares methods are general techniques in
numerical methods. We refer to \cite{Bochev2009least} and the
references therein for an overview of least squares finite element
methods. Based on discontinuous approximation, the discontinuous least
squares finite element methods have also been developed for many
problems including the Stokes problem, and we refer to
\cite{Bochev2013nonconforming, Bochev2012locally,
Bensow2005discontinuous, Bensow2005div, li2019least} for more details.

In this paper, we propose a new least squares finite element method
with the reconstructed discontinuous approximation. The novelty is
that we propose three specific approximation spaces which allow us to
solve the Stokes problem in two sequential steps. The sequential
process is motivated from the idea in \cite{li2019sequential,
Cai1997first} to define two least-squares-type functionals to
approximate unknowns sequentially. The feasibility of this method is
based on the new approximation spaces which are obtained by solving
local least squares problems on each element. In the first step, we
reconstruct an approximation space that consists of piecewise
irrotational polynomials with zero trace to approximate the gradient
of the velocity. This space is an extension of the space proposed in
\cite{li2012efficient}, which will also be used in this step to
approximate the pressure. The functions in both approximation spaces
may be discontinuous across interior faces and we define a least
squares functional with the weak imposition of the continuity across
the interior faces to seek numerical solutions in approximation to the
gradient and the pressure. In the second step, we reconstruct a
piecewise divergence-free polynomial space to approximate the
velocity. Here this reconstructed approximation space is also a
generalization of the space in \cite{li2012efficient}. We minimize
another least squares functional, together with the numerical gradient
obtained in the first step, to solve the numerical solution for the
velocity. For the error estimate, we introduce a series of projection
operators to derive the convergence rates for all variables with
respect to $L^2$ norms and energy norms. We prove that the
convergence orders under energy norms are all optimal and the $L^2$
errors for all variables can only be proved to be sub-optimal.  We
conduct a series of numerical examples in two dimensions and three
dimensions to confirm our theoretical error estimates. In addition,
we observe that the $L^2$ errors for all unknowns are optimally
convergent for approximation spaces of odd orders. Another advantage
of our method is the implementation is quite simple. The different
types of the reconstruction can be implemented in a uniform way. We
present the details to the computer implementation of our method in
Appendix \ref{sec_appendix}. 

The rest of our paper is organized as follow. In Section
\ref{sec_preliminaries}, we give the notations that will be used in
this paper. In Section \ref{sec_space}, we introduce the
reconstruction operators and the corresponding approximation spaces.
The approximation properties of spaces are also presented in this
section. In Section \ref{sec_LSMStokes}, we define two least squares
functionals for sequentially solving the Stokes problem with
reconstructed spaces. We also prove the error estimates under $L^2$
norms and energy norms. In Section \ref{sec_numerical_results}, we
present a series of numerical results in two dimensions and three
dimensions to illustrate the accuracy and the flexibility of our
method.

%%% Local Variables:
%%% mode: latex
%%% TeX-master: "article"
%%% End:

% vim:spell:tw=70:fo+=Mn:cc=70
\section{Preliminaries}
\label{sec_preliminaries}
We let $\Omega$ be a convex bounded polygonal (polyhedral) domain in
$\mb{R}^d (d = 2, 3)$ with the boundary $\partial \Omega$. We denote
by $\MTh$ a set of polygonal (polyhedral) elements which partition the
domain $\Omega$. We denote by $\MEh^i$ the set of interior faces of
$\MTh$ and by $\MEh^b$ the set of the faces that are on the boundary
$\partial \Omega$. Let $\MEh = \MEh^i \cup \MEh^b$ be the set of all
$d - 1$ dimensional faces. Further, for any element $K \in \MTh$ and
any face $e \in \MEh$ we set $h_K$ as the diameter of the element $K$
and $h_e$ as the size of the face $e$, and the mesh size is denoted by
$h = \max_{K \in \MTh} h_K$. It is assumed that the partition $\MTh$
is shape-regular in the sense of that: there exist
\begin{itemize}
  \item an integer $N$ that is independent of $h$;
  \item a number $\sigma > 0$ that is independent of $h$;
  \item a compatible sub-decomposition $\wt{\mathcal T}_h$ into
    triangles (tetrahedrons);
\end{itemize}
such that 
\begin{itemize}
  \item every polygon (polyhedron) $K$ in $\MTh$ admits a
    decomposition $\wt{\mathcal T}_{h|K}$ which has at most $N$
    triangles (tetrahedrons);
  \item for any triangle (tetrahedron) $\wt{K} \in \wt{\mathcal T}_h$,
    the ratio $h_{\wt{K}} / \rho_{\wt{K}}$ is bounded by $\sigma$
    where $h_{\wt{K}}$ denotes the diameter of $\wt{K}$ and
    $\rho_{\wt{K}}$ denotes the radius of the largest disk (ball)
    inscribed in $\wt{K}$.
\end{itemize}
The above regularity requirements have several consequences
\cite{Brezzi2009mimetic, li2016discontinuous}: 
\begin{itemize}
  \item[M1] there exists a positive constant $\sigma_v$ that is
    independent of $h$  such that $\sigma_v h_K \leq h_e$ for any
    element $K \in \MTh$ and any face $e \subset \partial K$;
  \item[M2][{\it trace inequality}] there exists a constant $C$ that
    is independent of $h$ such that 
    \begin{equation}
      \| v \|_{L^2(\partial K)}^2 \leq C \left( h_K^{-1}
      \|v\|_{L^2(K)}^2 + h_K \| \nabla v \|_{L^2(K)}^2 \right), \quad
      \forall v \in H^1(K);
      \label{eq_traceinequality}
    \end{equation}
  \item[M3][{\it inverse inequality}] there exists a constant $C$ that
    is independent of $h$ such that 
    \begin{equation}
      \| \nabla v \|_{L^2(K)} \leq Ch_K^{-1} \| v \|_{L^2(K)}, \quad
      \forall v \in \mb P_k(K),
      \label{eq_inverseinequality}
    \end{equation}
    where $\mb{P}_k(\cdot)$ is the polynomial space of degree less
    than $k$.
\end{itemize}
For the sub-decomposition $\wt{\mc{T}}_h$, we denote by
$\wt{\mc{E}}_h$ the collection of all $d - 1$ dimensional faces in
$\wt{\mc{T}}_h$, and we decompose $\wt{\mc{E}}_h$ into $\wt{\mc{E}}_h
= \wt{\mc{E}}_h^i \cup \wt{\mc{E}}_h^b$ where $\wt{\mc{E}}_h^i$ and
$\wt{\mc{E}}_h^b$ are the sets of interior faces and boundary faces,
respectively. From the regularity assumption, it is clear that $\MEh
\subset \wt{\mc{E}}_h$, $\MEh^i \subset \wt{\mc{E}}_h^i$ and $\MEh^b
\subset \wt{\mc{E}}_h^b$.

Then we introduce the trace operators that are associated with weak
formulations. Let $e$ be an interior face shared by two
adjacent elements $K^+$ and $K^-$, and we let $\un^+$ and $\un^-$ be
the unit outward normal on $e$ corresponding to $\partial K^+$
and $\partial K^-$, respectively. For the scalar-valued function $v$
and the $d$ dimensional vector-valued function $\bm{q}$ and $d \times
d$ dimensional tensor-valued function $\bm{\tau}$, we define $v^+ :=
v|_{e \subset \partial K^+}$, $v^- := v|_{e \subset \partial K^-}$,
$\bm{q}^+ := \bm{q}|_{e \subset \partial K^+}$, $\bm{q}^- :=
\bm{q}|_{e \subset \partial K^-}$, $\bm{\tau}^+ := \bm{\tau}|_{e
\subset \partial K^+}$, $\bm{\tau}^- := \bm{\tau}|_{e \subset \partial
K^-}$. The average operator $\aver{\cdot}$ on $e$ is defined as 
\begin{displaymath}
  \aver{v} := \frac{1}{2} \left( v^+ + v^- \right), \quad
  \aver{\bm{q}} := \frac{1}{2} \left( \bm{q}^+ + \bm{q}^- \right),
  \quad \aver{\bm{\tau}} := \frac{1}{2} \left( \bm{\tau}^+ +
  \bm{\tau}^- \right).
\end{displaymath}
The jump operator $\jump{\cdot}$ on $e$ is defined as 
\begin{displaymath}
  \begin{aligned}
    \jump{v} &:= v^+ \un^+ + v^- \un^-, \quad \jump{ \un \times
    \bm{q}} := \un^+\times  \bm{q}^+ +  \un^-\times \bm{q}^- , \\
    \jump{\un \otimes \bm{q}} &:= \un^+ \otimes \bm{q}^+ + \un^-
    \otimes \bm{q}^- , \quad \jump{\un \otimes \bm{\tau}} := \un^+
    \otimes \bm{\tau}^+ + \un^- \otimes \bm{\tau}^-,
  \end{aligned}
\end{displaymath}
where $\otimes$ denotes the tensor product between two vectors. For
any tensor $\widehat{\bm{\tau}}$, if $\mc{O}$ is an operator on
vector-valued functions, we expend $\mc{O}$ to act on
$\widehat{\bm{\tau}}$ columnwise.  For example, $\un \otimes
\widehat{\bm{\tau}}$ is defined by
\begin{displaymath}
  \un \otimes \widehat{\bm{\tau}} := \un \otimes \left(
  \widehat{\bm{\tau}}_1, \widehat{\bm{\tau}}_2, \ldots,
  \widehat{\bm{\tau}}_n    \right) = \left(  \un \otimes
  \widehat{\bm{\tau}}_1, \un \otimes \widehat{\bm{\tau}}_2, \ldots,
  \un \otimes \widehat{\bm{\tau}}_n \right)^T,
\end{displaymath}
and $\nabla \cdot \wh{\bm{\tau}}$ is defined by 
\begin{displaymath}
  \nabla \cdot \wh{\bm{\tau}} := \nabla \cdot \left(
  \widehat{\bm{\tau}}_1, \widehat{\bm{\tau}}_2, \ldots,
  \widehat{\bm{\tau}}_n    \right) = \left(  \nabla \cdot
  \widehat{\bm{\tau}}_1, \nabla \cdot \widehat{\bm{\tau}}_2, \ldots,
  \nabla \cdot \widehat{\bm{\tau}}_n \right)^T.
\end{displaymath}
For the boundary face $e$, the trace operators are defined by
\begin{displaymath}
  \begin{aligned}
    \aver{v} := v|_e, \quad \aver{\bm{q}} &:= \bm{q}|_e, \quad
    \aver{\bm{\tau}} := \bm{\tau}|_e, \\
    \jump{v} := v|_e \un, \quad \jump{\un \times \bm{q}} &:= \un
    \times \bm{q}|_e, \quad \jump{\un \otimes \bm{q}} := \un \otimes
    \bm{q}|_e, \quad \jump{\un \otimes \bm{\tau}} := \un \otimes
    \bm{\tau}|_e,
  \end{aligned}
\end{displaymath}
where $\un$ is the unit outward normal on $e$. 

Hereafter, we let $C$ and $C$ with subscripts denote the generic
positive constants that may be different from line to line but
independent of the mesh size $h$. For a bounded domain $D$, we will
use the standard notations and definitions for the Sobolev spaces
$L^2(D)$, $L^2(D)^d$, $L^2(D)^{d \times d}$, $H^r(D)$, $H^r(D)^d$,
$H^r(D)^{d \times d}$ with $r$ a positive integer, and we will also
use their associated inner products, semi-norms and norms. We define
the space of divergence-free functions by 
\begin{displaymath}
  \bmr{S}^r(D) := \left\{ \bm{v} \in H^r(D)^d \  | \  \nabla \cdot
  \bm{v} = 0 \text{ in } D\right\}, 
\end{displaymath}
and we further define the space of tensor-valued functions by 
\begin{displaymath}
  \bmr{I}^r(D) := \left\{ \bm{\tau} \in H^r(D)^{d \times d} \ | \
  \nabla \times \bm{\tau} = 0, \quad \tr{\bm{\tau}} = 0 \text{ in } D
  \right\},
\end{displaymath}
where $\tr{\cdot}$ denotes the standard trace operator. For the
partition $\MTh$, we will use the definitions for the broken Sobolev
spaces $L^2(\MTh)$, $L^2(\MTh)^d$, $L^2(\MTh)^{d \times d}$,
$H^r(\MTh)$, $H^r(\MTh)^d$, $H^r(\MTh)^{d \times d}$ and their
corresponding broken semi-norms and norms. Moreover, for any space $X
\in L^2(\Omega)$, we let $X / \mb{R}$ consist of the functions in $X$
that have zero mean value on $\Omega$, 
\begin{displaymath}
  X / \mb{R} := \left\{ v \in X \ | \ \int_\Omega v \d{x} = 0
  \right\}.
\end{displaymath}

The incompressible Stokes problem we are studying in this paper is
formulated as: {\it seeks the velocity fields $\bm{u}$ and the
  pressure $p$ such that
\begin{equation}
  \begin{aligned}
    - \nu \Delta \bm{u} + \nabla p = \bm{f}, & \quad \text{in }
    \Omega, \\
    \nabla \cdot \bm{u} = 0, & \quad \text{in } \Omega, \\
    \bm{u} = \bm{g}, & \quad \text{on } \partial \Omega, \\
  \end{aligned}
  \label{eq_Stokes}
\end{equation}
where $\bm{f}$ is a give source term and $\bm{g}$ is the boundary data
and $\nu$ is the reciprocal of the Reynolds number.}

As being declared above, we will propose a new least squares finite
element method, together with the reconstructed discontinuous
approximation, to solve the problem \eqref{eq_Stokes}. We introduce a
new inter-mediate variable $\bmr{U}$ to substitute
$\nabla \bm{u} = (\nabla u^1, \ldots, \nabla u^d)$, thus one can
reformulate the problem \eqref{eq_Stokes} into an equivalent
first-order system:
\begin{equation}
  \begin{aligned}
    - \nu \nabla \cdot \bmr{U} + \nabla p = \bm{f},  & \quad \text{in
    } \Omega, \\
    \nabla \bm{u} - \bmr{U} = \bm{0}, & \quad \text{in } \Omega, \\
    \nabla \cdot \bm{u} = 0, & \quad \text{in } \Omega, \\
    \bm{u} = \bm{g}, & \quad \text{on } \partial \Omega. \\
  \end{aligned}
  \label{eq_firstorderStokes}
\end{equation}
In Section \ref{sec_LSMStokes}, we will go back to the Stokes problem
for the details of the sequential least squares method for the
first-order system \eqref{eq_firstorderStokes} with the reconstructed
spaces introduced in Section \ref{sec_space}.

%%% Local Variables:
%%% mode: latex
%%% TeX-master: "article"
%%% End:

% vim:spell:tw=70:fo+=Mn:cc=70
\section{Reconstructed Approximation Space}
\label{sec_space}
In this section, we define three types of reconstruction operators
that will be used in numerically solving \eqref{eq_firstorderStokes}.
The first one is the reconstruction operator which has been used in
\cite{li2019eigenvalue, li2016discontinuous, li2019least} and the
other two operators are extensions from the first one. The
reconstruction procedure includes two parts.  The first part is to
construct the element patch and this part is the same for all
reconstruction operators. Now we present the details of the
construction to the element patch. We begin by assigning a collocation
point in each element. For each element $K$, we specify the barycenter
of $K$ as its corresponding collocation point $\bm{x}_K$. Then for
each element $K$ we aggregate an element patch $S(K)$ which is a set
of elements and consists of $K$ itself and some neighbour elements.
Specifically, the element patch $S(K)$ is constructed in a recursive
strategy. We first appoint a threshold value $\# S$ to control the
cardinality of $S(K)$.  We let $S_0(K) = \left\{ K \right\}$ and
construct a sequence of element sets $S_t(K) (t \geq 1)$ recursively: 
\begin{displaymath}
  S_{t+1}(K) = \bigcup_{\wt{K} \in S_{t}(K)}\  \bigcup_{\widehat{K}
    \in \mathcal{N}(\wt{K})} \widehat{K}, \quad t = 0, 1,
  2, \ldots
\end{displaymath}
where $\mathcal{N}(\wt{K})$ denotes the set of elements
face-neighbouring to $\wt{K}$. This recursion ends when the depth $t$
satisfies that the cardinality of $S_t(K)$ is greater than $\# S$. We
compute the distances between the collocation points of all elements
in $S_t(K)$ and the point $\bm{x}_K$. We choose the $\# S$ smallest
distances and gather the corresponding elements to form the element
patch $S(K)$. By this recursive strategy, for any element $K$ the
cardinality of $S(K)$ is always $\# S$. After constructing the element
patch, for element $K$ we denote by $\mc{I}(K)$ the set of the
collocation points to the elements in $S(K)$:
\begin{displaymath}
  \mc{I}(K) := \left\{ \bm{x}_{\wt{K}} \ | \ \forall \wt{K} \in S(K)
  \right\}.
\end{displaymath}
Then we will define three reconstruction operators to approximate the
functions in $H^r(\Omega)$, $\bmr{S}^r(\Omega)$ and $\bmr{I}^r(\Omega)
(r \geq 2)$, respectively.

\subsection{Reconstruction for Scalar-valued Functions}
\label{subsec_recon_scalar}
We denote by ${U}_h$ the piecewise constant space associated with
$\MTh$:
\begin{displaymath}
  {U}_h := \left\{ v_h \in L^2(\Omega) \ | \  v_h|_K \in \mb{P}_0(K),
  \quad \forall K \in \MTh\right\}.
\end{displaymath}
We will define a reconstruction operator $\mc{R}^m$ from the space
$U_h$ onto a subspace of the piecewise polynomial space
\cite{li2016discontinuous}. Given a piecewise constant function $g \in
U_h$, we will seek a polynomial of degree $m (m \geq 1)$ on every
element. For each element $K \in \MTh$, we solve the following
discrete least squares problem to determine a polynomial $\mc{R}_K^m
g$ of degree $m$: 
\begin{equation}
  \begin{aligned}
    \mc{R}_K^m g = & \mathop{\arg \min}_{q \in \mb{P}_m(S(K))}
    \sum_{\bm{x} \in \mc{I}(K)} |q(\bm{x}) - g(\bm{x})|^2 \\
    & \text{s.t. } q(\bm{x}_K) = g(\bm{x}_K).
  \end{aligned}
  \label{eq_scalarls}
\end{equation}
We follow \cite[Assumption B]{li2016discontinuous} to make the
assumption to ensure the problem \eqref{eq_scalarls} has a unique
solution. 
\begin{assumption}
  For any element $K \in \MTh$ and any polynomial $q(\bm{x}) \in
  \mb{P}_m(S(K))$, we have that 
  \begin{equation}
    p|_{\mc{I}(K)} = 0 \quad \text{implies} \quad p|_{S(K)} = 0.
    \label{eq_uniqueassump}
  \end{equation}
  \label{as_uniqueassump}
\end{assumption}
The Assumption \ref{as_uniqueassump} is actually a geometrical
assumption which rules out the case that the points in $\mc{I}(K)$ are
lying on an algebraic curve of degree $m$ and requires the value of
$\# S$ shall be greater than the dimension of the polynomial space
$\mb{P}_m(\cdot)$.

Then the reconstruction operator $\mc{R}^m$ is defined in a piecewise
manner: 
\begin{displaymath}
  (\mc{R}^m g)|_K = (\mc{R}_K^m g)|_K, \quad \text{on any element } K
  \in \MTh.
\end{displaymath}
We note that the polynomial $\mc{R}_K^m g$ is linearly dependent on
$g$.  Hence, the operator $\mc{R}^m$ maps $U_h$ onto a subspace of the
piecewise polynomial space, which is denoted by $U_h^m = \mc{R}^m
U_h$. For any smooth function $q \in C^0(\Omega)$, we define $\wt{q}
\in U_h$ such that 
\begin{displaymath}
  \wt{q}(\bm{x}_K) = q(\bm{x}_K), \quad \forall K \in \MTh,
\end{displaymath}
and we extend the operator $\mc{R}^m$ to act on the smooth function
by defining $\mc{R}^m q := \mc{R}^m \wt{q}$. 

In addition, we outline a group of basis functions of $U_h^m$. For any
element $K$, we define the characteristic function $w_K(\bm{x}) \in
U_h$ as
\begin{displaymath}
  w_K(\bm{x}) = \begin{cases}
    1, & \bm{x} \in K, \\
    0, & \text{otherwise}, \\
  \end{cases}
\end{displaymath}
and we denote $\lambda_K$ as $\lambda_K = \mc{R}^m w_K$.

\begin{lemma}
  The functions $\left\{ \lambda_K \right\}(\forall K \in \MTh)$ are
  linearly independent. 
  \label{le_basisfunction}
\end{lemma}

\begin{proof}
  For any $\lambda_K(\bm{x})$, the constraint in \eqref{eq_scalarls}
  implies that 
  \begin{equation}
    \lambda_K(\bm{x}_{\wt{K}}) = \begin{cases}
      1, & \wt{K} = K, \\
      0, & \wt{K} \neq K. \\
    \end{cases}
    \label{eq_lambdaKwK}
  \end{equation}
  We assume that there exist coefficients $\left\{ a_K \right\}(
  \forall K \in \MTh)$ such that 
  \begin{equation}
    \sum_{K \in \MTh} a_K \lambda_K(\bm{x}) = 0, \quad \forall \bm{x}
    \in \Omega.
    \label{eq_aKlambdaK}
  \end{equation}
  For any element $K$, let $\bm{x} = \bm{x}_K$ in \eqref{eq_aKlambdaK}
  and by \eqref{eq_lambdaKwK} one see that $a_K = 0$, which shows
  that $\left\{ \lambda_K \right\}$ are linearly independent. This
  completes the proof.
\end{proof}
Clearly, we have that $\text{dim}( \{ \lambda_K \}) =
\text{dim}(U_h)$.  The Lemma \ref{le_basisfunction} in fact gives that
the reconstructed space satisfies $\text{dim}(U_h^m) =
\text{dim}(U_h)$ and $U_h^m$ is spanned by $\left\{ \lambda_K
\right\}$. Moreover, for the function $g \in U_h$ or $g \in
C^0(\Omega)$, one may write $\mc{R}^m g$ explicitly: 
\begin{equation}
  \mc{R}^m g = \sum_{K \in \MTh} g(\bm{x}_K) \lambda_K(\bm{x}).
  \label{eq_scalarRmg}
\end{equation}
Then we give the approximation property of the space $U_h^m$. For any
element $K$, we define a constant $\Lambda(m, K)$ as
\begin{displaymath}
  \Lambda(m, K) := \max_{q \in \mb{P}_m(S(K))} \frac{\max_{\bm{x} \in
  S(K)} |q(\bm{x})| }{ \max_{ \bm{x} \in \mc{I}(K)} |q(\bm{x})|}.
\end{displaymath}
We let $\Lambda_m := \max_{K \in \MTh} (1 + \Lambda(m, K) (\#
S)^{1/2})$ and in \cite{li2012efficient, li2016discontinuous} the
authors proved that under some mild conditions on element patches, the
constant $\Lambda_m$ can be bounded uniformly with respect to the
partition.  These conditions reply on the size of element patches and
the authors also proved that if the number $\# S$ is greater than a
certain number, the conditions on element patches will be fulfilled,
see \cite[Lemma 6]{li2016discontinuous} and \cite[Lemma
3.4]{li2012efficient}.  This certain number is usually too large to be
impractical and is not recommended in the implementation. The
numerical results demonstrate that our method still has a very good
performance even we take the value $\# S$ to be far less than that
certain number. In Section \ref{sec_numerical_results} we list the
values of $\# S$ for different $m$ that are adopted in the numerical
tests.  

Then we state the following estimate. 
\begin{lemma}
  For any element $K$ and any function $g \in C^0(\Omega)$, there
  holds
  \begin{displaymath}
    \|g - \mc{R}^m g \|_{L^{\infty}(K)} \leq \Lambda_m \inf_{q \in
    \mb{P}_m(S(K))} \| g - q\|_{L^\infty(S(K))}.
  \end{displaymath}
  \label{le_scalarinftyestimate}
\end{lemma}
\begin{proof}
  We refer to \cite[Lemma 3]{li2016discontinuous} for the proof.
\end{proof}
From Lemma \ref{le_scalarinftyestimate}, it is easy to derive the
following approximation properties. 
\begin{theorem}
  For any element $K$, there exist constants $C$ such that 
  \begin{equation}
    \begin{aligned}
      \| g - \mc{R}^m g \|_{H^q(K)} & \leq C \Lambda_m h_K^{m
      + 1 - q} \| g \|_{H^{m+1}( S(K))}, \quad q = 0, 1, \\
      \| \nabla^q (g - \mc{R}^m g) \|_{L^2(\partial K)} & \leq C
      \Lambda_m h_K^{m + 1 - q - 1/2} \| g \|_{H^{m+1}(S(K))},
      \quad q = 0, 1, \\
    \end{aligned}
    \label{eq_scalarapproximation}
  \end{equation}
  for any $g \in H^{m+1}(\Omega)$.
  \label{th_scalarapproximation}
\end{theorem}
\begin{proof}
  We refer to \cite[Lemma 4]{li2016discontinuous} for the proof.
\end{proof}

\subsection{Reconstruction for Vector-valued Functions}
\label{subsec_recon_vector}
Here we consider to extend the reconstruction process for
vector-valued functions.  Precisely, we will introduce a
reconstruction operator for functions in the space
$\bmr{S}^{m+1}(\Omega)$. We also start from the piecewise constant
space $(U_h)^d$. Given a function $\bm{g} \in (U_h)^d$ and for any
element $K \in \MTh$, we solve a polynomial $\wt{\mc{R}}_K^m \bm{g}$
of degree $m$ on $S(K$) by the following discrete least squares
problem: 
\begin{equation}
  \begin{aligned}
    \wt{\mc{R}}_K^m \bm{g} = \mathop{\arg \min}_{\bm{q} \in \mb{P}_m
    (S(K))^d} & \sum_{\bm{x} \in \mc{I}(K)} \| \bm{q}(\bm{x}) -
    \bm{g}(\bm{x})\|_{l^d}^2, \\ 
    %\text{s.t. } &\  \bm{q}(\bm{x}_K) = \bm{g}(\bm{x}_K), \\
    %& \ \ \ \nabla \cdot \bm{q} = 0 . \\
    \text{s.t. } & \begin{cases}
      \bm{q}(\bm{x}_K) = \bm{g}(\bm{x}_K), \\
      \nabla \cdot \bm{q} = 0 , \\
    \end{cases}
  \end{aligned}
  \label{eq_vectorlsproblem}
\end{equation}
where
\begin{displaymath}
  \|\bm{v} \|_{l^d}^2 := v_1^2 + \ldots + v_d^2, \quad \forall \bm{v}
  = (v_1, \ldots, v_d)^T \in \mb{R}^d.
\end{displaymath}
Based on Assumption \ref{as_uniqueassump}, it is trivial to check the
uniqueness and the existence of the solution to the problem
\eqref{eq_vectorlsproblem}. Then the reconstruction operator
$\wt{\mc{R}}^m$ is piecewise defined as 
\begin{displaymath}
  (\wt{\mc{R}}^m \bm{g})|_K = (\wt{\mc{R}}_K^m \bm{g})|_K, \quad
  \text{on any element } K \in \MTh.
\end{displaymath}
It should be noted that $\wt{\mc{R}}^m_K \bm{g}$ still has a linear
dependence on $\bm{g}$.  Therefore, we can know that the operator
$\wt{\mc{R}}^m$ maps the space $(U_h)^d$ into the piecewise
divergence-free polynomial space of degree $m$, and we denote by
$\bmr{S}_h^m = \wt{\mc{R}}^m (U_h)^d$. We also extend the operator
$\wt{\mc{R}}^m$ to act on the smooth function as the operator
$\mc{R}^m$. For the function $\bm{g}(\bm{x}) \in C^0(\Omega)^d$, we
define a piecewise constant function $\wt{\bm{g}}(\bm{x}) \in (U_h)^d$
such that 
\begin{displaymath}
  \wt{\bm{g}}(\bm{x}_K) = \bm{g}(\bm{x}_K), \quad \forall K \in \MTh,
\end{displaymath}
and we define $\wt{\mc{R}}^m \bm{g} := \wt{\mc{R}}^m \wt{\bm{g}}$.  We
also present a group of basis functions to the space $\bmr{S}_h^m$.
For any element $K$, we define an indicator function
$\wt{\bm{w}}_K^i(\bm{x}) \in (U_h)^d$, which reads 
\begin{displaymath}
  \wt{\bm{w}}_K^i(\bm{x}) = \begin{cases}
    \bm{e}_i, & \bm{x} \in K, \\
    0, & \text{otherwise}, \\
  \end{cases}
\end{displaymath}
where $\bm{e}_i$ is a $d \times 1$ unit vector whose $i$-th entry is
$1$.  Then we define $\wt{\bm{\lambda}}_K^i$ as $\wt{\bm{\lambda}}_K^i
= \wt{\mc{R}}^m_K \wt{\bm{w}}_K^i$ and we have the following lemma. 

\begin{lemma}
  The functions $\left\{ \bm{\lambda}_K^i \right\}(\forall K \in \MTh,
  1 \leq i \leq d)$ are linearly independent.
  \label{le_vectorbasisfunction}
\end{lemma}
\begin{proof}
  The proof results from the constraint in \eqref{eq_vectorlsproblem}
  and is similar to the proof of Lemma \ref{le_basisfunction}.
\end{proof}
Analogously, we conclude that the space $\bmr{S}_h^m$ is spanned by $\{
\wt{\bm{\lambda}}_K^i \}$ and for the function $\bm{g} = (g^1, \ldots,
g^d) \in (U_h)^d$ or $\bm{g} = (g^1, \ldots, g^d) \in C^0(\Omega)^d$,
we can write $\wt{\mc{R}}^m \bm{g}$ as 
\begin{displaymath}
  \wt{\mc{R}}^m \bm{g} = \sum_{K \in \MTh} \sum_{i = 1}^d
  g^i(\bm{x}_K) \wt{\bm{\lambda}}_K^i(\bm{x}).
\end{displaymath}
Further, we give the approximation property of the operator
$\wt{\mc{R}}^m$.

\begin{lemma}
  For any element $K$ and any function $\bm{g} \in C^0(\Omega)^d$,
  there holds
  \begin{equation}
    \|\bm{g} - \wt{\mc{R}}^m \bm{g} \|_{L^\infty(K)} \leq \sqrt{d}
    \Lambda_m \inf_{\bm{q} \in \mb{P}_m(S(K))^d \cap \bmr{S}^0(S(K))}
    \| \bm{g} - \bm{q} \|_{L^\infty(S(K))}.
    \label{eq_vectorinfestimate}
  \end{equation}
  \label{le_vectorinfestimate}
\end{lemma}

\begin{proof}
  For any divergence-free polynomial $\bm{q} \in \mb{P}_m(S(K))^d \cap
  \bmr{S}^0(S(K))$, we clearly have that $\wt{\mc{R}}^m_K \bm{q} =
  \bm{q}$ from the definition of the least squares problem
  \eqref{eq_vectorlsproblem}. We deduce that
  \begin{displaymath}
    \begin{aligned}
      \|\bm{g} - \wt{\mc{R}}^m \bm{g} \|_{L^\infty(K)} & \leq
      \|\bm{g} - \bm{q} \|_{L^\infty(K)} +  \| \wt{\mc{R}}^m ( \bm{q}
      - \bm{g}) \|_{L^\infty(K)} \\
      & \leq   \|\bm{g} - \bm{q} \|_{L^\infty(K)} + \Lambda(m, K)
      \max_{\bm{x} \in \mc{I}(K)} |\wt{\mc{R}}^m ( \bm{q} - \bm{g}) |
      \\
      & \leq  \|\bm{g} - \bm{q} \|_{L^\infty(K)} + \Lambda(m, K)
      \sqrt{d \# S } \max_{\bm{x} \in \mc{I}(K)} | \bm{q} - \bm{g} |
      \\
      & \leq \sqrt{d} \left( 1  + \Lambda(m, K)\sqrt{\# S} \right)
      \|\bm{g} - \bm{q} \|_{L^\infty(S(K))}, \\
    \end{aligned}
  \end{displaymath}
  which gives us the estimate \eqref{eq_vectorinfestimate} and
  completes the proof.
\end{proof}

\begin{theorem}
  For any element $K$, there exist constants $C$ such that 
  \begin{equation}
    \begin{aligned}
      \| \bm{g} - \wt{\mc{R}}^m \bm{g} \|_{H^q(K)} & \leq C \Lambda_m
      h_K^{m + 1 - q} \| \bm{g} \|_{H^{m+1}( S(K))}, \quad q = 0, 1,
      \\
      \| \nabla^q (\bm{g} - \wt{\mc{R}}^m \bm{g}) \|_{L^2(\partial K)}
      & \leq C \Lambda_m h_K^{m + 1 - q - 1/2} \| \bm{g}
      \|_{H^{m+1}(S(K))}, \quad q = 0, 1, \\
    \end{aligned}
    \label{eq_vectorapproximation}
  \end{equation}
  for any $\bm{g} \in \bmr{S}^{m+1}(\Omega)$.
  \label{th_vectorapproximation}
\end{theorem}
\begin{proof}
  By \cite[Theorem 4.1]{baker1990piecewise} and \cite[Assumption
  A]{li2016discontinuous}, there exists an approximation polynomial
  $\wt{\bm{q}} \in \mb{P}_m(S(K))^d \cap \bmr{S}^0(S(K))$ such that 
  \begin{displaymath}
    \| \bm{g} - \wt{\bm{q}} \|_{L^\infty (S(K))} \leq C h_K^{m + 1 -
    d/2} \| \bm{g} \|_{H^{m+1}(S(K))}. 
  \end{displaymath}
  By Lemma \ref{le_vectorinfestimate}, we obtain that
  \begin{displaymath}
    \begin{aligned}
      \|\bm{g} - \wt{\mc{R}}^m \bm{g} \|_{L^2(K)} & \leq
      Ch_K^{d/2} \| \bm{g} - \wt{\mc{R}}^m \bm{g} \|_{L^\infty(K)}
      \leq Ch_K^{d/2} \Lambda_m \| \bm{g} - \wt{\bm{q}} \|_{L^\infty(
      S(K))} \\ 
      &\leq C \Lambda_m h_K^{m + 1} \| \bm{g} \|_{H^{m+1}(
      S(K))}, \\
    \end{aligned}
  \end{displaymath}
  and together with the inverse inequality M2, we have that 
  \begin{displaymath}
    \begin{aligned}
      \| \bm{g} - \wt{\mc{R}}^m \bm{g} \|_{H^1(K)} & \leq \|\bm{g} -
      \wt{\bm{q}} \|_{H^1(K)} + \|\wt{\bm{q}} - \wt{\mc{R}}^m \bm{g}
      \|_{H^1(K)}  \leq  \|\bm{g} - \wt{\bm{q}} \|_{H^1(K)} +
      Ch_K^{-1} \|\wt{\bm{q}} - \wt{\mc{R}}^m \bm{g} \|_{L^2(K)} \\
      & \leq C \Lambda_m h_K^{m + 1} \| \bm{g} \|_{H^{m+1}( S(K))}. 
    \end{aligned}
  \end{displaymath}
  Applying the trace inequality M2 gives the trace estimate in
  \eqref{eq_vectorapproximation}, which completes the proof.
\end{proof}

\subsection{Reconstruction for Tensor-valued Functions}
\label{subsec_recon_tensor}
In this subsection, we consider the reconstruction for the
tensor-valued functions in the space $\bmr{I}^{m+1}(\Omega)$. Again,
we start from a piecewise constant space. Since the functions in
$\bmr{I}^{m+1}(\Omega)$ have zero trace, we define the space
$\bmr{U}_h$ consisting of piecewise constant functions with zero
trace, which reads
\begin{displaymath}
  \bmr{U}_h := \left\{ \bm{v}_h \in (U_h)^{d \times d} \ | \
  \tr{\bm{v}_h} = 0 \right\}.
\end{displaymath}
For the function $\bm{g} \in \bmr{U}_h$, we will seek a polynomial
$\wh{\mc{R}}_K^m \bm{g}$ of degree $m$ on $S(K)$ by solving the
following problem:
\begin{equation}
  \begin{aligned}
    \wh{\mc{R}}_K^m \bm{g} = \mathop{\arg \min}_{\bm{q} \in
    \mb{P}_m(S(K))^{d \times d}} & \sum_{\bm{x} \in \mc{I}(K)} \|
    \bm{q}(\bm{x}) - \bm{g}(\bm{x})\|_{l^{d \times d}}^2, \\
    %&\ \  \bm{q}(\bm{x}_K) = \bm{g}(\bm{x}_K), \\
    %\text{s.t. } & \ \  \nabla \times \bm{q} = 0,  \\
    %& \ \  \tr{\bm{q}} = 0. \\
    \text{s.t. } & \begin{cases}
    \bm{q}(\bm{x}_K) = \bm{g}(\bm{x}_K), \\
    \nabla \times \bm{q} = 0,  \\
     \tr{\bm{q}} = 0, \\
    \end{cases}
  \end{aligned}
  \label{eq_tensorlsproblem}
\end{equation}
where 
\begin{displaymath}
  \| \bm{\tau} \|_{l^d \times l^d}^2 := \| \bm{\tau}_1 \|_{l^d}^2 +
  \ldots + \| \bm{\tau}_d \|_{l^d}^2, \quad \forall \bm{\tau} =
  (\bm{\tau}_1, \ldots, \bm{\tau}_d) \in \mb{R}^{d \times d}.
\end{displaymath}
Similarly, we have that the problem \eqref{eq_tensorlsproblem} has a
unique solution by Assumption \ref{as_uniqueassump}. The global
reconstruction operator $\wh{\mc{R}}^m$ is piecewise defined by 
\begin{displaymath}
  (\wh{\mc{R}}^m \bm{g})|_K = (\wh{\mc{R}}_K^m \bm{g})|_K, \quad
  \text{on any element } K \in \MTh.
\end{displaymath}
The solution $\wh{\mc{R}}_K^m \bm{g}$ is linearly dependent on
$\bm{g}$ and we denote by $\bmr{I}_h^m$ the image of the operator
$\wh{\mc{R}}^m$.  Then we still extend the operator $\wh{\mc{R}}^m$ to
act on the smooth functions. For the function $\bm{g}(\bm{x}) \in
C^0(\Omega)^{d \times d}$ with zero trace, we let $\wh{\bm{g}}(\bm{x})
\in \bmr{U}_h$ such that
\begin{displaymath}
  \wh{\bm{g}}(\bm{x}_K) = \bm{g}(\bm{x}_K), \quad \forall K \in \MTh,
\end{displaymath}
and again we define $\wh{\mc{R}}^m \bm{g} := \wh{\mc{R}}^m
\wh{\bm{g}}$.  Here we give a group of basis functions of the space
$\bmr{I}_h^m$. We will define a group of characteristic functions
$\wh{\bm{w}}_K^{i, j}(\bm{x}) \in \bmr{U}_h$ but we shall consider the
zero trace condition of functions in $\bmr{U}_h$. Actually this
condition implies that there are only $d^2 - 1$ characteristic
functions on every element.  For any element $K$, we define
$\wh{\bm{w}}_K^{i, j}(\bm{x})$ as 
\begin{displaymath}
  \wh{\bm{w}}_K^{i, j}(\bm{x}) = \begin{cases}
    \bm{e}_{i, j}, & \bm{x} \in K, \\
    0, & \text{otherwise}, \\
  \end{cases} \quad 1 \leq i \neq j \leq d, 
\end{displaymath}
where $e_{i, j}$ is the $d \times d$ matrix whose $(i, j)$ entry is
$1$ and the other entries are $0$. For $1 \leq i < d$,  we 
define $\wh{\bm{w}}_K^{i, i}(\bm{x})$ as 
\begin{displaymath}
  \wh{\bm{w}}_K^{i, i}(\bm{x}) = \begin{cases}
    \wh{\bm{e}}_{i, i}, & \bm{x} \in K, \\
    0, & \text{otherwise}, \\
  \end{cases}
\end{displaymath}
where $\wh{\bm{e}}_{i, i}$ is the $d \times d$ matrix whose $(i, i)$
entry is $1$, $(d, d)$ entry is $-1$ and other entries are $0$. 

We define $\wh{\bm{\lambda}}_K^{i, j} = \wh{\mc{R}}^m
\wh{\bm{w}}_K^{i,j}$ and we state that the functions $\{
\wh{\bm{\lambda}}_K^{i, j} \}$ are a group of basis functions to the
space $\bmr{I}_h^m$.

\begin{lemma}
  The functions $\{ \wh{\bm{\lambda}}_K^{i, j} \}(\forall K \in \MTh,
  1 \leq i, j \leq d, i + j < 2d)$ are linearly independent.
  \label{le_tensorbasisfunctions}
\end{lemma}
\begin{proof}
  The proof of Lemma \ref{le_tensorbasisfunctions} is analogous to the
  proof of Lemma \ref{le_vectorbasisfunction} and Lemma
  \ref{le_basisfunction}.
\end{proof}
Clearly, we can know that $\bmr{I}_h^m = \text{span}\{
\wh{\bm{\lambda}}_K^{i, j} \}$ and for any function $\bm{g} = (g^{i,
j}(\bm{x}))_{d \times d} \in \bmr{U}_h $ or $\bm{g}  = (g^{i,
j}(\bm{x}))_{d \times d}\in C^0(\Omega)^{d \times d}$ with
$\tr{\bm{g}} = 0$, $\wh{\mc{R}}^m \bm{g}$ can be expressed as 
\begin{equation}
  \wh{\mc{R}}^m \bm{g} = \sum_{K \in \MTh} \sum_{1 \leq i, j \leq d
  \text{ and } i + j < 2d} {g}^{i, j}(\bm{x}_K)
  \wh{\bm{\lambda}}_K^{i, j}(\bm{x}).
  \label{eq_tensorRmg}
\end{equation}
Moreover, we give the approximation property of the space
$\bmr{I}_h^m$.

\begin{lemma}
  For any element $K$ and any function $\bm{g} \in C^0(\Omega)^{d
  \times d}$ with $\tr{\bm{g}} = 0$, there holds
  \begin{equation}
    \| \bm{g} - \wh{\mc{R}}^m \bm{g} \|_{L^\infty (K)} \leq d
    \Lambda_m \inf_{\bm{q} \in \mb{P}_m(S(K))^{d \times d} \cap
    \bmr{I}^0(S(K)) } \| \bm{g} - \bm{q} \|_{L^\infty (S(K))}.
    \label{eq_tensorinfestimate}
  \end{equation}
  \label{le_tensorinfestimate}
\end{lemma}

\begin{proof}
  According to the problem \eqref{eq_scalarls}, we have that
  $\wh{\mc{R}}^m \bm{q} = \bm{q}$ for any $\bm{q} \in
  \mb{P}_m(S(K))^{d \times d} \cap \bmr{I}^0(S(K))$. We obtain that 
  \begin{displaymath}
    \begin{aligned}
      \|\bm{g} - \wh{\mc{R}}^m \bm{g} \|_{L^\infty(K)} & \leq
      \|\bm{g} - \bm{q} \|_{L^\infty(K)} +  \| \wh{\mc{R}}^m ( \bm{q}
      - \bm{g}) \|_{L^\infty(K)} \\
      & \leq   \|\bm{g} - \bm{q} \|_{L^\infty(K)} + \Lambda(m, K)
      \max_{\bm{x} \in \mc{I}(K)} |\wt{\mc{R}}^m ( \bm{q} - \bm{g}) |
      \\
      & \leq  \|\bm{g} - \bm{q} \|_{L^\infty(K)} + d \Lambda(m, K)
      \sqrt{ \# S } \max_{\bm{x} \in \mc{I}(K)} | \bm{q} - \bm{g} | \\
      & \leq d \left( 1  + \Lambda(m, K)\sqrt{\# S} \right)
      \|\bm{g} - \bm{q} \|_{L^\infty(S(K))}, \\
    \end{aligned}
  \end{displaymath}
  which completes the proof.
\end{proof}

\begin{theorem}
  For any element $K$, there exist constants $C$ such that 
  \begin{equation}
    \begin{aligned}
      \| \bm{g} - \wt{\mc{R}}^m \bm{g} \|_{H^q(K)} & \leq C \Lambda_m
      h_K^{m + 1 - q} \| \bm{g} \|_{H^{m+1}( S(K))}, \quad q = 0, 1,
      \\
      \| \nabla ^q (\bm{g} - \wt{\mc{R}}^m \bm{g}) \|_{L^2(\partial
      K)} & \leq C \Lambda_m h_K^{m + 1 - q - 1/2} \| \bm{g}
      \|_{H^{m+1}(S(K))}, \quad q = 0, 1, \\
    \end{aligned}
    \label{eq_tensorapproximation}
  \end{equation}
  for any $\bm{g} \in \bmr{I}^{m+1}(\Omega)$.
  \label{th_tensorapproximation}
\end{theorem}

\begin{proof}
  Since $\nabla \times \bm{g} = 0$ and $\tr{\bm{g}} = 0$, there exists
  a function $\wh{\bm{g}} \in \bmr{S}^{m+2}(S(K))$ such that
  $\bm{g} = \nabla \wh{\bm{g}}$ \cite[Lemma 2.1]{girault1986finite}. 
  By \cite[Theorem 4.1]{baker1990piecewise}, there exists a polynomial
  $\wh{\bm{q}} \in \mb{P}_{m+1}(S(K))^d \cap \bmr{S}^0(S(K))$ such
  that 
  \begin{displaymath}
    \begin{aligned}
      \| \bm{g} - \nabla \wh{\bm{q}} \|_{L^\infty (S(K))} & = \|
      \nabla ( \wh{\bm{g}} - \wh{\bm{q}}) \|_{L^\infty(S(K))} \leq C
      h_K^{m + 1 - d/2} \| \wh{\bm{g}} \|_{H^{m+2}(S(K))} \\
      & =  C h_K^{m + 1 - d/2} \| \bm{g} \|_{H^{m+1}(S(K))}. \\
    \end{aligned}
  \end{displaymath}
  Clearly, $\nabla \wh{\bm{q}} \in \mb{P}_m(S(K))^{d \times d} \cap
  \bmr{I}^0(S(K))$. By Lemma \ref{le_tensorinfestimate} and the
  approximation estimate of $\wh{\bm{q}}$,  we deduce that 
  \begin{displaymath}
    \begin{aligned}
      \|\bm{g} - \wh{\mc{R}}^m \bm{g} \|_{L^2(K)} & \leq
      Ch_K^{d/2} \| \bm{g} - \wh{\mc{R}}^m \bm{g} \|_{L^\infty(K)}
      \leq Ch_K^{d/2} \Lambda_m \| \bm{g} - \nabla
      \wh{\bm{q}}\|_{L^\infty( S(K))} \\ 
      &\leq C \Lambda_m h_K^{m + 1} \| \bm{g} \|_{H^{m+1}(
      S(K))}. \\
    \end{aligned}
  \end{displaymath}
  Together with the inverse inverse M3, we have 
  \begin{displaymath}
    \begin{aligned}
      \| \bm{g} - \wh{\mc{R}}^m \bm{g} \|_{H^1(K)} & \leq \|\bm{g} -
      \nabla \wh{\bm{q}} \|_{H^1(K)} + \|\nabla \wh{\bm{q}} -
      \wh{\mc{R}}^m \bm{g} \|_{H^1(K)} \\
      &\leq  \|\bm{g} - \nabla \wh{\bm{q}} \|_{H^1(K)} + Ch_K^{-1}
      \|\nabla \wh{\bm{q}} - \wh{\mc{R}}^m \bm{g} \|_{L^2(K)} \\
      & \leq  \|\bm{g} - \nabla \wh{\bm{q}} \|_{H^1(K)} + Ch_K^{-1}
      \left( \| \bm{g} - \nabla \wh{\bm{q}} \|_{L^2(K)} + \| \bm{g} -
      \wh{\mc{R}}^m \bm{g} \|_{L^2(K)} \right) \\
      & \leq C \Lambda_m h_K^{m} \| \bm{g} \|_{H^{m+1}( S(K))}. 
    \end{aligned}
  \end{displaymath}
  The trace estimate of \eqref{eq_tensorapproximation} follows from
  the trace inequality M2, and this completes the proof.
\end{proof}

We have established three types of reconstruction operators and their
corresponding approximation spaces and the approximation results. In
Appendix \ref{sec_appendix}, we present some details of the computer
implementation to reconstructed spaces. 

%%% Local Variables:
%%% mode: latex
%%% TeX-master: "article"
%%% End:

% vim:spell:tw=70:fo+=Mn:cc=70
\section{Sequential Least Squares Method for Stokes Problem}
\label{sec_LSMStokes}
In this section, we propose a sequential least squares finite element
method to solve the Stokes problem based on the first-order system
\eqref{eq_firstorderStokes}. We are motivated by the idea in
\cite{Cai1997first, li2019sequential} to decouple the system
\eqref{eq_firstorderStokes} into two steps. The first first-order
system is defined to seek the numerical approximations to the gradient
$\bmr{U}$ and the pressure $p$, which reads
\begin{equation}
  \begin{aligned}
    - \nu \nabla \cdot \bmr{U} + \nabla p &= \bm{f},   \quad \text{in
    } \Omega, \\
    \un \times \bmr{U} &= \un \times \nabla \bm{g},  \quad \text{on }
    \partial \Omega. \\
  \end{aligned}
  \label{eq_firstorderUp}
\end{equation}
The first equation in \eqref{eq_firstorderUp} is the same as the first
equation in \eqref{eq_firstorderStokes} and the boundary condition in
\eqref{eq_firstorderStokes} provides the tangential trace $\un \times
\bmr{U}$ on the boundary $\partial \Omega$. Then, we define a least
squares functional $J_h^{\bmr{p}}(\cdot, \cdot)$ for numerically
solving the system \eqref{eq_firstorderUp}, which reads 
\begin{equation}
  \begin{aligned}
    J_h^{\bmr{p}}(\bmr{V}_h, q_h) := \sum_{K \in \MTh}  \| -\nu &
    \nabla \cdot \bmr{V}_h  + \nabla q_h - \bm{f} \|_{L^2(K)}^2 +
    \sum_{e \in \MEh^i} \left( \frac{\eta}{h_e} \| \jump{q_h}
    \|_{L^2(e)}^2 + \frac{\eta}{h_e} \| \jump{\un \otimes \bmr{V}_h}
    \|_{L^2(e)}^2 \right) \\
    & + \sum_{e \in \MEh^b} \frac{\eta}{h_e} \| \un \times \bmr{V}_h -
    \un \times \nabla \bm{g} \|_{L^2(e)}^2, \quad \forall (\bmr{V}_h,
    q_h) \in H^1(\MTh)^{d \times d} \times H^1(\MTh), \\
  \end{aligned}
  \label{eq_functionalJ1}
\end{equation}
where $\eta$ is a positive parameter and will be specified later on.
In \eqref{eq_functionalJ1}, the trace terms defined on $\MEh^i$ are
used to weakly impose the continuity condition since
the polynomials in reconstructed spaces may be discontinuous across
the interior faces.  We minimize the functional
\eqref{eq_functionalJ1} over the spaces $\bmr{I}_h^m \times
\wt{U}_h^m$ to give approximations to $\bmr{U}$ and $p$ and here the
space $\wt{U}_h^m$ is defined by ${U}_h^m / \mb{R}$. 

Specifically, the minimization problem is defined as to find
$\bmr{U}_h \in \bmr{I}_h^m$ and $p_h \in \wt{U}_h^m$ such that 
\begin{equation}
  (\bmr{U}_h, p_h) = \mathop{\arg \min}_{ (\bmr{V}_h, q_h) \in
  \bmr{I}_h^m \times \wt{U}_h^m} J_h^{\bmr{p}}( \bmr{V}_h, q_h).
  \label{eq_infJp}
\end{equation}
We write the Euler-Lagrange equation to solve the problem
\eqref{eq_infJp} and the corresponding discrete variational problem
reads: find $(\bmr{U}_h, p_h) \in \bmr{I}_h^m \times \wt{U}_h^m$ such
that 
\begin{equation}
  a_h^{\bmr{p}}(\bmr{U}_h, p_h; \bmr{V}_h, q_h) =
  l_h^{\bmr{p}}(\bmr{V}_h, q_h), \quad \forall (\bmr{V}_h, q_h) \in
  \bmr{I}_h^m \times \wt{U}_h^m,
  \label{eq_bilinearp}
\end{equation}
where the bilinear form $a_h^{\bmr{p}}(\cdot; \cdot)$ is 
\begin{equation}
  \begin{aligned}
    a_h^{\bmr{p}}(\bmr{U}_h, p_h; \bmr{V}_h, q_h) = & \sum_{K \in
    \MTh} \int_K (-\nu \nabla \cdot \bmr{U}_h + \nabla p_h) (-\nu
    \nabla \cdot \bmr{V}_h + \nabla q_h) \d{x} \\
    + &\sum_{e \in \MEh^i} \int_e \frac{\eta}{h_e} \jump{p_h} \cdot
    \jump{q_h} \d{s} + \sum_{e \in \MEh^i} \int_e
    \frac{\eta}{h_e} \jump{\un \otimes \bmr{U}_h} : \jump{\un \otimes
    \bmr{V}_h} \d{s} \\ 
    + & \sum_{e \in \MEh^b} \int_e \frac{\eta}{h_e} (\un \times
    \bmr{U}_h) \cdot (\un \times \bmr{V}_h) \d{s},
  \end{aligned}
  \label{eq_apbilinear}
\end{equation}
and the linear form $l_h^{\bmr{p}}(\cdot)$ is 
\begin{displaymath}
  l_h^{\bmr{p}}(\bmr{V}_h, q_h) = \sum_{K \in \MTh} \int_K \bm{f}
  \left( - \nu \nabla \cdot \bmr{V}_h + \nabla q_h \right) \d{x} +
  \sum_{e \in \MEh^b} \int_e \frac{\eta}{h_e} (\un \times \bmr{V}_h)
  \cdot (\un \times \nabla \bm{g}) \d{s}.
\end{displaymath}

Then we will focus on the error estimate to the problem
\eqref{eq_bilinearp}. To do this, we will require some projection
results which play a key role in the analysis.  We define $V_h^m$ and
$\wt{V}_h^m$ as piecewise polynomial spaces with respect to the
partition $\MTh$ and the sub-decomposition $\wt{\mc{T}}_h$, 
\begin{displaymath}
  \begin{aligned}
    V_h^m &:= \{ v_h \in L^2(\Omega) \ | \  v_h|_K \in
    \mb{P}_m(K), \quad \forall K \in \MTh \}, \\
    \wt{V}_h^m &:= \{ v_h \in L^2(\Omega) \ | \  v_h|_{\wt{K}} \in
    \mb{P}_m(\wt{K}), \quad \forall \wt{K} \in \wt{\mc{T}}_h \}, \\
  \end{aligned}
\end{displaymath}
and clearly we have that $V_h^m \subset \wt{V}_h^m$. Then we state 
following lemmas. 

\begin{lemma}
  For any $v_h \in V_h^m$, there exists a function $\wt{v}_h \in
  \wt{V}_h^m$ such
  that 
  \begin{equation}
    \begin{aligned}
%      \|\nabla^\alpha v_h \|_{L^2(\wt{K})} &= \| \nabla^\alpha
      %\wt{v}_h \|_{L^2(\wt{K})}, && \text{on any } \wt{K} \in
      %\wt{\mc{T}}_h \text{ and for any } \alpha \geq 0, \\
      v_h &= \wt{v}_h, && \text{in any } \wt{K} \in \wt{\mc{T}}_h, \\
      \jump{\wt{v}_h} &= 0, && \text{on any } \wt{e} \in \wt{\mc{E}}_h
      \backslash \MEh, \\
     \sum_{\wt{e} \in w(e)} h_{\wt{e}}^{\beta} \| \jump{\wt{v}_h}
     \|_{L^2(\wt{e})}^2  & \leq C h_e^{\beta} \| \jump{v_h}
     \|_{L^2(e)}^2 , && \text{on any } e \in \MEh \text{ and }
     \beta = -1, 1, \\
    \end{aligned}
    \label{eq_vhwtvh}
  \end{equation}
  where $w(e) = \{ \wt{e} \in \wt{\mc{E}}_h \ | \ \wt{e} \subset e
  \}$.
  \label{le_vhwtvh}
\end{lemma}
\begin{proof}
  The fact $V_h^m \subset \wt{V}_h^m$ directly implies that there
  exists a polynomial $\wt{v}_h \in \wt{V}_h^m$ satisfying the
  equalities in \eqref{eq_vhwtvh} and  
  \begin{displaymath}
    \sum_{\wt{e} \in w(e)} \| \jump{\wt{v}_h} \|_{L^2(\wt{e})}^2 = \|
    \jump{v_h} \|_{L^2(e)}^2. 
  \end{displaymath}
  Hence, 
  \begin{displaymath}
    \begin{aligned} 
      \sum_{\wt{e} \in w(e)} h_{\wt{e}}^{\beta} \| \jump{\wt{v}_h}
      \|_{L^2(\wt{e})}^2 = \sum_{\wt{e} \in w(e)} h_{e}^{\beta}
      \left( \frac{h_{\wt{e}}}{h_e} \right)^{\beta} \|
      \jump{\wt{v}_h} \|_{L^2(\wt{e})}^2 \leq C h_e^{\beta}
      \sum_{\wt{e} \in w(e)} \| \jump{\wt{v}_h} \|_{L^2(\wt{e})}^2
      = C h_e^{\beta} \| \jump{v_h} \|_{L^2(e)}^2,
    \end{aligned}
  \end{displaymath}
  where the last inequality follows from the mesh regularity
  assumption. This completes the proof.
\end{proof}

\begin{lemma}
  For any $v_h \in V_h^m$, there exists a function $\wt{v}_h \in
  \wt{V}_h^m \cap H^1(\Omega)$ such that 
  \begin{equation}
    \sum_{{K} \in {\mc{T}}_h} \|\nabla^\alpha (v_h - \wt{v}_h)
    \|_{L^2({K})}^2 \leq C \sum_{e \in \MEh^i} h_e^{1 - 2\alpha} \|
    \jump{v_h} \|_{L^2(e)}^2, \quad \alpha = 0, 1.
    \label{eq_vdiff}
  \end{equation}
  \label{le_vdiff}
\end{lemma}

\begin{proof}
  By Lemma \ref{le_vhwtvh}, there exists a piecewise polynomial
  $\wh{v}_h \in \wt{V}_h^m $ satisfying the estimate
  \eqref{eq_vhwtvh}. By \cite[Theorem
  2.1]{Karakashian2007convergence}, there exists a function
  $\wt{v}_h \in \wt{V}_h^m \cap H^1(\Omega)$ such that 
  \begin{displaymath}
    \sum_{\wt{K} \in \wt{\mc{T}}_h} \| \nabla^\alpha (\wh{v}_h -
    \wt{v}_h) \|_{L^2(\wt{K})}^2 \leq C \sum_{\wt{e} \in
    \wt{\mc{E}}_h^i} h_{\wt{e}}^{1 - 2\alpha} \| \jump{\wh{v}_h}
    \|_{L^2(\wt{e})}^2, \quad \alpha = 0, 1.
  \end{displaymath}
  Combining \eqref{eq_vhwtvh}, we have that 
  \begin{displaymath}
%    \sum_{\wt{e} \in \wt{\mc{E}}_h^i}
    %h_{\wt{e}}^{1 - 2\alpha} \| \jump{\wh{v}_h} \|_{L^2(\wt{e})}^2
    %\leq C \sum_{e \in \MEh^i} h_e^{1 - 2\alpha} \| \jump{v_h}
    %\|_{L^2(e)}^2,
    \begin{aligned}
      \sum_{{K} \in {\mc{T}}_h} \| \nabla^\alpha ({v}_h - \wt{v}_h)
      \|_{L^2({K})}^2 & = \sum_{\wt{K} \in \wt{\mc{T}}_h} \|
      \nabla^\alpha (\wh{v}_h - \wt{v}_h) \|_{L^2(\wt{K})}^2 \leq C 
      \sum_{\wt{e} \in \wt{\mc{E}}_h^i}
      h_{\wt{e}}^{1 - 2\alpha} \| \jump{\wh{v}_h} \|_{L^2(\wt{e})}^2
      \\
      & \leq C  \sum_{e \in \MEh^i} h_e^{1 - 2\alpha} \| \jump{v_h}
      \|_{L^2(e)}^2,
    \end{aligned}
  \end{displaymath}
  which gives the inequality \eqref{eq_vdiff} and this completes the
  proof.
\end{proof}

\begin{lemma}
  For any $\bm{q}_h \in (\wt{V}_h^m)^d$, there exists a function
  $\wt{\bm{q}}_h \in (\wt{V}_h^m)^d \cap H^1(\Omega)^d$ such that 
  \begin{equation}
    \sum_{\wt{K} \in \wt{\mc{T}}_h} \|\nabla^\alpha (\bm{q}_h -
    \wt{\bm{q}}_h) \|_{L^2(\wt{K})}^2 \leq C\left( \sum_{\wt{e} \in
    \wt{\mc{E}}_h^i} h_{\wt{e}}^{1 - 2\alpha} \| \jump{\wt{\un}
    \otimes \bm{q}_h} \|_{L^2(\wt{e})}^2 + \sum_{\wt{e} \in
    \wt{\mc{E}}_h^b} h_{\wt{e}}^{1 - 2\alpha} \|\wt{\un} \times
    \bm{q}_h \|_{L^2(\wt{e})}^2 \right), \quad \alpha = 0, 1, 
    \label{eq_wqwqdiff}
  \end{equation}
  and the tangential trace $\un \times \wt{\bm{q}}_h$ vanishes on the
  boundary $\partial \Omega$.
  \label{le_wqwqinterpolate}
\end{lemma}

\begin{proof}
  Again by \cite[Theorem 2.1]{Karakashian2007convergence}, there
  exists a piecewise polynomial $\wh{\bm{q}}_h \in (\wt{V}_h^m)^d \cap
  H^1(\Omega)^d$ such that 
  \begin{equation}
    \sum_{\wt{K} \in \wt{\mc{T}}_h} \|\nabla^\alpha (\bm{q}_h -
    \wh{\bm{q}}_h) \|_{ L^2(\wt{K})}^2 \leq C \sum_{\wt{e} \in
    \wt{\mc{E}}_h^i} h_{\wt{e}}^{1 - 2\alpha} \| \jump{ \wt{\un}
    \otimes \bm{q}_h } \|_{L^2(\wt{e})}^2.
    \label{eq_qwhqdiff}
  \end{equation}
  We will construct a new piecewise polynomial $\wt{\bm{q}}_h \in
  (\wt{V}_h^m)^d \cap H^1(\Omega)^d$ based on $\wh{\bm{q}}_h$, which
  satisfies the inequality \eqref{eq_wqwqdiff} and its tangential
  trace vanishes on the boundary. 

  \newcommand\bu{\bm{\nu}}

  We denote by $\mc{N} = \left\{ \bm{\nu}_0, \bm{\nu}_1, \ldots,
  \bm{\nu}_n \right\}$ the Lagrange points with respect to the
  partition $\wt{\mc{T}}_h$ and we let $\{ \phi_{\bu_0}, \phi_{\bu_1},
  \ldots, \phi_{\bu_n} \}$ be the corresponding basis functions such
  that $\phi_{\bu_i}(\bu_j) = \delta_{ij} $. Then we divide the set
  $\mc{N}$ into three disjoint subsets (see Fig.~\ref{fig_Npoint}): 
  \begin{equation}
    \begin{aligned}
      \mc{N}_i & := \left\{ \bu \in \mc{N} \ | \  \bu \text{ is
      interior to the domain } \Omega \right\}, \\
      \mc{N}_v & := \{ \bu \in \mc{N} \ | \ \bu \text{ is shared
      by two different slides of the boundary } \partial \Omega \}, 
      \\
      \mc{N}_b & := \mc{N} \backslash ( \mc{N}_i \cup \mc{N}_v). \\
    \end{aligned}
    \label{eq_divideN}
  \end{equation}
  We note that the points in $\mc{N}_b$ are interior to one slide of
  the boundary $\partial \Omega$, and particularly in two dimensions
  the points in $\mc{N}_v$ are vertices of the boundary $\partial
  \Omega$. 

  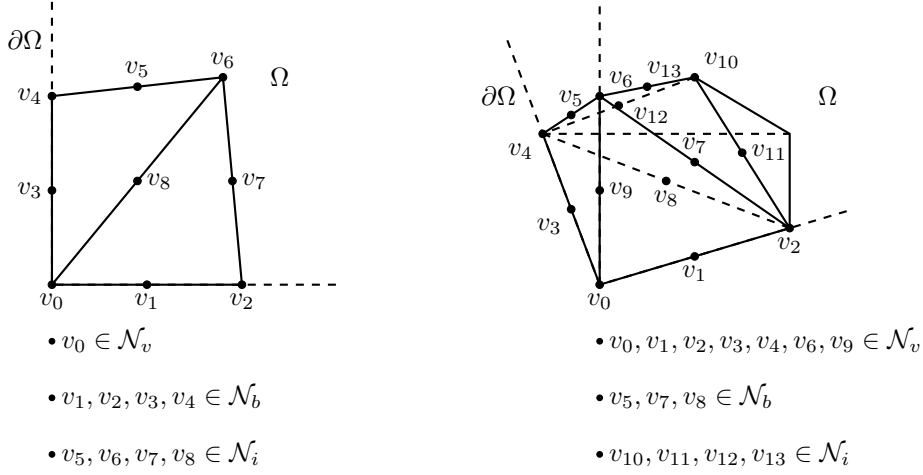
\begin{figure}[htp] 
    \centering
    \begin{minipage}[t]{0.46\textwidth}
      \begin{center}
        \begin{tikzpicture}[scale=2.5]
          \draw[thick] (0, 1) -- (0, 0) -- (1, 0);
          \draw[thick] (0, 0) -- (0.9, 1.1) -- (0, 1);
          \draw[thick] (0.9, 1.1) -- (1, 0);
          \draw[thick, dashed] (0, 1.5) -- (0, 0) -- (1.5, 0);
          \node[left] at (0, 1.3) {$\partial \Omega$};
          \node[] at (1.2, 1.1) {$\Omega$};
          \draw[fill, black] (0, 0) circle [radius = 0.02];
          \node[below] at (0, 0) {$v_0$};
          \draw[fill, black] (0.5, 0) circle [radius = 0.02];
          \node[below] at (0.5, 0) {$v_1$};
          \draw[fill, black] (1, 0) circle [radius = 0.02];
          \node[below] at (1, 0) {$v_2$};
          \draw[fill, black] (0, 0.5) circle [radius = 0.02];
          \node[left] at (0, 0.5) {$v_3$};
          \draw[fill, black] (0, 1) circle [radius = 0.02];
          \node[left] at (0, 1) {$v_4$};
          \draw[fill, black] (0.45, 1.05) circle [radius = 0.02];
          \node[above] at (0.45, 1.05) {$v_5$};
          \draw[fill, black] (0.9, 1.1) circle [radius = 0.02];
          \node[above] at (0.9, 1.1) {$v_6$};
          \draw[fill, black] (0.95, 0.55) circle [radius = 0.02];
          \node[right] at (0.95, 0.55) {$v_7$};
          \draw[fill, black] (0.45, 0.55) circle [radius = 0.02];
          \node[right] at (0.45, 0.55) {$v_8$};
          \draw[fill, black] (0, -0.3) circle [radius = 0.015];
          \node[right] at (0, -0.3) {$v_0 \in \mc{N}_v$};
          \draw[fill, black] (0, -0.6) circle [radius = 0.015];
          \node[right] at (0, -0.6) {$v_1, v_2, v_3, v_4 \in
          \mc{N}_b$};
          \draw[fill, black] (0, -0.9) circle [radius = 0.015];
          \node[right] at (0, -0.9) {$v_5, v_6, v_7, v_8 \in
          \mc{N}_i$};
        \end{tikzpicture}
      \end{center}
    \end{minipage}
    \begin{minipage}[t]{0.46\textwidth}
      \begin{center}
        \begin{tikzpicture}[scale=2.5]
          \draw[thick] (0, 0) -- (1, 0.3) -- (0, 1) -- (0, 0);
          \draw[thick] (0, 0) -- (-0.3, 0.8) -- (0, 1);
          \draw[thick, dashed] (-0.3, 0.8) -- (1, 0.3);
          \draw[thick] (0.5, 1.1) -- (0, 1);
          \draw[thick] (0.5, 1.1) -- (1, 0.3);
          \draw[thick, dashed] (0.5, 1.1) -- (-0.3, 0.8);
          \draw[thick, dashed] (1.3, 0.39) -- (0, 0) -- (-0.5,
          1.333333);
          \draw[thick, dashed] (0, 0) -- (0, 1.5);
          \draw[thick] (1, 0.3) -- (1, 0.8) -- (0.5, 1.1);
          \draw[thick, dashed] (-0.3, 0.8) -- (1, 0.8);
          \node[left] at (-0.39, 1.0) {$\partial \Omega$};
          \node[] at (1.2, 1) {$\Omega$};
          \draw[fill, black] (0.0, 0) circle [radius = 0.02];
          \node[below] at (0, 0) {$v_0$};
          %\draw[fill, black] (1, 0.8) circle [radius = 0.02];
          \draw[fill, black] (0.5, 0.15) circle [radius = 0.02];
          \node[below] at (0.5, 0.15) {$v_1$};
          \draw[fill, black] (1.0, 0.3) circle [radius = 0.02];
          \node[below] at (1.0, 0.3) {$v_2$};
          \draw[fill, black] (-0.15, 0.4) circle [radius = 0.02];
          \node[below left] at (-0.15, 0.4) {$v_3$};
          \draw[fill, black] (-0.3, 0.8) circle [radius = 0.02];
          \node[below left] at (-0.3, 0.8) {$v_4$};
          \draw[fill, black] (-0.15, 0.9) circle [radius = 0.02];
          \node[above] at (-0.15, 0.9) {$v_5$};
          \draw[fill, black] (0, 1.0) circle [radius = 0.02];
          \node[above right] at (0, 1.0) {$v_6$};
          \draw[fill, black] (0.5, 0.65) circle [radius = 0.02];
          \node[above] at (0.5, 0.65) {$v_7$};
          \draw[fill, black] (0.35, 0.55) circle [radius = 0.02];
          \node[below] at (0.35, 0.55) {$v_8$};
          \draw[fill, black] (0., 0.5) circle [radius = 0.02];
          \node[right] at (0., 0.5) {$v_9$};
          \node[above right] at (0.5, 1.1) {$v_{10}$};
          \draw[fill, black] (0.5, 1.1) circle [radius = 0.02];
          \node[above right] at (0.2, 1.05) {$v_{13}$};
          \draw[fill, black] (0.25, 1.05) circle [radius = 0.02];
          \node[right] at (0.75, 0.7) {$v_{11}$};
          \draw[fill, black] (0.1, 0.95) circle [radius = 0.02];
          \node[right] at (0.13, 0.9) {$v_{12}$};
          \draw[fill, black] (0.75, 0.7) circle [radius = 0.02];
          \draw[fill, black] (0, -0.3) circle [radius = 0.015];
          \node[right] at (0, -0.3) {$v_0, v_1, v_2, v_3, v_4,
          v_6, v_9 \in \mc{N}_v$};
          \draw[fill, black] (0, -0.6) circle [radius = 0.015];
          \node[right] at (0, -0.6) {$v_5, v_7, v_8 \in \mc{N}_b$};
          \draw[fill, black] (0, -0.9) circle [radius = 0.015];
          \node[right] at (0, -0.9) {$v_{10}, v_{11}, v_{12}, v_{13}
          \in \mc{N}_i$};
        \end{tikzpicture}
      \end{center}
    \end{minipage}
    \caption{Examples of Lagrange nodes in two dimensions (left) /
    in three dimensions (right).}
    \label{fig_Npoint}
  \end{figure}
  By $\{ \phi_{\bu_i} \}$, the function $\wh{\bm{q}}_h = \left(
  \wh{q}^1_h, \ldots, \wh{q}^d_h\right)$ can be expanded as
  $\wh{q}^i_h = \sum_{\bu \in \mc{N}} \alpha_{\bu}^j \phi_{\bu}(1 \leq
  j \leq d)$. Then we construct a new group of coefficients $\{
  \beta_{\bu}^j \}$ by 
  \begin{equation}
    \beta_{\bu}^j := \begin{cases}
      \alpha_{\bu}^j, & \text{for } \bu \in \mc{N}_i, \\
      \wt{\beta}_{\bu}^j, & \text{for } \bu \in \mc{N}_b, \\
      0, & \text{for } \bu \in \mc{N}_v. \\
    \end{cases}
    \label{eq_betadef}
  \end{equation}
  For $\bu \in \mc{N}$, we denote $\bm{\alpha}_{\bu}$ and
  $\bm{\beta}_{\bu}$ as $\bm{\alpha}_{\bu} = (\alpha_{\bu}^1, \ldots,
  \alpha_{\bu}^d)^T$ and $\bm{\beta}_{\bu} = (\beta_{\bu}^1, \ldots,
  \beta_{\bu}^d)^T$, respectively.  Then we determine the values of
  $\wt{\beta}_{\bu}^j$. By the definition \eqref{eq_divideN}, for any
  $\bu \in \mc{N}_b$ there exists a boundary face $\wt{e} \in
  \wt{\mc{E}}_h^b$ that includes the point $\bu$, and we let its
  corresponding coefficients satisfy that 
  \begin{displaymath}
    \wt{\un} \times \bm{\beta}_{\bu} = 0, \quad \wt{\un} \cdot
    \bm{\beta}_{\bu} = \wt{\un} \cdot \bm{\alpha}_{\bu},
  \end{displaymath}
  where $\wt{\un}$ is the unit outward normal with respect to the
  boundary face $\wt{e}$. We construct a new piecewise polynomial
  $\wt{\bm{q}}_h = (\wt{q}_h^1, \ldots, \wt{q}_h^d)^T $ where
  $\wt{q}_h^j = \sum_{\bu \in \mc{N}} \beta_{\bu}^j \phi_{\bu}(1 \leq
  j \leq d)$. It is trivial to check that the trace $\un \times
  \wt{\bm{q}}_h$ vanishes on the boundary $\partial \Omega$.  Then we
  will estimate the term $\|\nabla ( \wh{\bm{q}}_h - \wt{\bm{q}}_h)
  \|_{L^2(\Omega)}$.  Since $\wh{\bm{q}}_h$ and $\wt{\bm{q}}_h$ have
  the same value on the points in $\mc{N}_i$, one can see that
  \begin{displaymath}
    \| \nabla^\alpha ( \wh{\bm{q}}_h - \wt{\bm{q}}_h)
    \|_{L^2(\Omega)}^2 \leq C \sum_{\bu \in \mc{N}_b \cup \mc{N}_v}
    |\bm{\alpha}_{\bu} - \bm{\beta}_{\bu} |^2 \|\nabla^\alpha
    \phi_{\bu} \|_{L^2(\Omega)}^2.
  \end{displaymath}
  We first consider the points in the set $\mc{N}_b$. Again for any
  $\bu \in \mc{N}_b$, we have that there exists a face $\wt{e} \in
  \wt{\mc{E}}_h^b$ such that $\bu \in \wt{e}$, and by the scaling
  argument and the shape regularity of the partition $\wt{\mc{T}}_h$,
  there holds $\|\nabla^\alpha \phi_{\bu} \|_{L^2(\Omega)}^2 \leq C
  h_{\wt{e}}^{d - 2\alpha}$.  Combining with \eqref{eq_betadef} and
  the inverse estimate, we obtain that 
  \begin{displaymath}
    \begin{aligned}
      \sum_{\bu \in \mc{N}_b} |\bm{\alpha}_{\bu} - \bm{\beta}_{\bu}
      |^2 \|\nabla^\alpha \phi_{\bu} \|_{L^2(\Omega)}^2 &= \sum_{\bu
      \in \mc{N}_b} \|\nabla^\alpha \phi_{\bu} \|_{L^2(\Omega)}^2
      \left( | \wt{\un} \times (\bm{\alpha}_{\bu} - \bm{\beta}_{\bu})
      |^2 + | \wt{\un} \cdot (\bm{\alpha}_{\bu} - \bm{\beta}_{\bu})
      |^2 \right) \\
      & \leq C \sum_{\bu \in \mc{N}_b} h_{\wt{e}}^{d - 2\alpha}
      |\wt{\un} \times \bm{\alpha}_{\bu} |^2 = C  \sum_{\bu \in
      \mc{N}_b} h_{\wt{e}}^{d - 2 \alpha} |\wt{\un} \times
      \wh{\bm{q}}_h(\bu) |^2 \\
      &\leq C \sum_{\bu \in \mc{N}_b} h_{\wt{e}}^{d - 2 \alpha}
      \|\wt{\un} \times \wh{\bm{q}}_h \|_{L^{\infty}(\wt{e})}^2 \leq C
      \sum_{\bu \in \mc{N}_b} h_{\wt{e}}^{1 - 2\alpha} \|\wt{\un}
      \times \wh{\bm{q}}_h \|_{L^{2}(\wt{e})}^2 \\
      & \leq C \sum_{\wt{e} \in \wt{\mc{E}}_h^b} h_{\wt{e}}^{1 -
      2\alpha} \|\wt{\un} \times \wh{\bm{q}}_h \|_{L^{2}(\wt{e})}^2.
    \end{aligned}
  \end{displaymath}
  Then we move on to the points in $\mc{N}_v$. By definition
  \eqref{eq_divideN}, for every $\bu \in \mc{N}_v$ there exist two
  adjacent faces $\wt{e}_1 \in \wt{\mc{E}}_h^b$ and $\wt{e}_2 \in
  \wt{\mc{E}}_h^b$ such that $\bu \in \wt{e}_1 \cap \wt{e}_2$. We note
  that $\wt{e}_1$ and $\wt{e}_2$ are not parallel and are included in
  two different slides of $\partial \Omega$. We let $\wt{\un}_1$ and
  $\wt{\un}_2$ be the unit outward normals corresponding to $\wt{e}_1$
  and $\wt{e}_2$, respectively, and clearly we have $\wt{\un}_1 \neq
  \wt{\un}_2$. This fact implies that there exists a positive constant
  $C$ such that 
  \begin{equation}
    |\bm{v} |^2 \leq C \left( | \wt{\un}_1 \times \bm{v} |^2 + |
    \wt{\un}_2 \times \bm{v} |^2 \right), \quad \text{for } \forall
    \bm{v} \in \mb{R}^d.
    \label{eq_vn1vn2v}
  \end{equation}
  It should be noted that the constant $C$ only replies on the angle
  of $\wt{\un}_1$ and $\wt{\un}_2$ and this angle only depends on the
  boundary $\partial \Omega$. Then we derive that 
  \begin{displaymath}
    \begin{aligned}
      \sum_{\bu \in \mc{N}_v} |\bm{\alpha}_{\bu} - \bm{\beta}_{\bu}
      |^2 \|\nabla^\alpha \phi_{\bu} \|_{L^2(\Omega)}^2 &= \sum_{\bu
      \in \mc{N}_v} \|\nabla^\alpha \phi_{\bu} \|_{L^2(\Omega)}^2 |
      \bm{\alpha}_{\bu} |^2 \\ 
      &\leq C \sum_{\bu \in \mc{N}_v} \|\nabla^\alpha \phi_{\bu}
      \|_{L^2(\Omega)}^2 \left( | \wt{\un}_1 \times \bm{\alpha}_{\bu}
      |^2 + | \wt{\un}_2 \times \bm{\alpha}_{\bu} |^2 \right) \\
      & \leq C \sum_{\bu \in \mc{N}_v} \left( h_{\wt{e}_1}^{d -
      2\alpha} |\wt{\un}_1 \times \wh{\bm{q}}_h({\bu}) |^2 +
      h_{\wt{e}_2}^{d -
      2\alpha} |\wt{\un}_2 \times \wh{\bm{q}}_h({\bu}) |^2\right) \\
      & \leq C \sum_{\bu \in \mc{N}_v}  \left( h_{e_1}^{d - 2\alpha}
      \| \wt{\un}_1 \times \wh{\bm{q}}_h \|_{L^{\infty}(\wt{e}_1)}^2 +
      h_{\wt{e}_2}^{d - 2\alpha} \| \wt{\un}_2 \times \wh{\bm{q}}_h
      \|_{L^{\infty}(\wt{e}_2)}^2\right) \\
      & \leq C \sum_{\bu \in \mc{N}_v} \left(  h_{\wt{e}_1}^{1 -
      2\alpha} \| \wt{\un}_1 \times \wh{\bm{q}}_h
      \|_{L^{2}(\wt{e}_1)}^2 + h_{\wt{e}_2}^{1 - 2\alpha} \|
      \wt{\un}_1 \times \wh{\bm{q}}_h \|_{L^{2}(\wt{e}_2)}^2 \right)
      \\
      & \leq C \sum_{\wt{e} \in \wt{\mc{E}}_h^b} h_{\wt{e}}^{1 -
      2\alpha} \| \wt{\un} \times \wh{\bm{q}}_h \|_{L^{2}(\wt{e})}^2.
    \end{aligned}
  \end{displaymath}
  Thus, we arrive at 
  \begin{displaymath}
    \|\nabla^\alpha (\wh{\bm{q}}_h - \wt{\bm{q}}_h) \|_{L^2(\Omega)}^2
    \leq C \sum_{e \in \wt{\mc{E}}_h^b} h_e^{1 - 2\alpha} \|\un \times
    \wh{\bm{q}}_h \|_{L^{2}(e)}^2.
  \end{displaymath}
  We finally present that the error $\|\nabla^\alpha (\bm{q}_h -
  \wt{\bm{q}}_h) \|_{L^2(\wt{\mc{T}}_h)}$ satisfies the estimate
  \eqref{eq_wqwqdiff}. We have that 
  \begin{displaymath}
    \begin{aligned}
      \| \nabla^\alpha (\bm{q}_h - \wt{\bm{q}}_h)
      \|_{L^2(\wt{\mc{T}}_h)}^2 & \leq C \left( \| \nabla^\alpha
      (\bm{q}_h - \wh{\bm{q}}_h) \|_{L^2(\wt{\mc{T}}_h)}^2 + \|
      \nabla^\alpha (\wt{\bm{q}}_h - \wh{\bm{q}}_h)
      \|_{L^2(\Omega)}^2 \right) \\
      & \leq C \left( \sum_{\wt{e} \in \wt{\mc{E}}_h^i} h_{\wt{e}}^{1
      - 2\alpha} \| \jump{\un \otimes \bm{q}_h} \|_{L^2(\wt{e})}^2 +
      \sum_{\wt{e} \in \wt{\mc{E}}_h^b} h_{\wt{e}}^{1 - 2\alpha} \|
      \un \times \wh{\bm{q}}_h \|_{L^2(\wt{e})}^2   \right), 
    \end{aligned}
  \end{displaymath}
  and together with \cite[Theorem 2.1]{Karakashian2007convergence} and
  the trace inequality, we deduce that
  \begin{displaymath}
    \begin{aligned}
      \sum_{e \in \wt{\mc{E}}_h^b} h_{\wt{e}}^{1 - 2\alpha} \|
      \wt{\un} \times \wh{\bm{q}}_h \|_{L^2(\wt{e})}^2 & \leq C
      \sum_{e \in \wt{\mc{E}}_h^b} \left( h_{\wt{e}}^{1 - 2\alpha} \|
      \wt{\un} \times {\bm{q}}_h \|_{L^2(\wt{e})}^2 +   h_{\wt{e}}^{1
      - 2\alpha} \| \wt{\un} \times ( {\bm{q}}_h - \wh{\bm{q}}_h)
      \|_{L^2(\wt{e})}^2 \right) \\
      & \leq C \left(  \sum_{\wt{e} \in \wt{\mc{E}}_h^b} h_{\wt{e}}^{1
      - 2\alpha} \| \wt{\un} \times {\bm{q}}_h \|_{L^2(\wt{e})}^2 +
      \sum_{K \in \wt{\mc{T}}_h} h_K^{-2\alpha} \|\bm{q}_h -
      \wh{\bm{q}}_h \|_{L^2(K)}^2   \right) \\
      & \leq C   \left(  \sum_{\wt{e} \in \wt{\mc{E}}_h^i}
      h_{\wt{e}}^{1 - 2\alpha} \| \wt{\un} \otimes {\bm{q}}_h
      \|_{L^2(\wt{e})}^2 +   \sum_{\wt{e} \in \wt{\mc{E}}_h^b}
      h_{\wt{e}}^{1 - 2\alpha} \| \wt{\un} \times {\bm{q}}_h
      \|_{L^2(\wt{e})}^2 \right). \\
    \end{aligned}
  \end{displaymath}
  Combining all inequalities yields the estimate \eqref{eq_wqwqdiff}
  and completes the proof.
\end{proof}

\begin{lemma}
  For any $\bmr{V}_h \in (V_h^m)^{d \times d}$, there
  exists a function $\wt{\bmr{V}}_h \in (\wt{V}_h^m)^{d \times d}
  \cap H^1(\Omega)^{d \times d}$ such that 
  \begin{equation}
    \sum_{K \in \MTh} \| \nabla^\alpha (\bmr{V}_h - \wt{\bmr{V}}_h)
    \|_{L^2(K)}^2 \leq C \left( \sum_{e \in \MEh^i} h_e^{1 - 2\alpha}
    \| \jump{ \un \otimes \bmr{V}_h} \|_{L^2(e)}^2 + \sum_{e \in
    \MEh^b} h_e^{1 - 2\alpha} \|  \un \times \bmr{V}_h \|_{L^2(e)}^2
    \right), 
    \label{eq_qinterpolate}
  \end{equation}
  for $\alpha = 0, 1$, and the tangential trace $\un \times
  \wt{\bmr{V}}_h$ vanishes on the boundary $\partial \Omega$.
  \label{le_qinterpolate}
\end{lemma}

\begin{proof}
  We columnwise expand $\bmr{V}_h$ as $\bmr{V}_h = (\bm{v}_h^1,
  \ldots, \bm{v}_h^d)$ where $\bm{v}_h^i \in (V_h^m)^d(1 \leq i \leq
  d)$. By Lemma \ref{le_vhwtvh}, there exist a piecewise polynomial
  function $\wh{\bm{v}}_h^i \in (\wt{V}_h^m)^d$ such that
  \begin{equation}
    \begin{aligned}    
%      \|\nabla^\alpha \bm{v}_i \|_{L^2(\wt{K})} &= \| \nabla^\alpha
      %\wh{\bm{v}}_i \|_{L^2(\wt{K})}, &&\text{on} \quad \forall \wt{K}
      %\in \wt{\mc{T}}_h,  \\
      %\jump{\wh{\bm{v}}_i} &= 0, && \text{on} \quad \forall \wt{e} \in
      %\wt{\mc{E}}_h \backslash \MEh. \\
%      \|\nabla^\alpha \wh{\bm{v}}_h^i \|_{L^2(\wt{K})} &= \|
      %\nabla^\alpha {\bm{v}}_h^i \|_{L^2(\wt{K})}, && \text{on any }
      %\wt{K} \in \wt{\mc{T}}_h \text{ and for any } \alpha \geq 0, \\
      \bm{v}_h^i & = \wh{\bm{v}}_h^i, && \text{in any } \wt{K} \in
      \wt{\mc{T}}_h, \\
      \jump{\wt{\un} \otimes \wh{\bm{v}}_h^i} &= 0, && \text{on any }
      \wt{e} \in \wt{\mc{E}}_h \backslash \MEh, \\
      \sum_{\wt{e} \in w(e)} h_{\wt{e}}^{1 - 2\alpha} \|
      \jump{\wt{\un} \otimes \wh{\bm{v}}_h^i} \|_{L^2(\wt{e})}^2  &
      \leq C h_e^{1 -2\alpha} \| \jump{{\un} \otimes \bm{v}_h^i}
      \|_{L^2(e)}^2 , && \text{on any } e \in \MEh^i, \\
      \sum_{\wt{e} \in w(e)} h_{\wt{e}}^{1 - 2\alpha} \|
      \jump{\wt{\un} \times \wh{\bm{v}}_h^i} \|_{L^2(\wt{e})}^2  &
      \leq C h_e^{1 -2\alpha} \| \jump{{\un} \times \bm{v}_h^i}
      \|_{L^2(e)}^2 , && \text{on any } e \in \MEh^b. \\
      \end{aligned}
      \label{eq_bmVwtbmV}
  \end{equation}
  By Lemma \ref{le_wqwqinterpolate}, for every $\wh{\bm{v}}_h^i$ there
  exists a piecewise polynomial function $\wt{\bm{v}}_h^i \in
  (\wt{V}_h^m)^d \cap H^1(\Omega)^d$ satisfying the estimate
  \eqref{eq_wqwqdiff} and the tangential trace of $\wt{\bm{v}}_h^i$
  equals to $0$ on the boundary.  We define $\wt{\bmr{V}}_h$ as
  $\wt{\bmr{V}}_h = (\wt{\bm{v}}_h^1, \ldots, \wt{\bm{v}}_h^d)$. By
  \eqref{eq_bmVwtbmV} and the estimate \eqref{eq_wqwqdiff}, it can be
  seen that for the polynomial $\wt{\bmr{V}}_h$ the estimate
  \eqref{eq_qinterpolate} holds true and its tangential trace vanishes
  on $\partial \Omega$.  This completes the proof.
\end{proof}

Next, we focus on the continuity and the coercivity of the bilinear
form $a_h^{\bmr{p}}(\cdot; \cdot)$. We begin by defining the following
two energy norms $\Unorm{\cdot}$ and $\pnorm{\cdot}$: 
\begin{displaymath}
  \begin{aligned}
    \Unorm{\bmr{V}_h}^2 :=  \sum_{K \in \MTh}\|\nabla \cdot \bmr{V}_h
    \|_{L^2(K)}^2 + & \sum_{e \in \MEh^i} h_e^{-1} \| \jump{ \un
    \otimes \bmr{V}_h} \|_{L^2(e)}^2 + \sum_{e \in \MEh^b} h_e^{-1} \|
    \un \times \bmr{V}_h \|_{L^2(e)}^2, \\
  \end{aligned}
\end{displaymath}
for any $\bmr{V}_h \in \bmr{I}_h^m + \bmr{I}^1(\Omega)$, and
\begin{displaymath}
  \begin{aligned}
    \pnorm{q_h}^2 :=&  \sum_{K \in \MTh} \|\nabla q_h\|_{L^2(K)}^2 + 
    \sum_{e \in \MEh^i} h_e^{-1} \| \jump{q_h} \|_{L^2(e)}, \\
  \end{aligned}
\end{displaymath}
and any $q_h \in \wt{U}_h^m + H^1(\Omega) / \mb{R}$.  We have the
following estimates which show that $\Unorm{\cdot}$ and
$\pnorm{\cdot}$ are actually norms on their corresponding spaces. 

\begin{lemma}
  There exist a constant $C$ such that
  \begin{equation}
    \|\bmr{V}_h \|_{L^2(\Omega)} \leq C \Unorm{\bmr{V}_h}, \quad
    \forall \bmr{V}_h \in \bmr{I}_h^m + \bmr{I}^1(\Omega).
    \label{eq_Unormbound}
  \end{equation}
  \label{le_Unormbound}
\end{lemma}

\begin{proof}
  we refer to \cite[Lemma 4.1]{li2019sequential} for the proof.
\end{proof}

\begin{lemma}
  There exist a constant $C$ such that 
  \begin{equation}
    \|q_h\|_{L^2(\Omega)} \leq C \pnorm{q_h}, \quad \forall q_h \in
    \wt{U} + H^1(\Omega)/R.
    \label{eq_pnormbound}
  \end{equation}
  \label{le_pnormblund}
\end{lemma}
\begin{proof}
  We refer to \cite{Brenner2003poincare} for the proof.
\end{proof}

Now we are ready to state that the bilinear form $a_h^{\bmr{p}}(\cdot;
\cdot)$ is bounded and coercive with respect to energy norms
$\Unorm{\cdot}$ and $\pnorm{\cdot}$ for any positive $\eta$. 

\begin{lemma}
  For the bilinear form $a_h^{\bmr{p}}(\cdot; \cdot)$ with any $\eta >
  0$, there exists a positive constant $C$ such that 
  \begin{equation}
    |a_h^{\bmr{p}}(\bmr{U}_h, p_h; \bmr{V}_h, q_h)| \leq C \left(
    \Unorm{\bmr{U}_h}^2 + \pnorm{p_h}^2 \right)^{1/2}  \left(
    \Unorm{\bmr{V}_h}^2 + \pnorm{q_h}^2 \right)^{1/2},
    \label{eq_apboundedness}
  \end{equation}
  for any $\bmr{U}_h, \bmr{V}_h \in \bmr{I}_h^m + \bmr{I}^1(\Omega)$
  and any $p_h, q_h \in \wt{U}_h^m + H^1(\Omega) / \mb{R}$.
  \label{le_apboundedness}
\end{lemma}

\begin{proof}
  We prove for the case $\nu = 1$ and it is trivial to extend the
  proof to the case $\nu > 0$.  Obviously we have that 
  \begin{displaymath}
    \int_K \| -\nabla \cdot \bmr{U}_h + \nabla p_h \|^2 \d{x}
    \leq C \left( \int_K \| \nabla \cdot \bmr{U}_h \|^2 \d{x}
    + \int_K \|\nabla p_h \|^2 \d{x} \right), \quad \forall K \in
    \MTh, 
  \end{displaymath}
  and applying the Cauchy-Schwarz inequality to \eqref{eq_apbilinear}
  directly gives the estimate \eqref{eq_apboundedness}, which
  completes the proof.
\end{proof}

\begin{lemma}
  For the bilinear form $a_h^{\bmr{p}}(\cdot; \cdot)$ with any $\eta >
  0$, there exists a positive constant $C$ such that 
  \begin{equation}
    a_h^{\bmr{p}}(\bmr{U}_h, p_h; \bmr{U}_h, p_h) \geq C \left(
    \Unorm{\bmr{U}_h}^2 + \pnorm{p_h}^2 \right),
    \label{eq_apcoercive}
  \end{equation}
  for any $\bmr{U}_h \in \bmr{I}_h^m$ and any $p_h \in \wt{U}_h^m$.
  \label{le_apcoercive}
\end{lemma}

\begin{proof}
  We prove for the case $\nu = 1$ and it is easy to extend the proof
  to other cases. For any $\bmr{U}_h \in \bmr{I}_h^m$, Lemma
  \ref{le_qinterpolate} implies that there exists a function
  $\bmr{V}_h \in H^1(\Omega)^{d \times d}$ such that $\un \times
  \bmr{V}_h = 0$ on $\partial \Omega$ and the estimate
  \eqref{eq_qinterpolate} holds. For any $p_h \in \wt{U}_h^m$, there
  exists a function $q_h \in H^1(\Omega)$ satisfying the estimate
  \eqref{eq_vdiff} by Lemma \ref{le_vdiff}. 

  Here we prove for the three-dimensional case.  Since $\un \times
  \bmr{V}_h  = 0$ on $\partial \Omega$ and the domain $\Omega$ is
  assumed to be a bounded convex polygon (polyhedron), we have the
  following Helmholtz decomposition \cite{Cai1997first}:  
  \begin{displaymath}
    \bmr{V}_h = \nabla \bm{q}^T + \nabla \times \bm{ \Psi },
  \end{displaymath}
  where $\bm{q} \in H_0^1(\Omega)^d \cap H^2(\Omega)^d$ is the
  solution of 
  \begin{equation}
    \Delta \bm{q} = \nabla \cdot \bmr{V}_h, \quad \text{in } \Omega,
    \quad \bm{q} = 0, \quad \text{on } \partial \Omega.
    \label{eq_bmqbmrVh}
  \end{equation}
  Since $q_h \in H^1(\Omega)$, the regularity of generalized Stokes
  problem \cite{Cai1997first, Kellogg1976regularity} provides that 
  \begin{equation}
    \| \Delta \bm{q} \|_{L^2(\Omega)}^2 + \| \nabla q_h
    \|_{L^2(\Omega)}^2 \leq C \left( \|- \Delta \bm{q} + \nabla
    q_h\|_{L^2(\Omega)}^2 + \| \nabla \cdot \bm{q} \|_{L^2(\Omega)}^2
    \right).
    \label{eq_regularityStokes}
  \end{equation}
  Together with \cite[Lemma 3.2]{Cai1997first} and the auxiliary
  problem \eqref{eq_bmqbmrVh}, we obtain 
  \begin{equation}
    \| \nabla \cdot \bmr{V}_h \|_{L^2(\Omega)}^2 + \| \nabla q_h
    \|_{L^2(\Omega)}^2 \leq C \left( \| -\nabla \cdot \bmr{V}_h +
    \nabla q_h \|_{L^2(\Omega)}^2 + \| \tr{\bmr{V}_h}
    \|_{H^1(\Omega)}^2 + \|\nabla \times \bmr{V}_h \|_{L^2(\Omega)}^2
    \right).
    \label{eq_VhqhStokes}
  \end{equation}
  We note that the inequality \eqref{eq_VhqhStokes} also holds in two
  dimensions and the proof is similar. 
  
  Then we are ready to establish the coercivity \eqref{eq_apcoercive}
  and we first take the parameter $\eta = 1$.  We have that 
  \begin{displaymath}
    \begin{aligned}
      \Unorm{\bmr{U}_h}^2 + \pnorm{p_h}^2  =  \sum_{K \in \MTh} 
      \|\nabla \cdot \bmr{U}_h & \|_{L^2(K)}^2 +   \sum_{K \in \MTh}
      \|\nabla p_h\|_{L^2(K)}^2 \\
      + \sum_{e \in \MEh^i} h_e^{-1} \| \jump{ \un \otimes \bmr{U}_h}
      \|_{L^2(e)}^2 & + \sum_{e \in \MEh^b} h_e^{-1} \| \un \times
      \bmr{V}_h \|_{L^2(e)}^2 + \sum_{e \in \MEh^i} h_e^{-1} \|
      \jump{q_h} \|_{L^2(e)}, \\
    \end{aligned}
  \end{displaymath}
  and
  \begin{displaymath}
    \begin{aligned}
      \| \nabla \cdot \bmr{U}_h \|_{L^2(\MTh)}^2 + \|\nabla p_h
      \|_{L^2(\MTh)}^2 & \leq C  \left( \| \nabla \cdot \bmr{V}_h
      \|_{L^2(\Omega)}^2 + \|\nabla q_h \|_{L^2(\Omega)}^2 \right) \\
      & + C \left(\| \nabla
      \cdot (\bmr{U}_h -  \bmr{V}_h) \|_{L^2(\MTh)}^2 + \|\nabla (p_h
      - q_h) \|_{L^2(\MTh)}^2 \right). \\
    \end{aligned}
  \end{displaymath}
  From \eqref{eq_VhqhStokes} and the above two inequalities, we
  arrive at 
  \begin{displaymath}
    \Unorm{\bmr{U}_h}^2 + \pnorm{p_h}^2 \leq C  \left(
    a_h^{\bmr{p}}(\bmr{U}_h, p_h; \bmr{U}_h, p_h) +  \| \nabla \cdot
    \bmr{V}_h \|_{L^2(\Omega)}^2 + \|\nabla q_h \|_{L^2(\Omega)}^2
    \right).
  \end{displaymath}
  Thus, it
  suffices to show that the right hand side of \eqref{eq_VhqhStokes}
  can be bounded by $a_h^{\bmr{p}}(\bmr{U}_h, p_h; \bmr{U}_h, p_h)$.
  We further deduce that 
  \begin{displaymath}
    \begin{aligned}
      \|- \nabla \cdot \bmr{V}_h + \nabla q_h \|_{L^2(\Omega)}^2 &
      \leq C \left(  \|- \nabla \cdot \bmr{U}_h + \nabla p_h
      \|_{L^2(\MTh)}^2 \right) \\ 
      & + C \left( \| \nabla \cdot (\bmr{U}_h -  \bmr{V}_h)
      \|_{L^2(\MTh)}^2 + \|\nabla (p_h - q_h) \|_{L^2(\MTh)}^2\right), 
    \end{aligned}
  \end{displaymath}
  and since $\bmr{U}_h \in \bmr{I}_h^m$, we have 
  \begin{displaymath}
    \begin{aligned}
      \|\tr{\bmr{V}_h} \|_{H^1(\Omega)}^2 + \|\nabla \times
      \bmr{V}_h\|_{L^2(\Omega)}^2 &= \| \tr{\bmr{V}_h -
      \bmr{U}_h } \|_{H^1(\MTh)}^2 + \| \nabla \times ( \bmr{V}_h -
      \bmr{U}_h) \|_{L^2(\MTh)}^2 \\
      & \leq C \|( \bmr{V}_h - \bmr{U}_h) \|_{H^1(\MTh)}^2.
    \end{aligned}
  \end{displaymath}
  Combining all inequalities above and the estimate
  \eqref{eq_wqwqdiff} and \eqref{eq_vdiff},  we conclude that 
  \begin{displaymath}
    \begin{aligned}
  %     \Unorm{\bmr{U}_h}^2 & + \pnorm{p_h}^2  \leq C \|- \nabla \cdot
       %\bmr{U}_h + \nabla p_h \|_{L^2(\MTh)}^2 \\
       %+ C &\left( \sum_{e \in \MEh^i} h_e^{-1} \| \jump{ \un \otimes
       %\bmr{U}_h} \|_{L^2(e)}^2 +  \sum_{e \in \MEh^b} h_e^{-1} \|
       %\un \times \bmr{U}_h \|_{L^2(e)}^2 +  \sum_{e \in \MEh^i}
       %h_e^{-1} \| \jump{q_h} \|_{L^2(e)} \right). \\
       \Unorm{\bmr{U}_h}^2 + \pnorm{p_h}^2 \leq C
       a_h^{\bmr{p}}(\bmr{U}_h, p_h; \bmr{U}_h, p_h). \\
     \end{aligned}
   \end{displaymath}
   By a scaling argument, we can obtain that for any positive parameter
   $\eta$ the coercivity \eqref{eq_apcoercive} holds true, which
   completes the proof.
\end{proof}

We have established the boundedness and coercivity of the bilinear
form $a_h^{\bmr{p}}(\cdot; \cdot)$, which implies that there exists a
unique solution the discrete problem \eqref{eq_bilinearp}. We state
the error estimate to the numerical approximations obtained by
\eqref{eq_bilinearp}. 

\begin{theorem}
  Let $(\bmr{U}, p) \in \bmr{I}^{m+1}(\Omega) \times H^{m + 1}(\Omega)
  / \mb{R}$ be the solution to the problem \eqref{eq_firstorderUp} and
  let $(\bmr{U}_h, p_h) \in \bmr{I}_h^m \times \wt{U}_h^m$ be the
  solution to the problem \eqref{eq_bilinearp}, there exists a
  constant $C$ such that 
  \begin{equation}
    \Unorm{\bmr{U} - \bmr{U}_h} + \pnorm{p - p_h} \leq Ch^m \left(
    \|\bmr{U}\|_{H^{m+1}(\Omega)} + \|p\|_{H^{m+1}(\Omega)} \right).
    \label{eq_perrorestimate}
  \end{equation}
  \label{th_perrorestimate}
\end{theorem}

\begin{proof}
  For the exact solution $(\bmr{U}, p)$, the jump term vanishes on 
  interior faces, that is
  \begin{displaymath}
    \jump{\un \times \bmr{U}} = 0, \quad \jump{p} = 0, \quad \text{on
    any } e \in \MEh^i.
  \end{displaymath}
  Hence, for any $(\bmr{V}_h, q_h) \in \bmr{I}_h^m \times \wt{U}_h^m$
  we have that 
  \begin{displaymath}
    J_h^{\bmr{p}} (\bmr{V}_h, q_h) = a_h^{\bmr{p}}(\bmr{U} -
    \bmr{V}_h, p - q_h; \bmr{U} - \bmr{V}_h, p - q_h).
  \end{displaymath}
  We let $\bmr{V}_h = \wh{\mc{R}}^m \bmr{U}$ and $q_h = \mc{R}^m p$,
  and together with the coercivity \eqref{eq_apcoercive} and the
  boundedness \eqref{eq_apboundedness}, we obtain that 
  \begin{displaymath}
    \begin{aligned}
      \Unorm{\bmr{U} - \bmr{U}_h} + \pnorm{p - p_h} & \leq C
      a_h^{\bmr{p}}(\bmr{U} - \bmr{U}_h, p - p_h; \bmr{U} - \bmr{U}_h,
      p - p_h)^{1/2} = C   J_h^{\bmr{p}} (\bmr{U}_h, p_h)^{1/2} \\
      & \leq C  J_h^{\bmr{p}} (\bmr{V}_h, q_h)^{1/2} \leq C
      a_h^{\bmr{p}}(\bmr{U} - \bmr{V}_h, p - q_h; \bmr{U} - \bmr{V}_h,
      p - q_h)^{1/2} \\
      & \leq C \left(  \Unorm{\bmr{U} - \bmr{V}_h} + \pnorm{p - q_h}
      \right).\\ 
    \end{aligned}
  \end{displaymath}
  Applying the approximation estimates \eqref{eq_tensorapproximation}
  and \eqref{eq_scalarapproximation} and the trace estimate, it is
  trivial to obtain 
  \begin{displaymath}
    \Unorm{\bmr{U} - \bmr{V}_h} \leq C h^m \|\bmr{U}
    \|_{H^{m+1}(\Omega)}, \quad \pnorm{p - q_h} \leq C h^m \|p
    \|_{H^{m+1}(\Omega)}, 
  \end{displaymath}
  which completes the proof.
\end{proof}

The error estimate of the numerical approximation $\bmr{U}_h$ to the
gradient with solving the minimization problem \eqref{eq_infJp} has
been established. Now let us consider another first-order system
to solve the velocity
$\bm{u}$:
\begin{equation}
  \begin{aligned}
    \nabla \bm{u} - \bmr{U} = \bm{0}, & \quad \text{in } \Omega, \\
    \bm{u} = \bm{g}, & \quad \text{on } \partial \Omega. \\
  \end{aligned}
  \label{eq_firstorderu}
\end{equation}
We define the least squares functional $J_h^{\bmr{u}}(\cdot)$ for
solving \eqref{eq_firstorderu}:
\begin{equation}
  J_h^{\bmr{u}}(\bm{v}_h) := \sum_{K \in \MTh} \| \nabla \bm{v}_h -
  \bmr{U}_h \|_{L^2(K)}^2 + \sum_{e \in \MEh^i} \frac{\mu}{h_e} \|
  \jump{\un \otimes \bm{v}_h} \|_{L^2(e)}^2 + \sum_{e \in \MEh^b}
  \frac{\mu}{h_e} \| \bm{v}_h - \bm{g} \|_{L^2(e)}^2,
  \label{eq_infJu}
\end{equation}
where $\mu$ is a positive parameter.  It is noticeable that in
\eqref{eq_infJu} the first term contains the numerical approximation
$\bmr{U}_h$ since the exact gradient is unavailable to us.  We
minimize the functional $J_h^{\bmr{u}}(\cdot)$ in the piecewise
divergence-free polynomial space $\bmr{S}_h^m$ to seek a numerical
solution. The piecewise divergence-free property provides a local mass
conservation, which is very desirable in solving the incompressible
fluid flow problem \cite{Heys2006mass}. The minimization problem is
given by 
\begin{equation}
  \bm{u}_h = \mathop{\arg \min}_{\bm{v}_h \in \bmr{S}_h^m}
  J_h^{\bmr{u}}(\bm{v}_h).
  \label{eq_functionalJu}
\end{equation}
We also write its Euler-Lagrange equation to solve
\eqref{eq_functionalJu}. Thus, the corresponding discrete variational
problem is defined as to find $\bm{u}_h \in \bmr{S}_h^m$ such that 
\begin{equation}
  a_h^{\bmr{u}}(\bm{u}_h, \bm{v}_h) = l_h^{\bmr{u}}(\bm{v}_h), \quad
  \forall \bm{v}_h \in \bmr{S}_h^m,
  \label{eq_bilinearu}
\end{equation}
where the bilinear form $a_h^{\bmr{u}}(\cdot, \cdot)$ is 
\begin{displaymath}
  a_h^{\bmr{u}}(\bm{u}_h, \bm{v}_h) = \sum_{K \in \MTh} \int_K \nabla
  \bm{u}_h : \nabla \bm{v}_h \d{x} + \sum_{e \in \MEh^i} \int_e
  \frac{\mu}{h_e} \jump{\un \otimes \bm{u}_h} : \jump{\un \otimes
  \bm{v}_h} \d{s} + \sum_{e \in \MEh^b} \int_e \frac{\mu}{h_e}
  \bm{u}_h \cdot \bm{v}_h \d{s},
\end{displaymath}
and the linear form $l_h^{\bmr{u}}(\cdot)$ is 
\begin{displaymath}
  l_h^{\bmr{u}}(\bm{v}_h) = \sum_{K \in \MTh} \int_K \nabla \bm{v}_h :
  \bmr{U}_h \d{x} + \sum_{e \in \MEh^b} \int_e \frac{\mu}{h_e}
  \bm{v}_h \cdot \bm{g} \d{s},
\end{displaymath}
Then we derive the error estimate of the numerical solution to
\eqref{eq_infJu}. We introduce an energy norm $\unorm{\cdot}$ which is
defined as 
\begin{displaymath}
  \unorm{\bm{v}_h}^2 := \sum_{K \in \MTh} \| \nabla \bm{v}_h
  \|_{L^2(K)}^2 + \sum_{e \in \MEh^i} h_e^{-1} \| \jump{\un \otimes
  \bm{v}_h} \|_{L^2(e)}^2 + \sum_{e \in \MEh^b} h_e^{-1} \| \bm{v}_h
  \|_{L^2(e)}^2,
\end{displaymath}
for any $\bm{v}_h \in H^1(\Omega)^d + \bmr{S}_h^m$. We state the
following lemma to give a bound for the norm $\unorm{\cdot}$.
\begin{lemma}
  There exists a positive constant $C$ such that 
  \begin{equation}
    \|\bm{v}_h \|_{L^2(\Omega)} \leq C \unorm{\bm{v}_h},
    \label{eq_unormbound}
  \end{equation}
  for any $\bm{v}_h \in H^1(\Omega)^d  + \bmr{S}_h^m$.
  \label{le_unormbound}
\end{lemma}
\begin{proof}
  We refer to \cite{Brenner2003poincare, arnold1982interior} for the
  proof.
\end{proof}
By the definition of the bilinear form $a_h^{\bmr{u}}(\cdot,
\cdot)$, it is easy to find that for any $\mu > 0$, there exist
constants $C$ such that 
\begin{displaymath}
  \begin{aligned}
    |a_h^{\bmr{u}}(\bm{u}_h, \bm{v}_h)| & \leq C \unorm{\bm{u}_h}
    \unorm{\bm{v}_h}, && \forall \bm{u}_h, \bm{v}_h \in
    \bmr{S}_h^m + H^1(\Omega)^d, \\
    |a_h^{\bmr{u}}(\bm{v}_h, \bm{v}_h)| & \geq C \unorm{\bm{v}_h}^2,
    && \forall \bm{u}_h \in \bmr{S}_h^m,  \\
  \end{aligned}
\end{displaymath}
which implies there exists a unique solution to the problem
\eqref{eq_bilinearu}. Finally, we present the error estimate of the
numerical solution in approximation to the velocity $\bm{u}$.
\begin{theorem}
  Let $\bm{u} \in \bmr{S}^{m+1}(\Omega)$ be the solution to
  \eqref{eq_firstorderStokes} and let $\bm{u}_h \in \bmr{S}_h^m$ be
  the solution to \eqref{eq_bilinearu}, there exists a positive
  constant $C$ such that
  \begin{equation}
    \unorm{\bm{u} - \bm{u}_h} \leq C \left( h^m \|\bm{u}
    \|_{H^{m+1}(\Omega)} + \|\bmr{U} - \bmr{U}_h \|_{L^2(\Omega)}
    \right),
    \label{eq_uerrorestimate}
  \end{equation}
  where $\bmr{U}_h$ is the numerical solution in \eqref{eq_infJp}.
  \label{th_uerrorestimate}
\end{theorem}

\begin{proof}
  Clearly the trace $\jump{\un \otimes \bm{u}} = 0$ on any interior
  faces since $\bm{u}$ is smooth. We let $\bm{v}_h = \wt{\mc{R}}^m
  \bm{u}$ be the interpolant of $\bm{u}$, and we obtain that 
  \begin{displaymath}
    \begin{aligned}
      \unorm{\bm{u} - \bm{u}_h}^2 &=  \sum_{K \in \MTh} \| \nabla
      \bm{u} - \nabla \bm{u}_h \|_{L^2(K)}^2 + \sum_{e \in \MEh^i}
      h_e^{-1} \| \jump{\un \otimes (\bm{u} - \bm{u}_h) }
      \|_{L^2(e)}^2 + \sum_{e \in \MEh^b} h_e^{-1} \| \bm{u} -
      \bm{u}_h \|_{L^2(e)}^2 \\
      & \leq C \left(  \sum_{K \in \MTh} \| \bmr{U}_h - \nabla
      {\bm{u}}_h \|_{L^2(K)}^2 +   \sum_{K \in \MTh} \| \bmr{U} -
      \bmr{U}_h \|_{L^2(K)}^2\right) \\
      & \hspace{60pt} + \sum_{e \in \MEh^i} h_e^{-1} \| \jump{\un
      \otimes (\bm{u} - \bm{u}_h) } \|_{L^2(e)}^2 + \sum_{e \in
      \MEh^b} h_e^{-1} \| \bm{u} - \bm{u}_h \|_{L^2(e)}^2 \\
      & \leq C \left( J_h^{\bmr{u}}(\bm{u}_h) + \|\bmr{U} - \bmr{U}_h
      \|_{L^2(\Omega)}^2 \right)  \leq C \left(
      J_h^{\bmr{u}}(\bm{v}_h) +   \|\bmr{U} - \bmr{U}_h
      \|_{L^2(\Omega)}^2 \right) \\
      & \leq C \left( \unorm{\bm{u} - \bm{v}_h} + \| \bmr{U} -
      \bmr{U}_h \|_{L^2(\Omega)} \right)^2 \\
      & \leq C \left( h^m \| \bm{u} \|_{H^{m+1}(\Omega)} + \| \bmr{U}
      - \bmr{U}_h \|_{L^2(\Omega)} \right)^2. \\
    \end{aligned}
  \end{displaymath}
  The last inequality follows from the approximation property
  \eqref{eq_vectorapproximation} and the trace estimate M2, which
  completes the proof.
\end{proof}

\section{Numerical Results}
\label{sec_numerical_results}
In this section, we present a series of numerical experiments to
demonstrate the accuracy of our method in both two dimensions and
three dimensions. 
%For $d = 2$, the computational domain is taken as
%$\Omega = (0, 1)^2$ and the parameter $\eta$ is selected to be $1$ for
%both bilinear forms $a_h^{\bmr{p}}(\cdot; \cdot)$ and
%$a_h^{\bmr{u}}(\cdot; \cdot)$. 
We take the accuracy order as $1 \leq m \leq 3$ and for different $m$
we list the values $\# S$ that are used in numerical experiments in
Tab.~\ref{tab_numSK}. For all test problems, the Reynolds number $\Re$
and the parameters $\eta$ and $\mu$ are chosen to be $1$. 

\subsection{2D Example} 

\begin{table}
  \centering
  \renewcommand\arraystretch{1.3}
  \begin{tabular}{p{1.0cm}| p{1.5cm}|p{1.2cm}|p{1.2cm}|p{1.2cm} }
    \hline\hline
    \multirow{2}{*}{$d = 2$} &  $m$ & 1 & 2 & 3 \\
    \cline{2-5}
     & $\# S$ & 5 & 10 & 15 \\
    \hline
    \multirow{2}{*}{$d = 3$} &  $m$ & 1 & 2 & 3 \\
    \cline{2-5}
     & $\# S$ & 8 & 18 & 36 \\
     \hline \hline
  \end{tabular}
  \caption{$\# S$ for $1 \leq m \leq 3$.}
  \label{tab_numSK}
\end{table}

\paragraph{\textbf{Example 1}.} We solve the Stokes problem
\eqref{eq_Stokes} on the domain $\Omega = (0, 1)^2$ with the
analytical solution
\begin{displaymath}
  \bm{u}(x, y) = \begin{bmatrix}
    \sin(2 \pi x) \cos(2 \pi y) \\
    -\cos(2 \pi x) \sin(2 \pi y) \\
  \end{bmatrix}, \quad p(x, y) = x^2 + y^2 - \frac{2}{3},
\end{displaymath}
to show the convergence rates of our method. The source term $\bm{f}$
and the Dirichlet data $\bm{g}$ are taken accordingly. We employ a
series of triangular meshes with mesh size $h = 1/10, 1/20, \ldots,
1/160$, see Fig.~\ref{fig_triangulation}. The numerical results are
displayed in Fig.~\ref{fig_ex1qpnorm} and Fig.~\ref{fig_ex1unorm}. For
the first part \eqref{eq_firstorderUp}, we plot the numerical error
under energy norm $\Unorm{\bmr{U} - \bmr{U}_h} + \pnorm{p - p_h}$ in
Fig.~\ref{fig_ex1qpnorm}, which approaches zero at the speed
$O(h^m)$.  The convergence rates are consistent with the theoretical
analysis \eqref{eq_perrorestimate}. For $L^2$ error, we also plot the
errors $\|\bmr{U} - \bmr{U}_h \|_{L^2(\Omega)}$ and $\| p - p_h
\|_{L^2(\Omega)}$ in Fig.~\ref{fig_ex1qpnorm}, and we numerically
detect the odd/even situation. For odd $m$, the $L^2$ errors converge
to zero at the optimal speed, and for even $m$, the errors have a
sub-optimal convergence rate. For the second system
\eqref{eq_firstorderu}, we show the numerical results in
Fig.~\ref{fig_ex1unorm}. Clearly, the error under energy norm
$\unorm{\cdot}$ converges to zero with the rate $O(h^m)$ as the mesh
size approaches $0$. For the $L^2$ norm, we also observe the optimal
rate and sub-optimal rate for odd $m$ and even $m$, respectively. We
note that the convergence orders under all error measurements are in
perfect agreement with our theoretical error estimates.

\begin{figure}[!htp]
  \centering
  \includegraphics[width=0.3\textwidth]{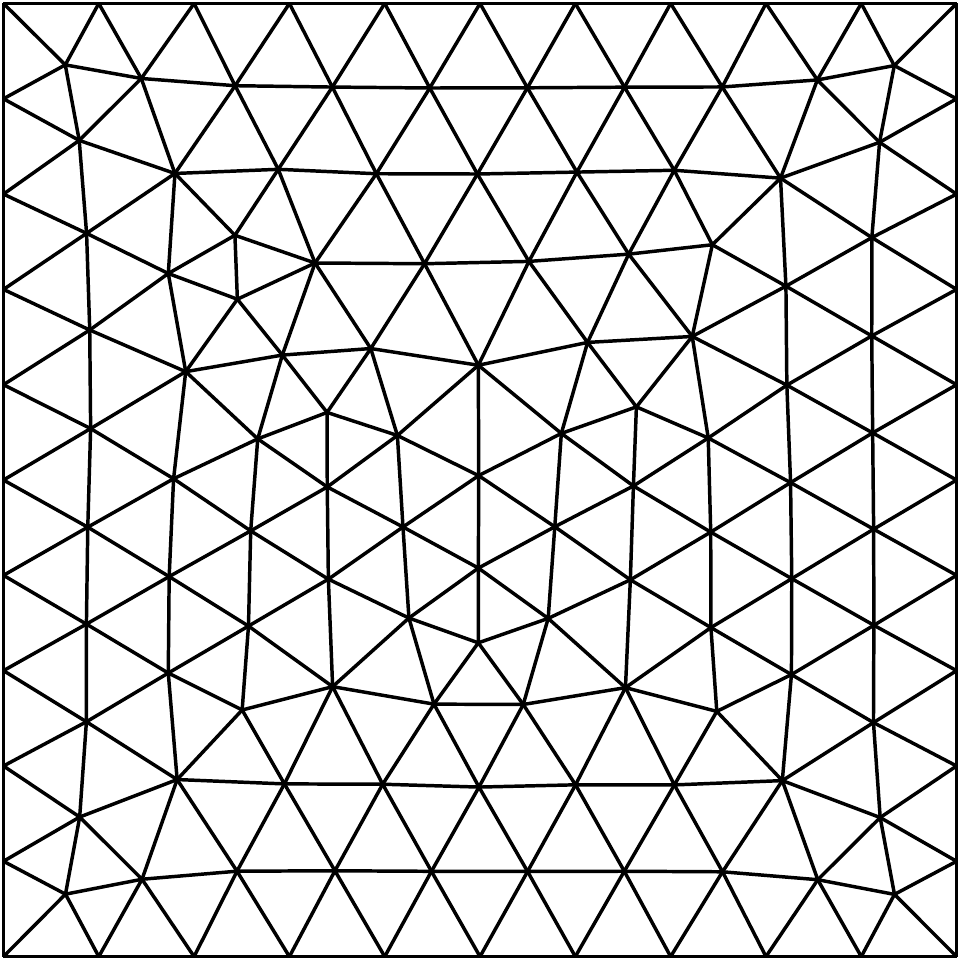}
  \hspace{25pt}
  \includegraphics[width=0.3006\textwidth]{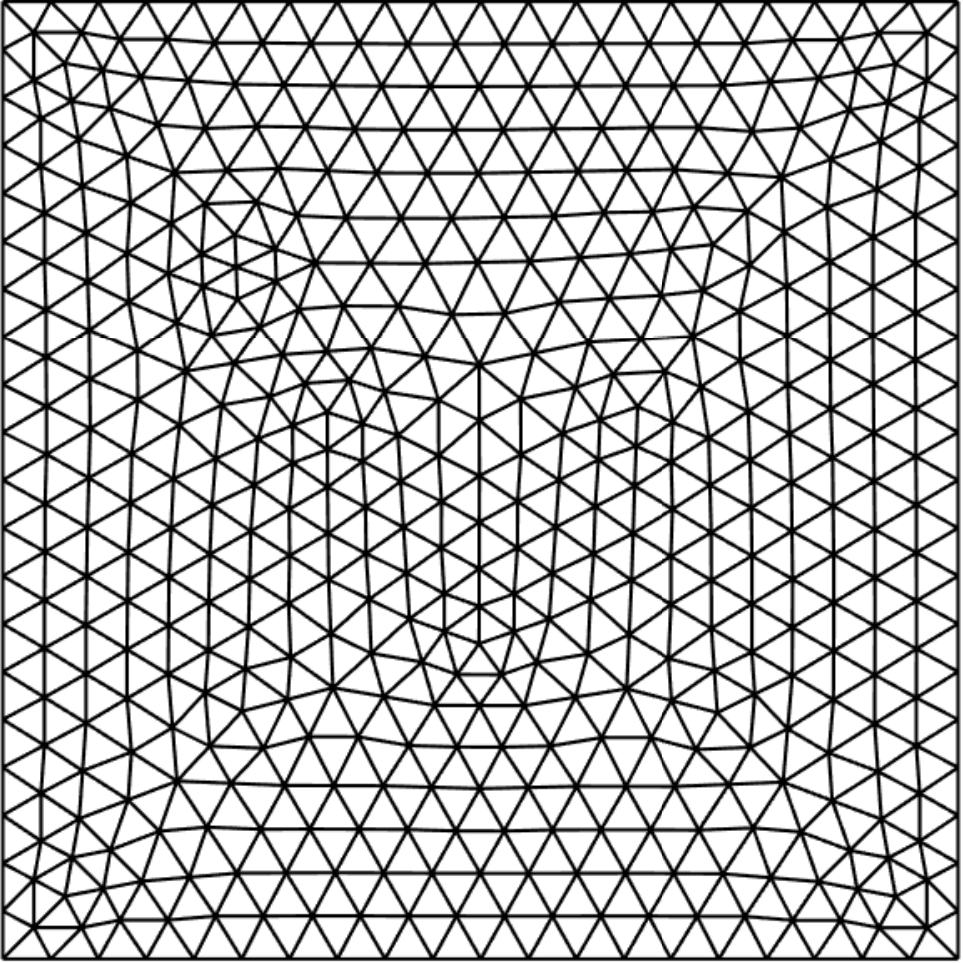}
  \caption{The triangular meshes with mesh size $h = 1/10$ (left) /
  $h = 1/20$ (right) for Example 1.}
  \label{fig_triangulation}
\end{figure}

\begin{figure}[!htp]
  \centering
  \includegraphics[width=0.3\textwidth]{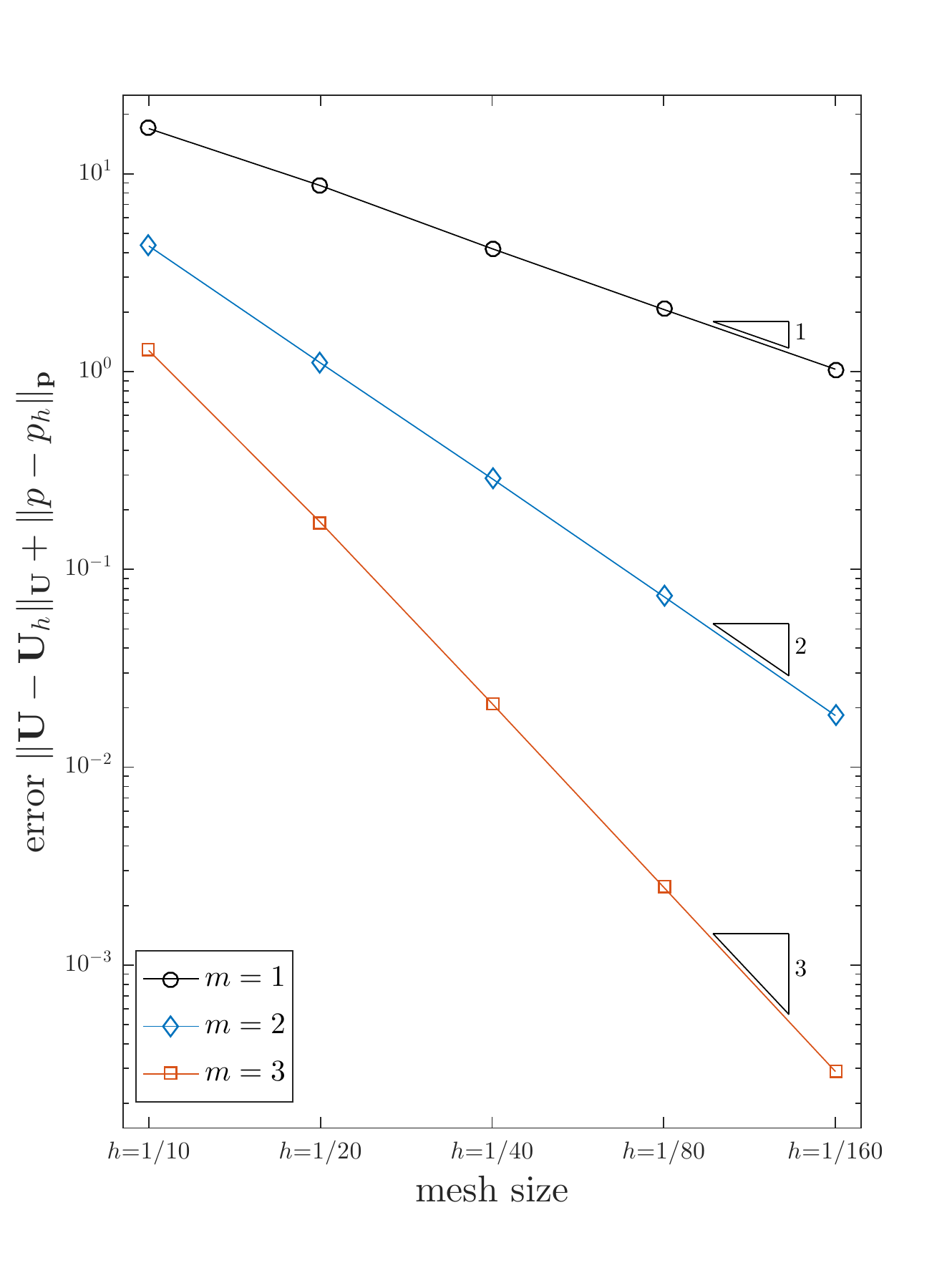}
  \hspace{10pt}
  \includegraphics[width=0.3\textwidth]{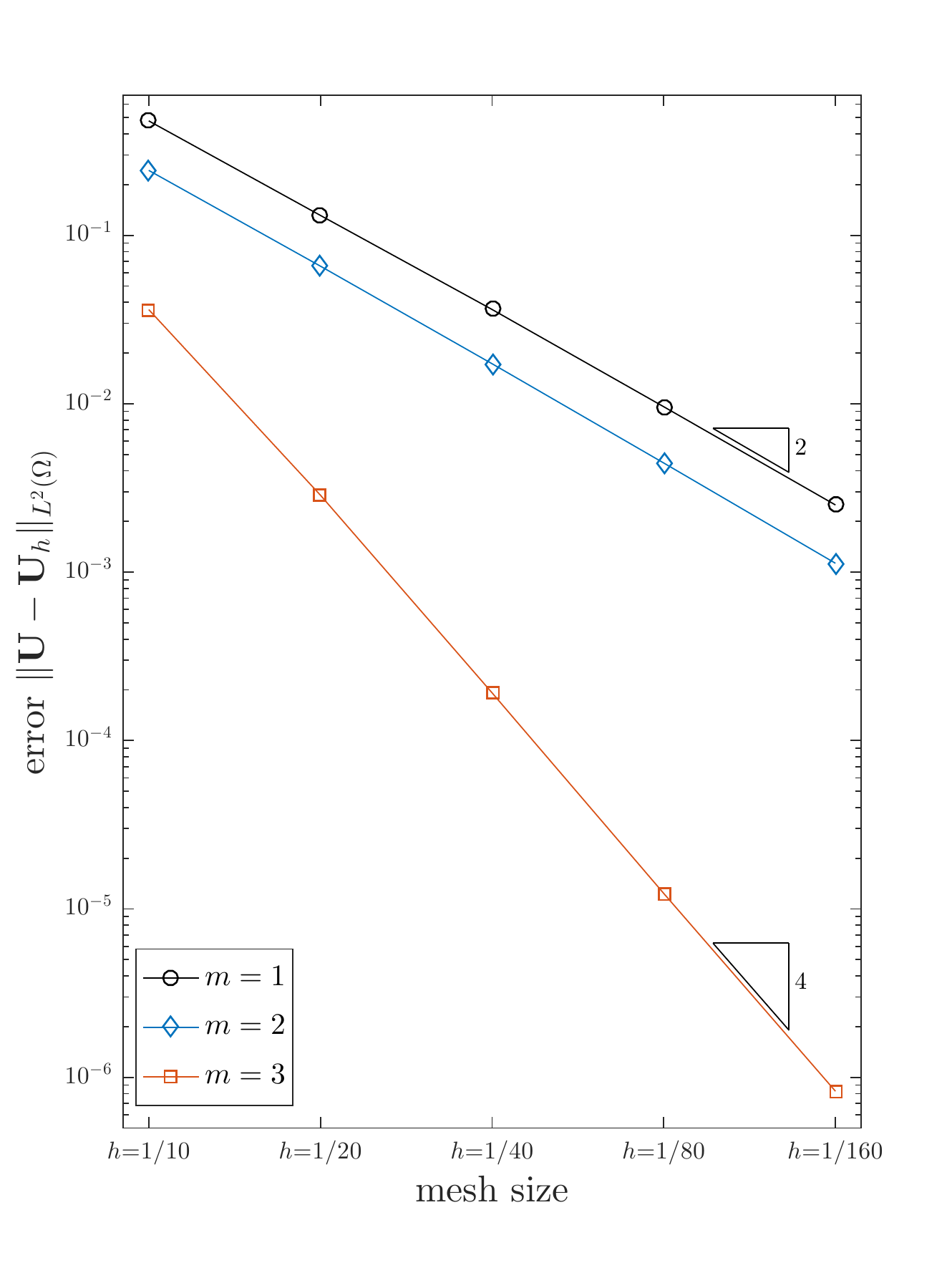}
  \hspace{10pt}
  \includegraphics[width=0.3\textwidth]{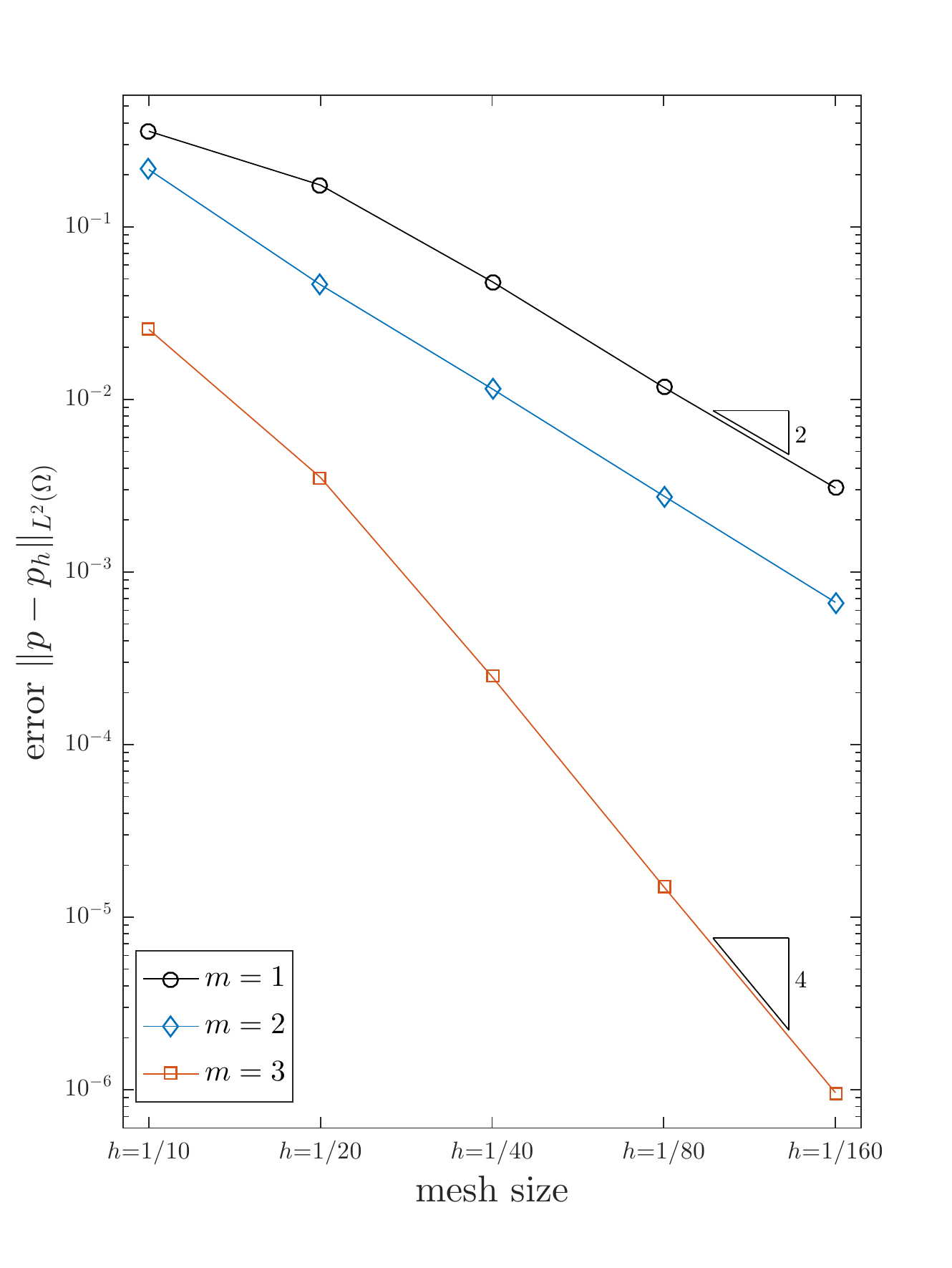}
  \caption{The convergence rates of $\Unorm{\bmr{U} - \bmr{U}_h} +
  \pnorm{p - p_h}$ (left) / $\| \bmr{U} - \bmr{U}_h \|_{L^2(\Omega)}$
  (middle) / $\| p - p_h\|_{L^2(\Omega)}$ (right) for Example 1. }
  \label{fig_ex1qpnorm}
\end{figure}

\begin{figure}[!htp]
  \centering
  \includegraphics[width=0.3\textwidth]{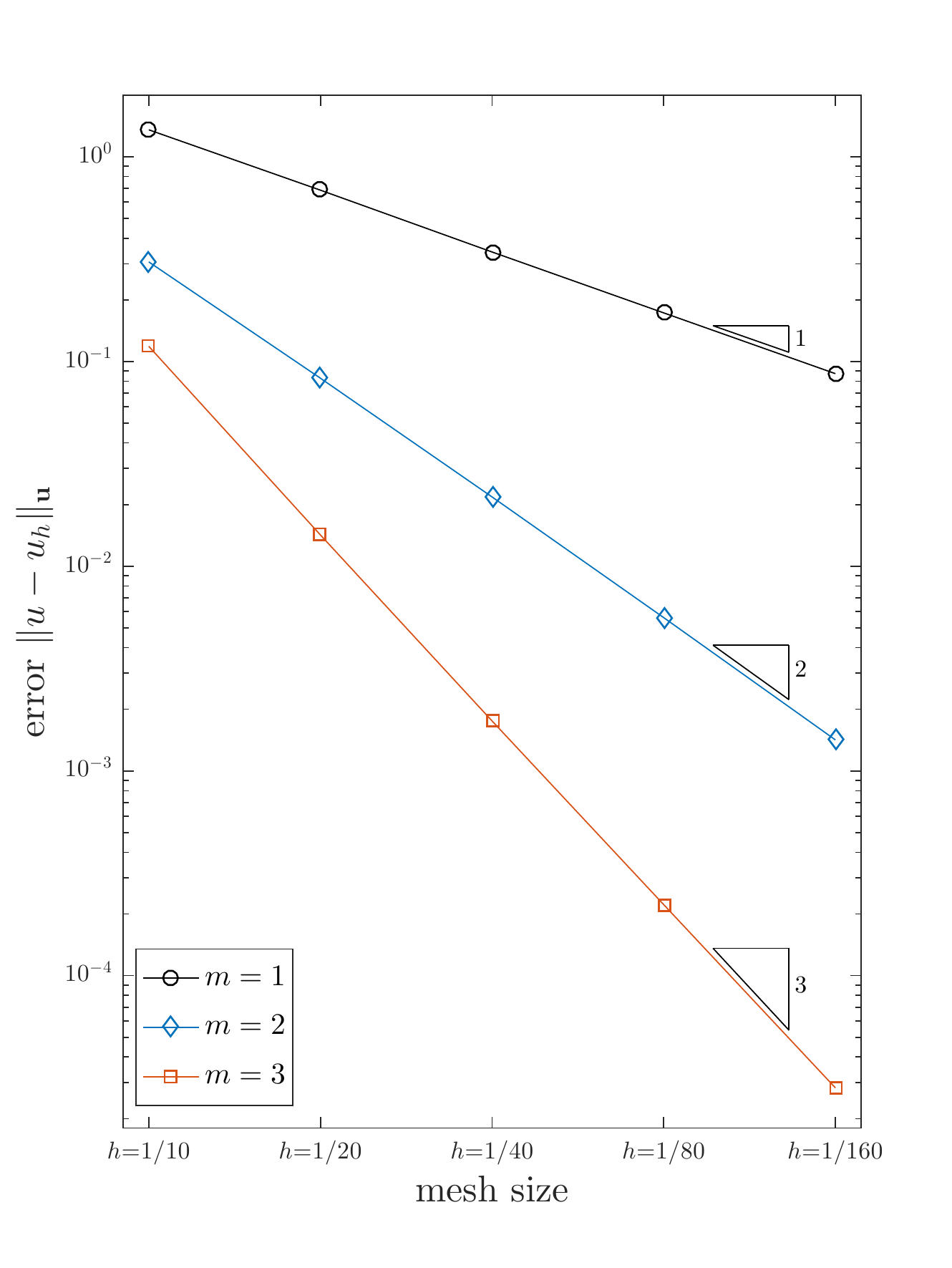}
  \hspace{25pt}
  \includegraphics[width=0.3\textwidth]{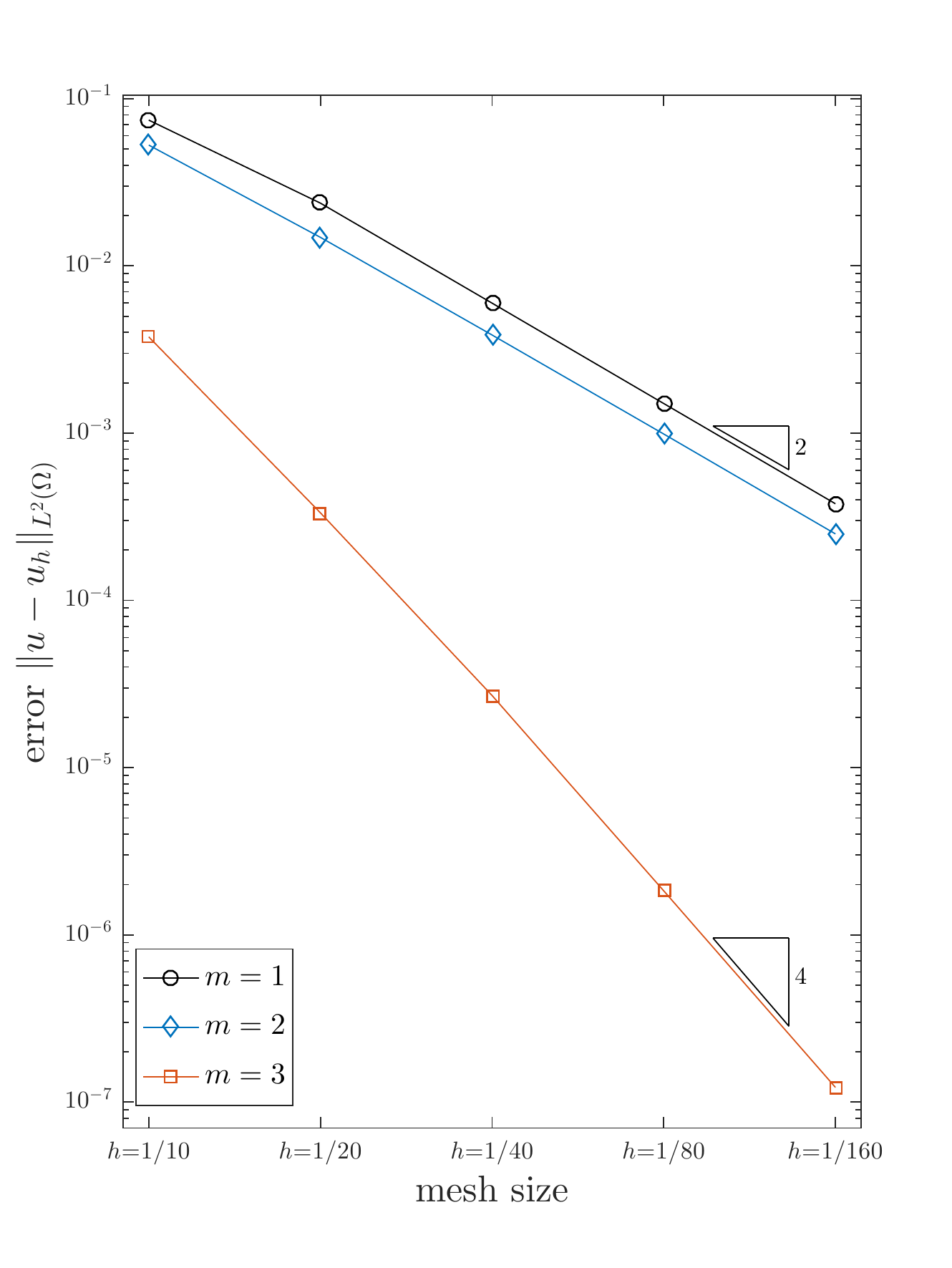}
  \caption{The convergence rates of $\unorm{\bm{u} - \bm{u}_h}$ (left)
  / $\| \bm{u} - \bm{u}_h \|_{L^2(\Omega)}$ (right) for Example 1.}
  \label{fig_ex1unorm}
\end{figure}

\paragraph{\textbf{Example 2}.} Here we solve the Stokes problem to
show the great flexibility of the proposed method. The exact solution
and the computational domain are taken the same as in Example 1 but in
the example we use a series of polygonal meshes with $250$, $1000$,
$4000$, $16000$ elements. These meshes consist of very general
elements and are generated by {\tt
PolyMesher}\cite{talischi2012polymesher}, see
Fig.~\ref{fig_polygonalmesh}. The numerical errors under all error
measurements for both two systems \eqref{eq_firstorderUp} and
\eqref{eq_firstorderu} are shown in Fig.~\ref{fig_ex2qpnorm} and
Fig.~\ref{fig_ex2unorm}, respectively. Again we observe the optimal
convergence orders for all energy norms. For $L^2$ norm, the odd/even
situation is still detected. On such polygonal meshes, the numerically
computed orders agree with our error estimates.

\begin{figure}[!htp]
  \centering
  \includegraphics[width=0.3\textwidth]{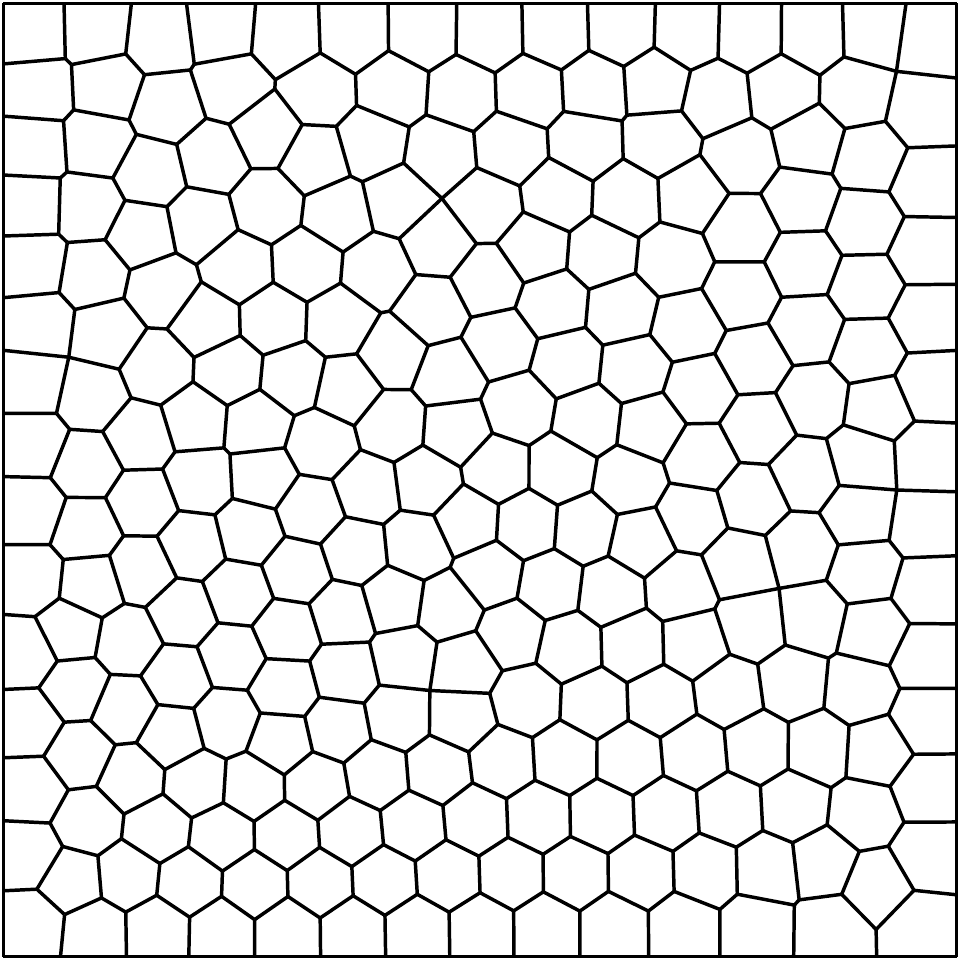}
  \hspace{25pt}
  \includegraphics[width=0.3\textwidth]{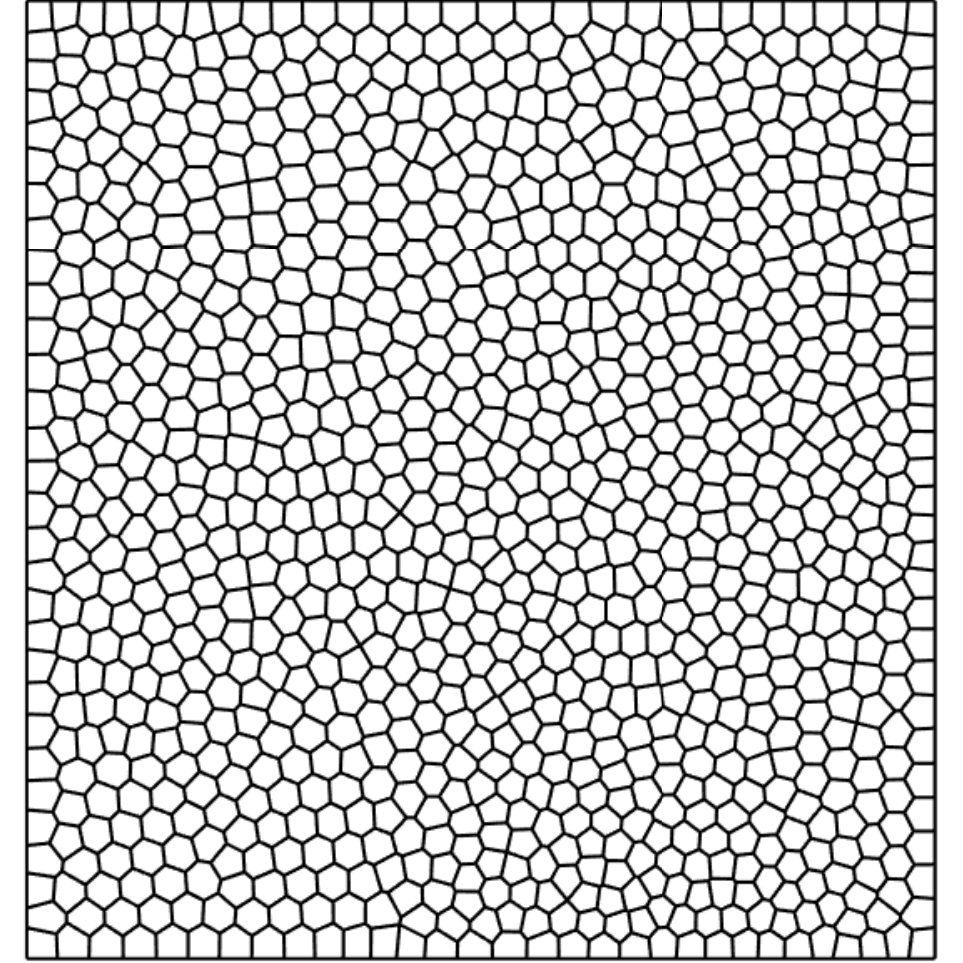}
  \caption{The polygonal meshes with 250 elements (left) /
  1000 elements (right) for Example 2.}
  \label{fig_polygonalmesh}
\end{figure}

\begin{figure}[!htp]
  \centering
  \includegraphics[width=0.3\textwidth]{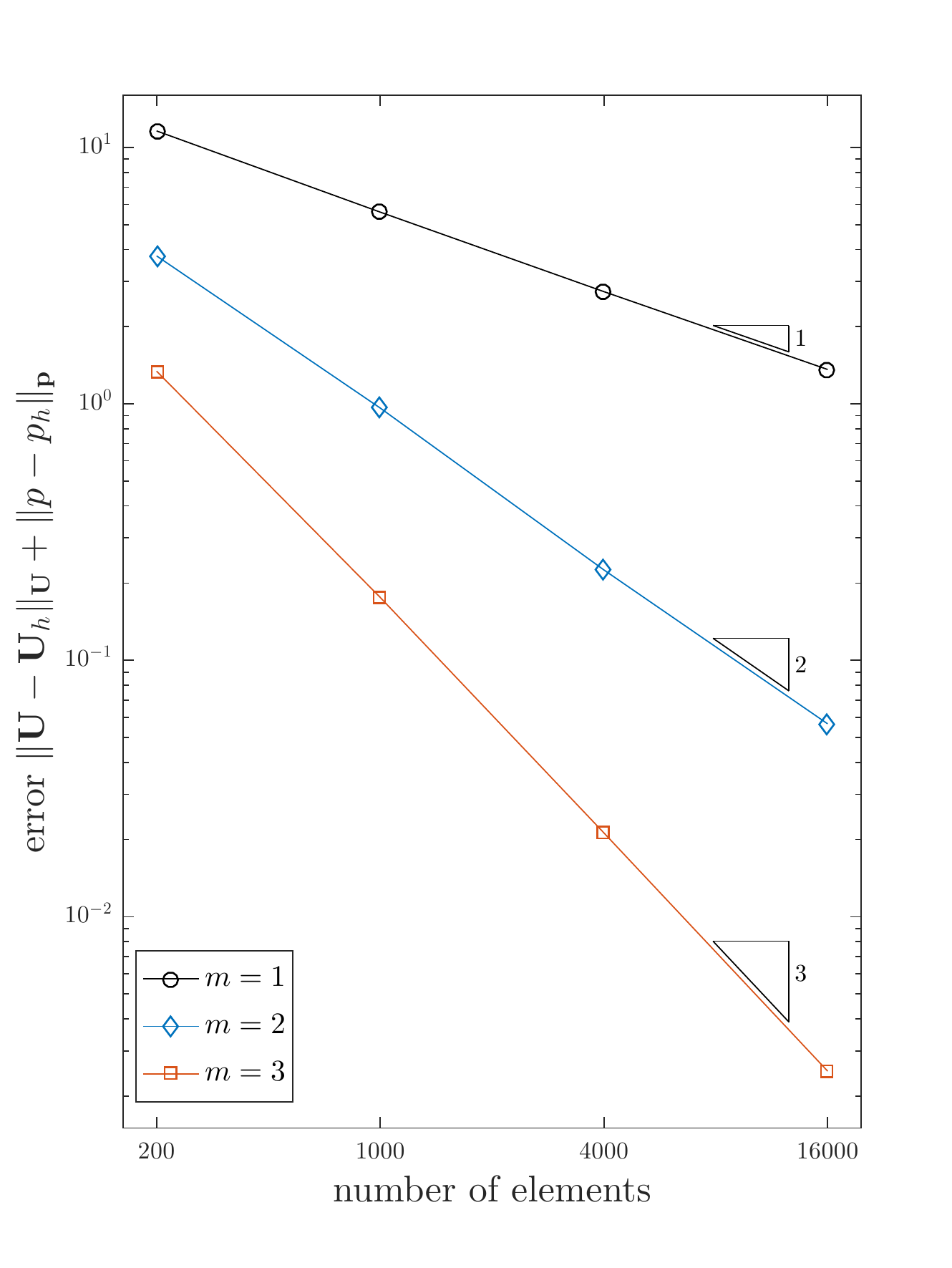}
  \hspace{10pt}
  \includegraphics[width=0.3\textwidth]{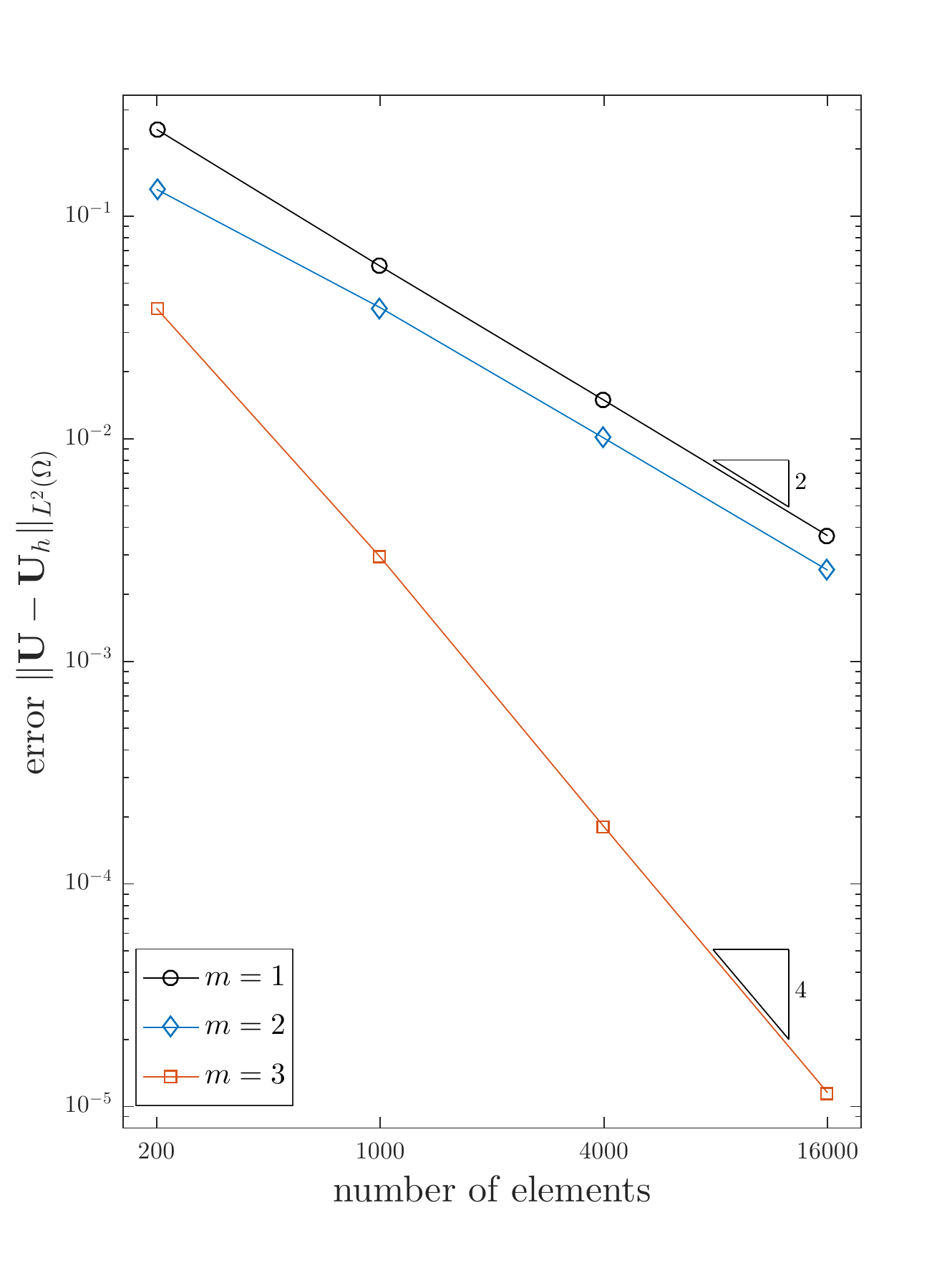}
  \hspace{10pt}
  \includegraphics[width=0.3\textwidth]{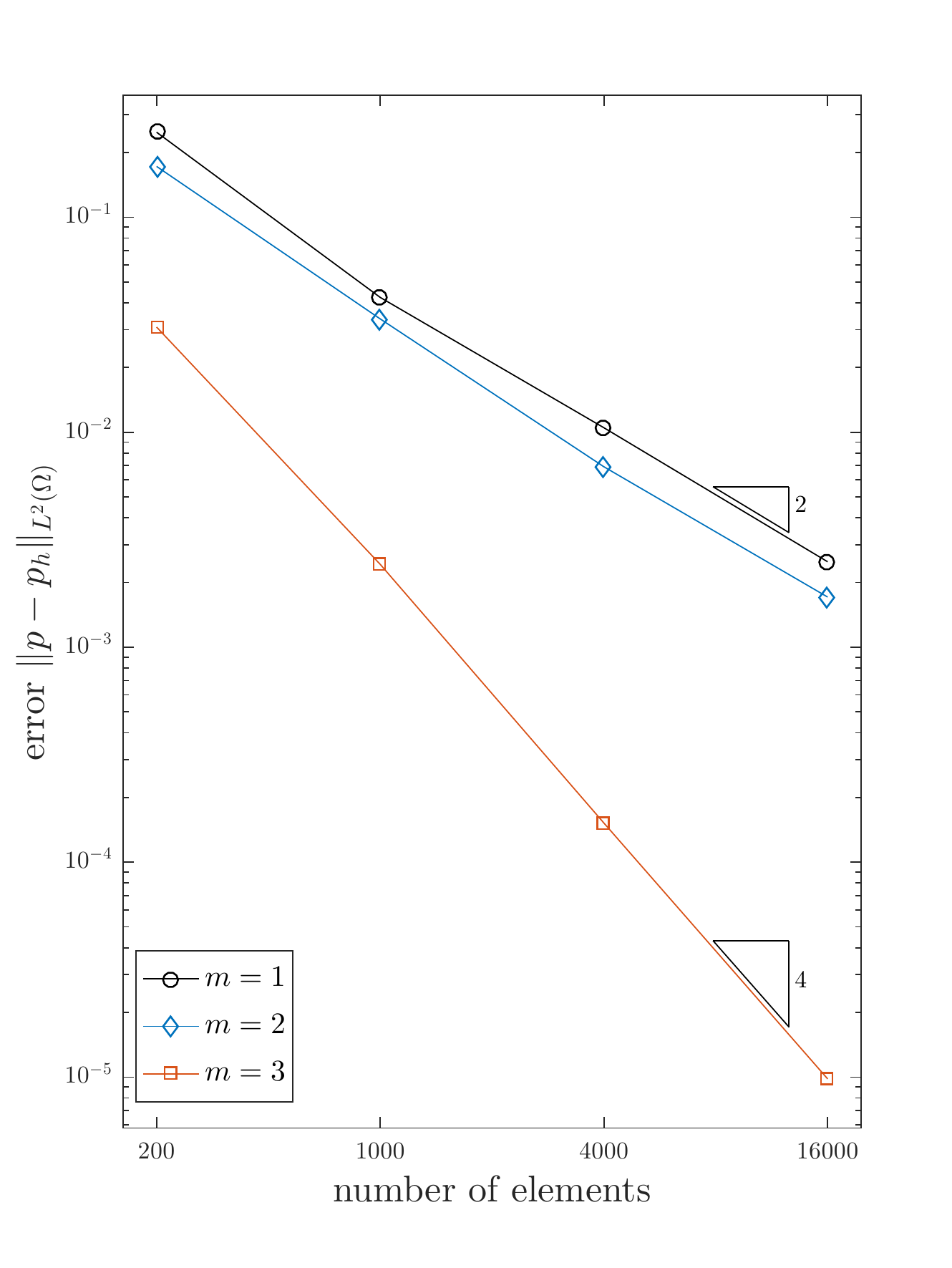}
  \caption{The convergence rates of $\Unorm{\bmr{U} - \bmr{U}_h} +
  \pnorm{p - p_h}$ (left) / $\| \bmr{U} - \bmr{U}_h \|_{L^2(\Omega)}$
  (middle) / $\| p - p_h\|_{L^2(\Omega)}$ (right) for Example 2. }
  \label{fig_ex2qpnorm}
\end{figure}

\begin{figure}[!htp]
  \centering
  \includegraphics[width=0.3\textwidth]{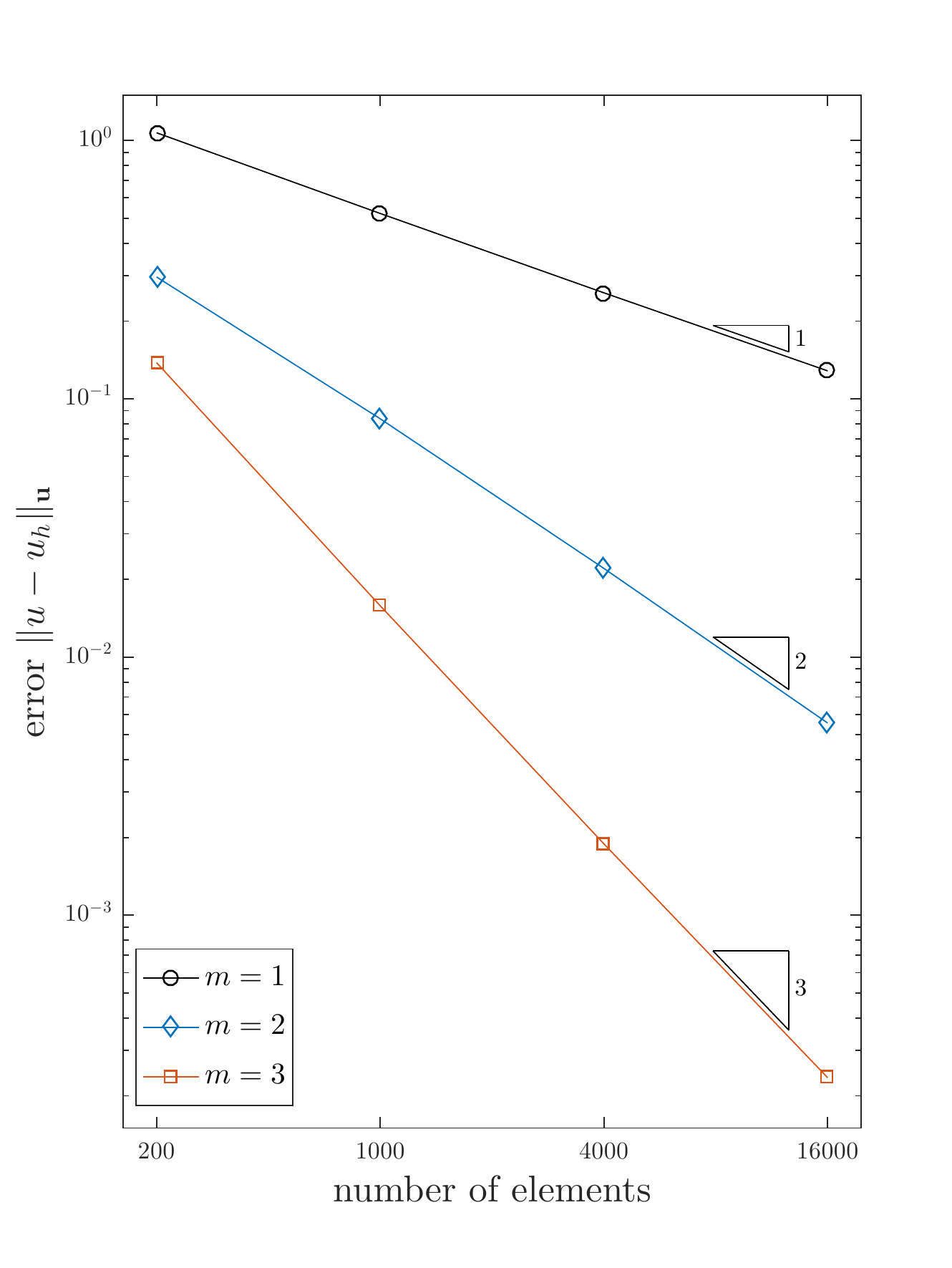}
  \hspace{25pt}
  \includegraphics[width=0.3\textwidth]{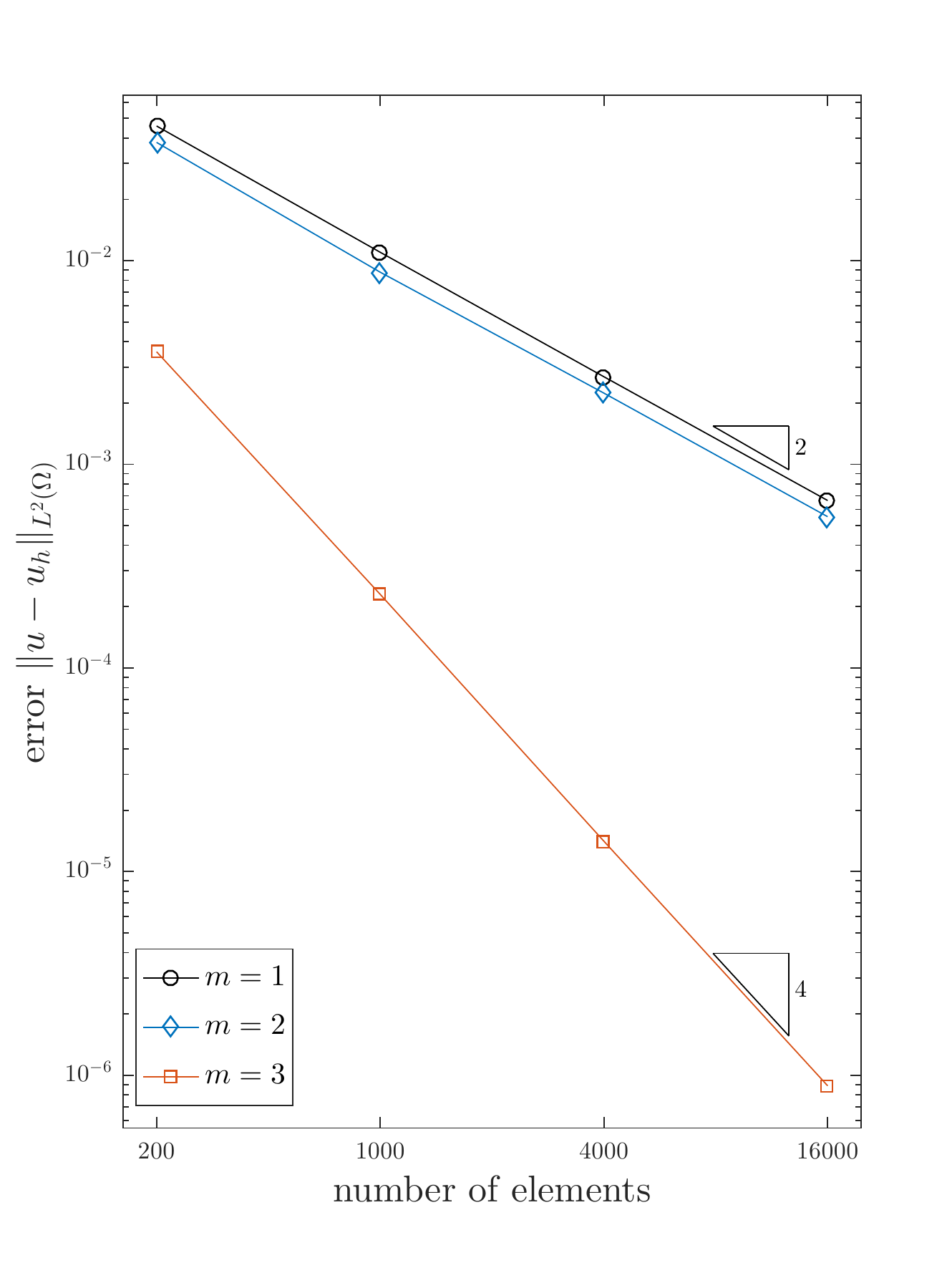}
  \caption{The convergence rates of $\unorm{\bm{u} - \bm{u}_h}$ (left)
  / $\| \bm{u} - \bm{u}_h \|_{L^2(\Omega)}$ (right) for Example 2.}
  \label{fig_ex2unorm}
\end{figure} 

\paragraph{\textbf{Example 3}.} In this example, we test the modified
lid-driven cavity problem \cite{Nister2019efficient} to investigate
the performance of our method dealing with the problem with low
regularities. We consider the unit square domain $\Omega = (0, 1)^2$,
which is subjected to a horizontal flow on the boundary $y = 1$ with
the velocity $\bm{u}(x, y) = (4x(1 - x), 0)^T$. The condition of
remaining boundaries is no-slip boundary condition. The source term
$\bm{f}$ is selected to be $(0, 0)^T$. The velocity field on the upper
boundary involves singularity in the upper right and left corners, but
the restraints are not as strong as for the well-known standard
lid-driven cavity problem \cite{Nister2019efficient}. We solve this
problem on the triangular partition with $h = 1/10$, $h = 1/20$, $h =
1/40$ and $h = 1/80$, see Fig.~\ref{fig_triangulation}. Since the
analytical solution is unknown and we take the numerical solution
which is obtained with the mesh size $h = 1/320$ and the accuracy $m =
3$ as the exact solution. The numerical errors in approximation to the
velocity are presented in Fig.~\ref{fig_ex3unorm}. The convergence
rates under energy norms and $L^2$ norms are detected to be $O(h)$ for
all accuracy $1 \leq m \leq 3$. A possible explanation of such
convergence rates can be traced back to the low regularity of this
problem. Figure \ref{fig_m2liddriven} and Figure \ref{fig_m3liddriven}
present the numerical results obtained on the mesh level $h = 1/80$
with the accuracy $m=2$ and $m=3$, respectively.  We can observe the
main vortex in the center of the domain and two small vortices in the
bottom left and right corners.

\begin{figure}[!htp]
  \centering
  \includegraphics[width=0.3\textwidth]{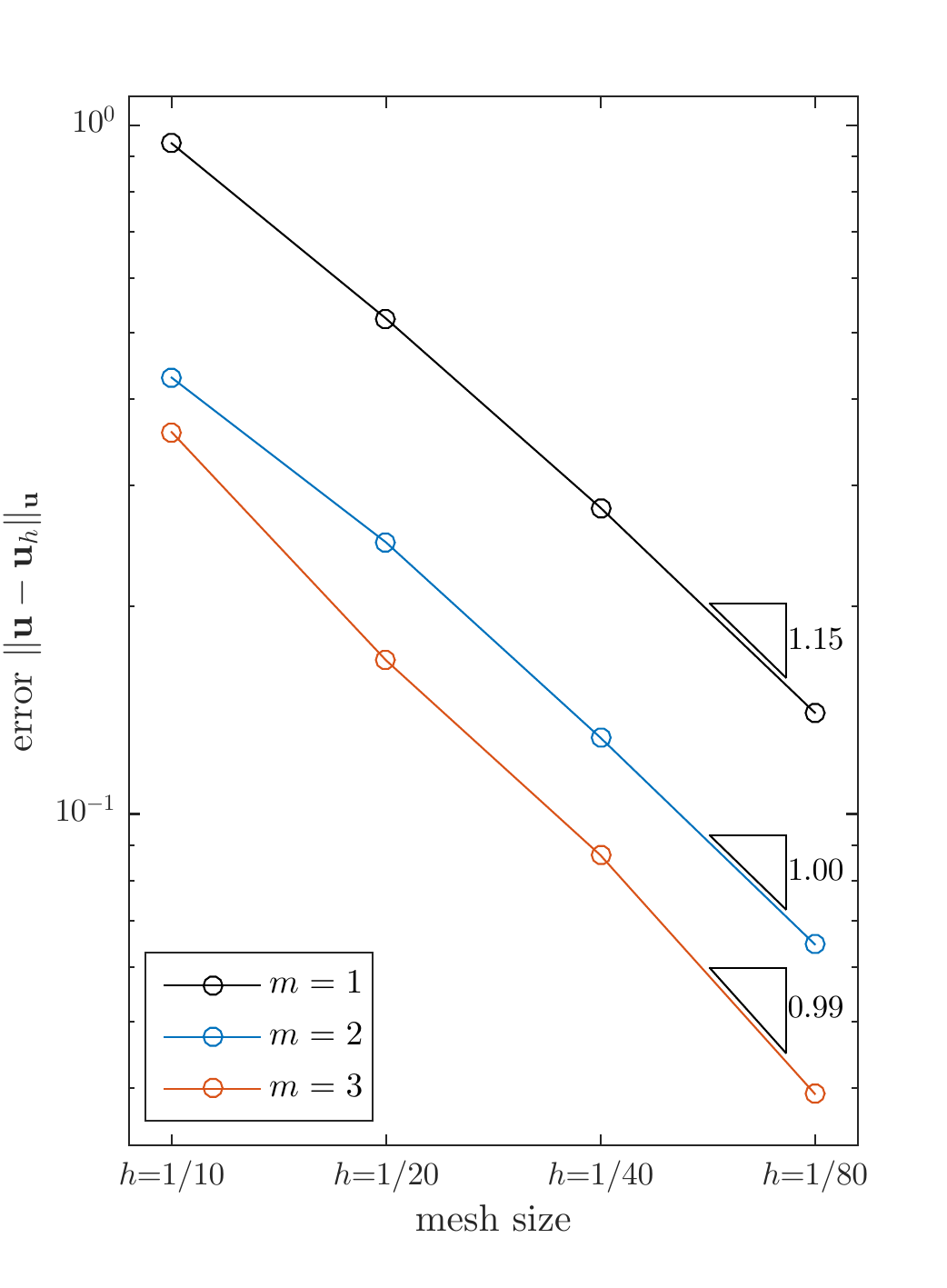}
  \hspace{25pt}
  \includegraphics[width=0.3\textwidth]{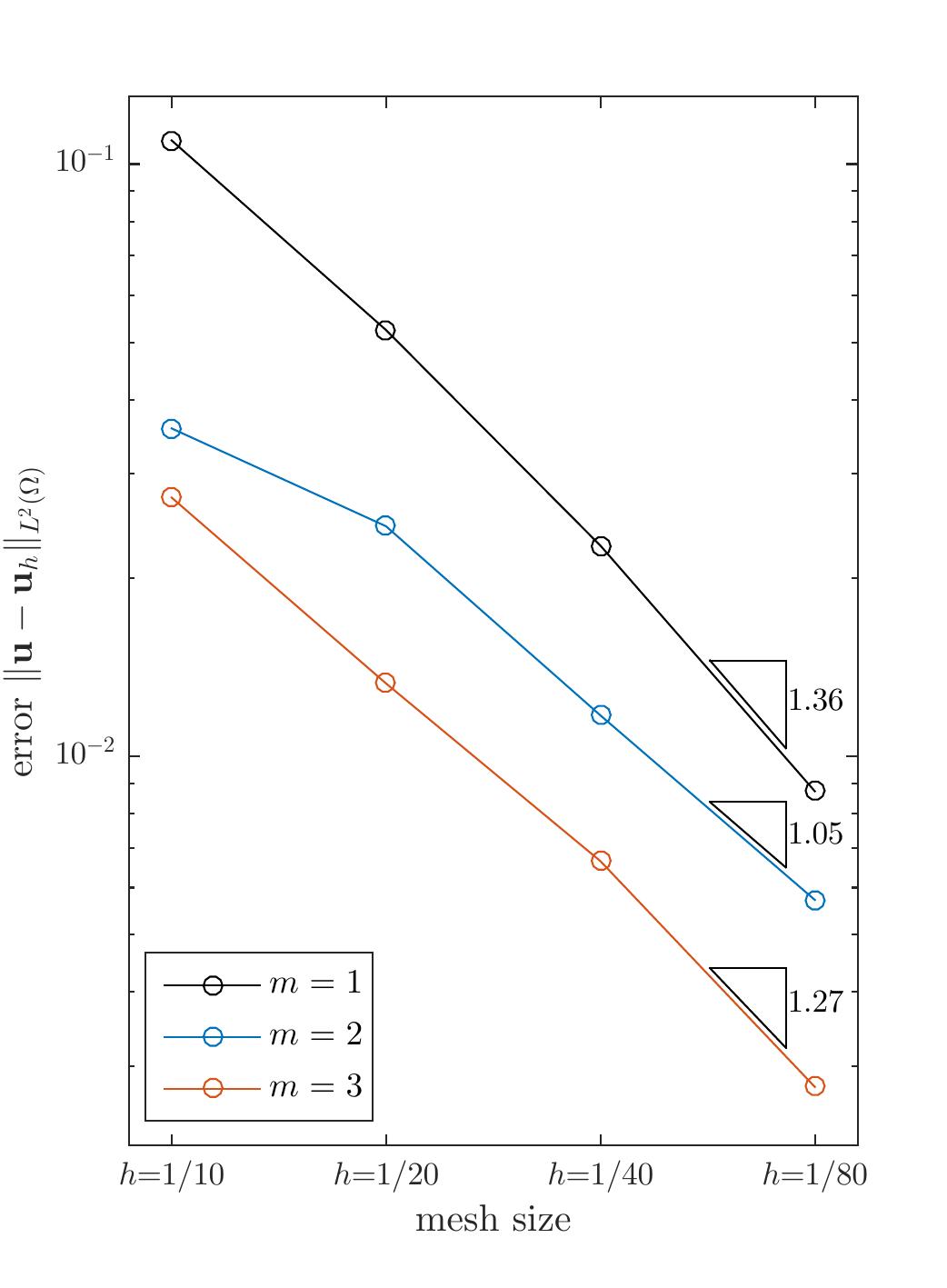}
  \caption{The convergence rates of $\unorm{\bm{u} - \bm{u}_h}$ (left)
  / $\| \bm{u} - \bm{u}_h \|_{L^2(\Omega)}$ (right) for Example 3.}
  \label{fig_ex3unorm}
\end{figure} 

\begin{figure}
  \centering
  \includegraphics[width=0.48\textwidth]{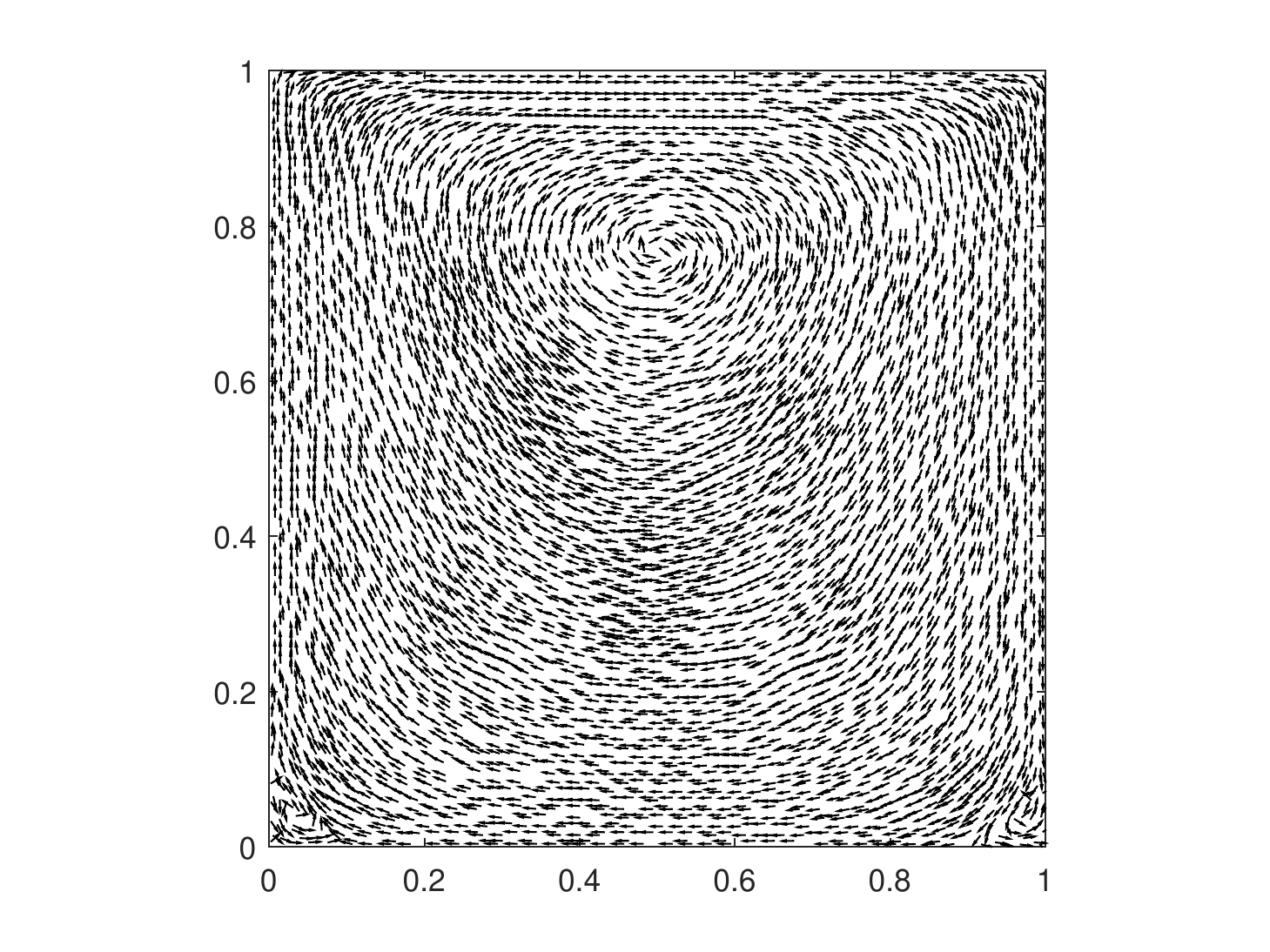}
  \includegraphics[width=0.48\textwidth]{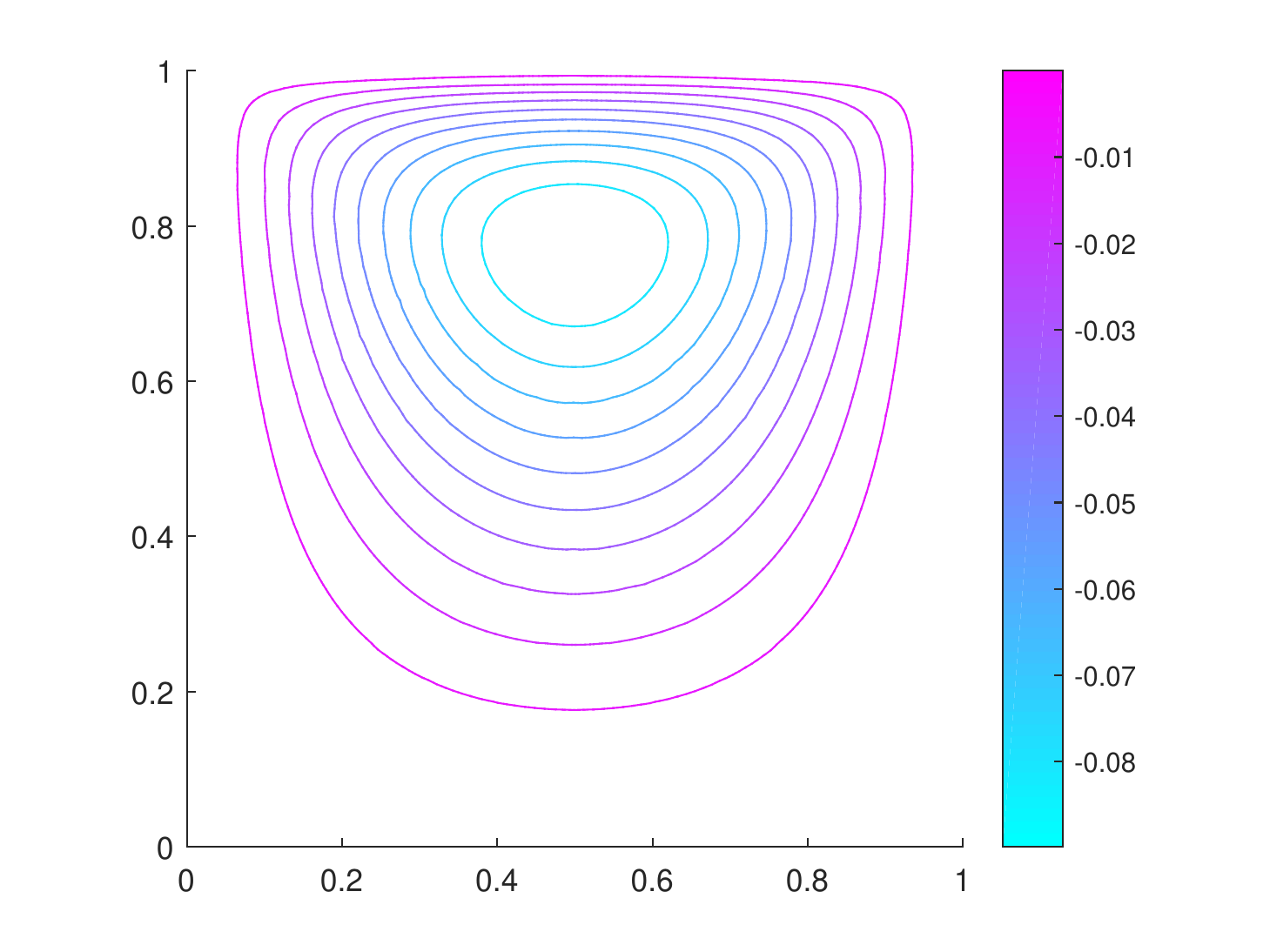}
  \caption{Velocity vectors (left) and the streamline of the
    flow (right) with the accuracy $m = 2$ for Example 3.}
  \label{fig_m2liddriven}
\end{figure}

\begin{figure}
  \centering
  \includegraphics[width=0.48\textwidth]{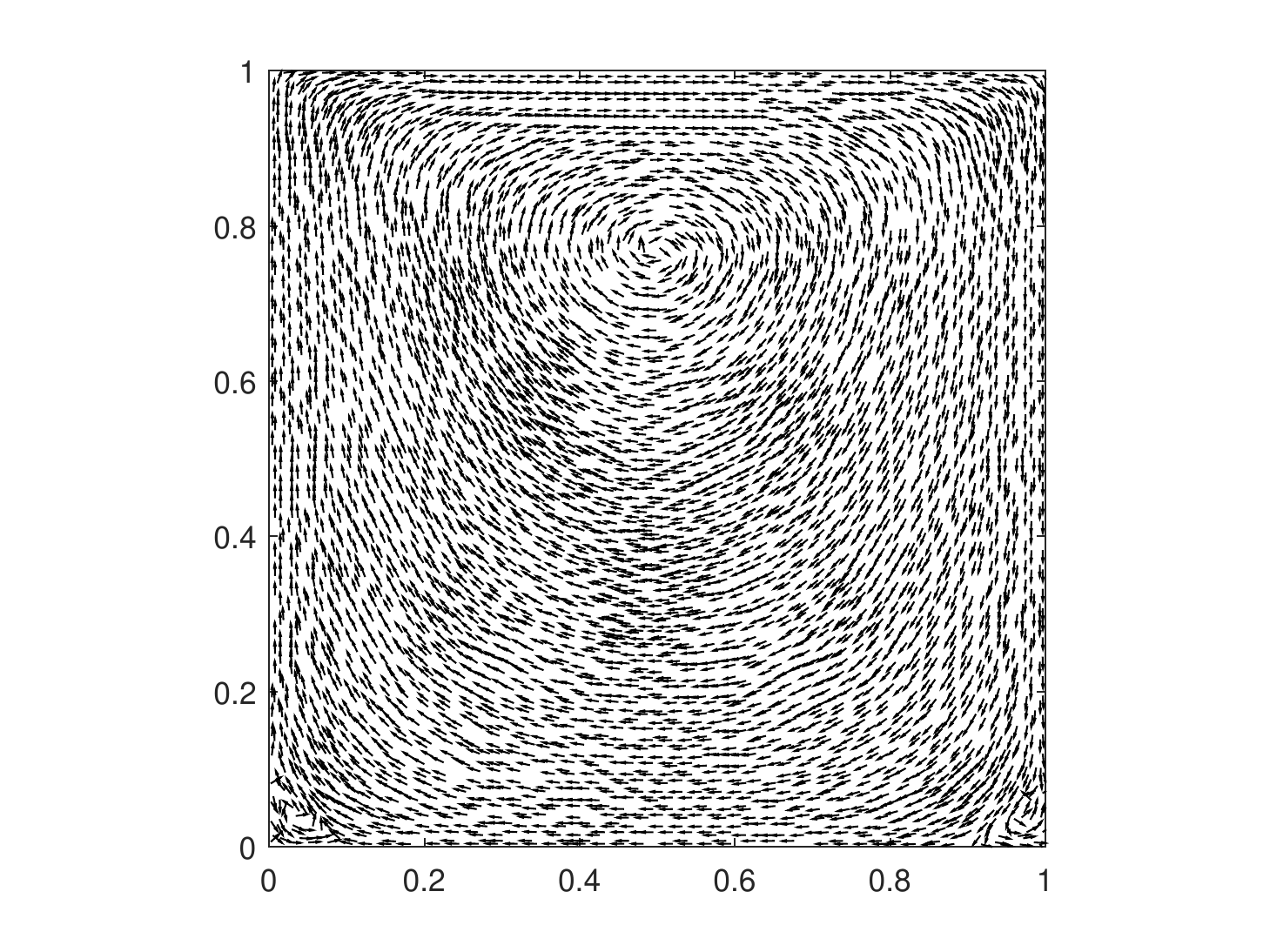}
  \includegraphics[width=0.48\textwidth]{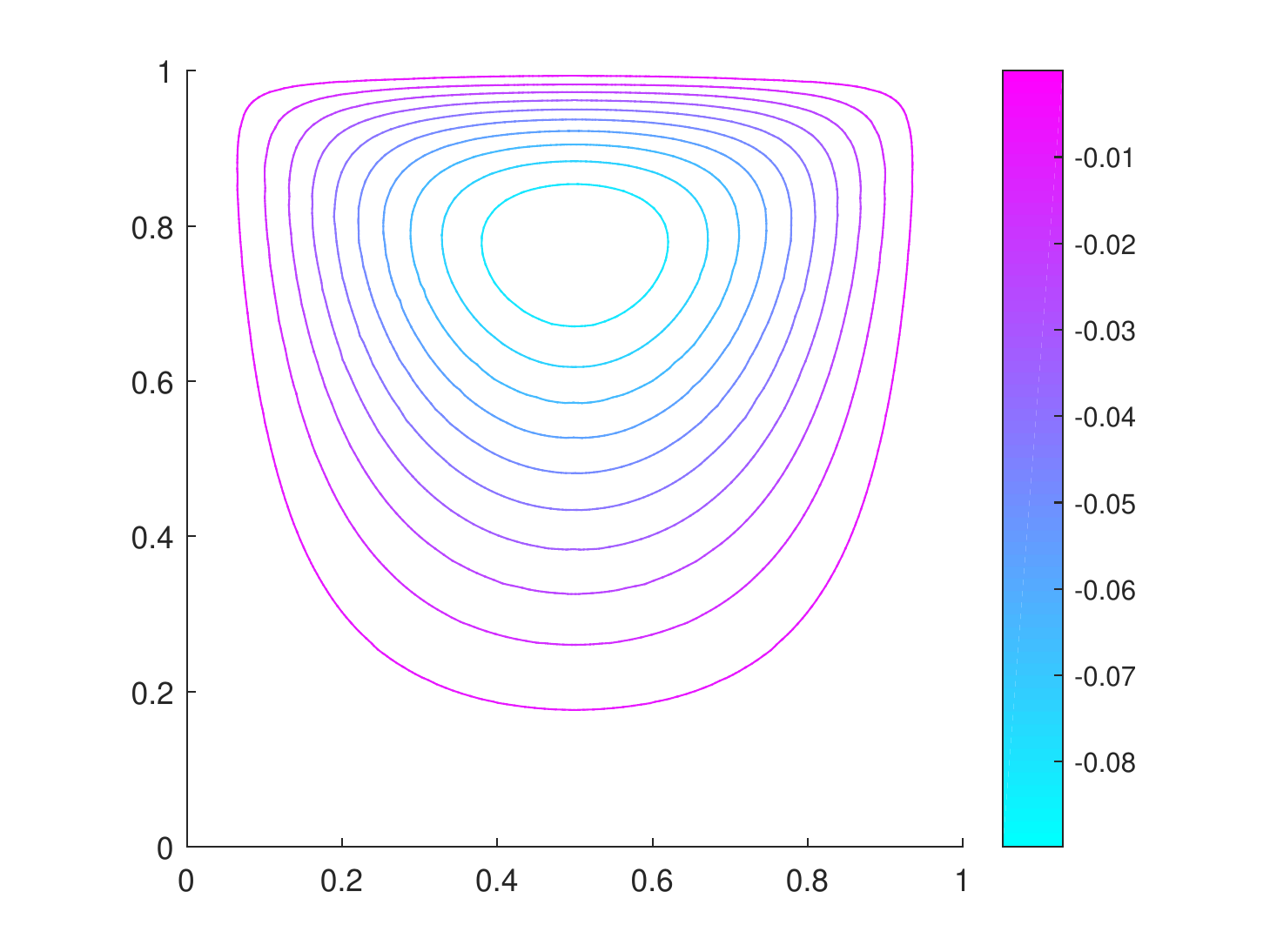}
  \caption{Velocity vectors (left) and the streamline of the
    flow (right) with the accuracy $m = 3$ for Example 3.}
  \label{fig_m3liddriven}
\end{figure}

\subsection{3D Example} 

\paragraph{\textbf{Example 4}.} In this example, we consider the
Stokes problem in three dimensions. We solve the problem on the domain
$\Omega = (0, 1)^3$ and we take a series of tetrahedral meshes with
the resolution $h = 1/4$, $h = 1/8$ and $h = 1/16$ (see
Fig.~\ref{fig_3dmesh}). We choose the analytical solution $\bm{u}$ and
$p$ as 
\begin{displaymath}
  \bm{u}(x, y, z) = \begin{bmatrix}
    1 - e^{x} \cos(2 \pi y) \\
    \frac{1}{2\pi} e^{x} \sin(2 \pi y) \\ 
    0 \\
  \end{bmatrix}, \quad p(x, y, z) = x^2 + y^2 - \frac{2}{3}.
\end{displaymath}
The numerical errors in approximation to the gradient and the pressure
are collected in Tab.~\ref{tab_ex4Up}, and the numerical errors in
solving the second first-order system are gathered in
Tab.~\ref{tab_ex4u}. We also depict the velocity field and the contour
of $| \bm{u}_h |$ in Fig.~\ref{fig_ex4slice} and the numerical
solution in this figure is obtained on the mesh level $h = 1/16$ with
the accuracy $m = 3$. Here, we still observe the odd/even situation.
For odd $m$, the errors under $L^2$ norm seem to converge to zero
optimally as the mesh size tends to zero. For even $m$, the
convergence rates for all variables under $L^2$ norm are numerically
detected to be sub-optimal.  Again we note that all computed
convergence orders are consistent with theoretical analysis.

\begin{figure}[!htp]
  \centering
  \includegraphics[width=0.3\textwidth]{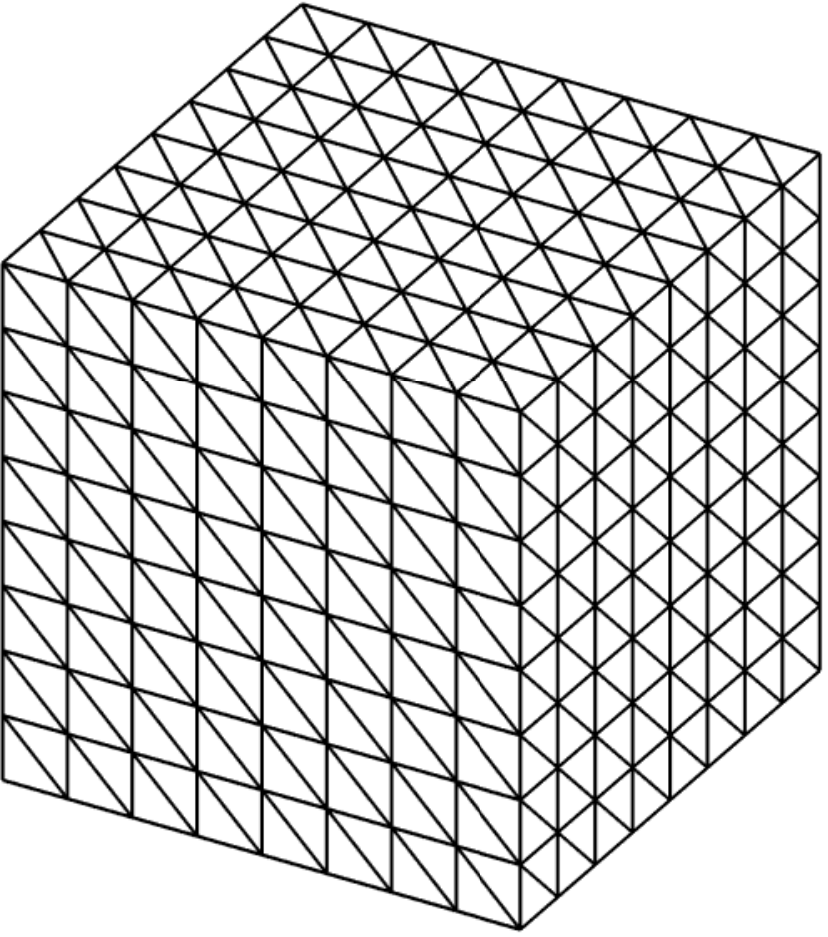}
  \hspace{25pt}
  \includegraphics[width=0.3\textwidth]{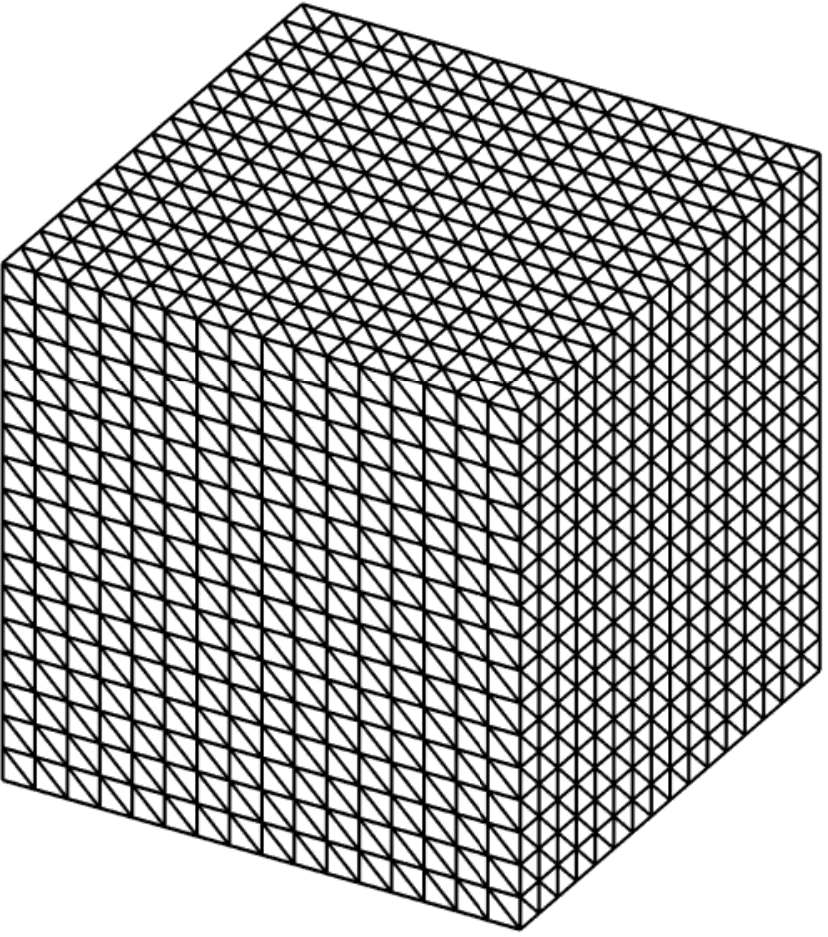}
  \caption{The tetrahedral meshes with mesh size $h = 1/8$ (left) /
  mesh size $h = 1/16$ (right) for three-dimensional examples.}
  \label{fig_3dmesh}
\end{figure}

\begin{figure}[htp]
  \centering
  \includegraphics[height=1.8in]{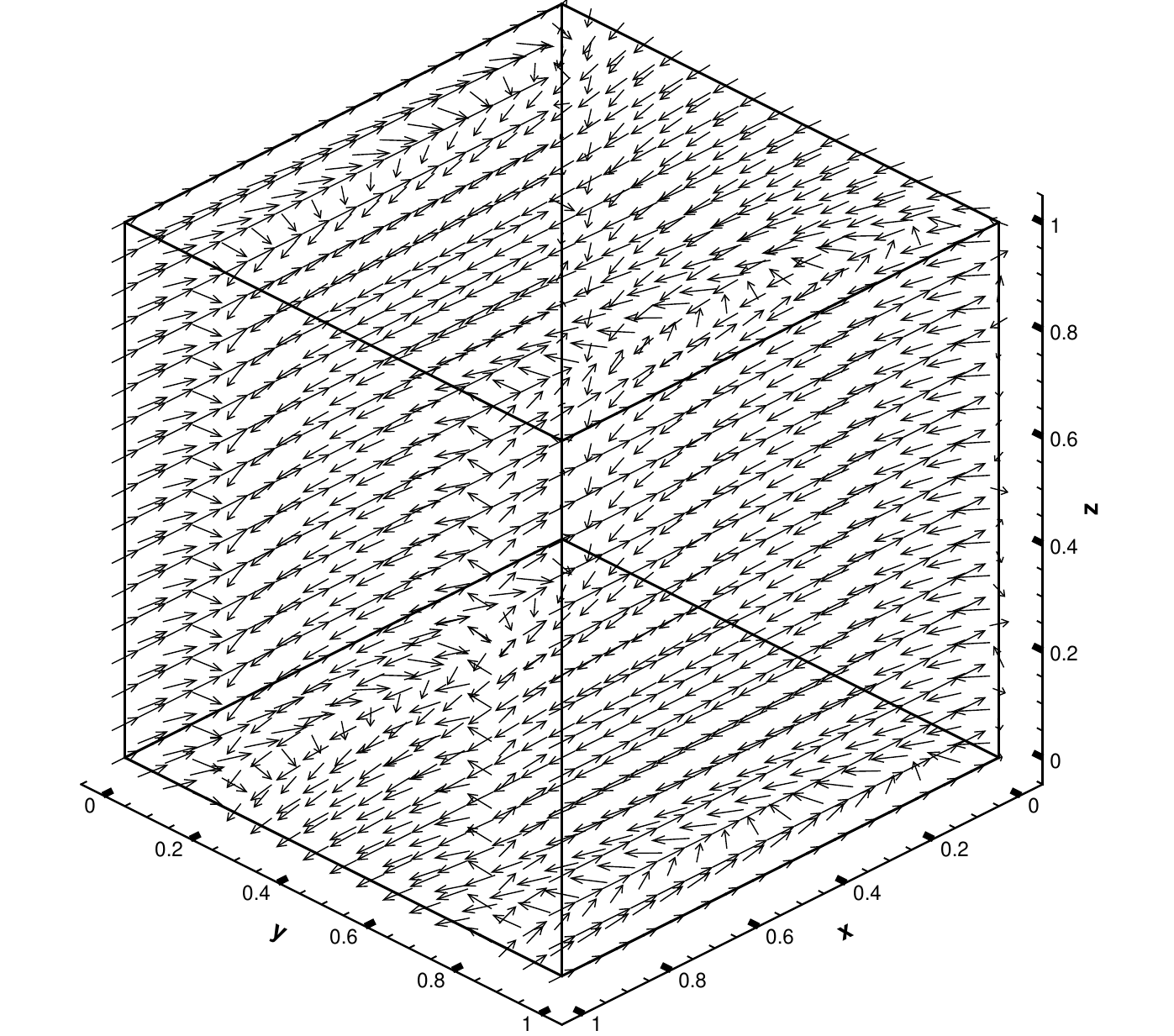}
  \hspace{30pt}
  \includegraphics[height=1.8in]{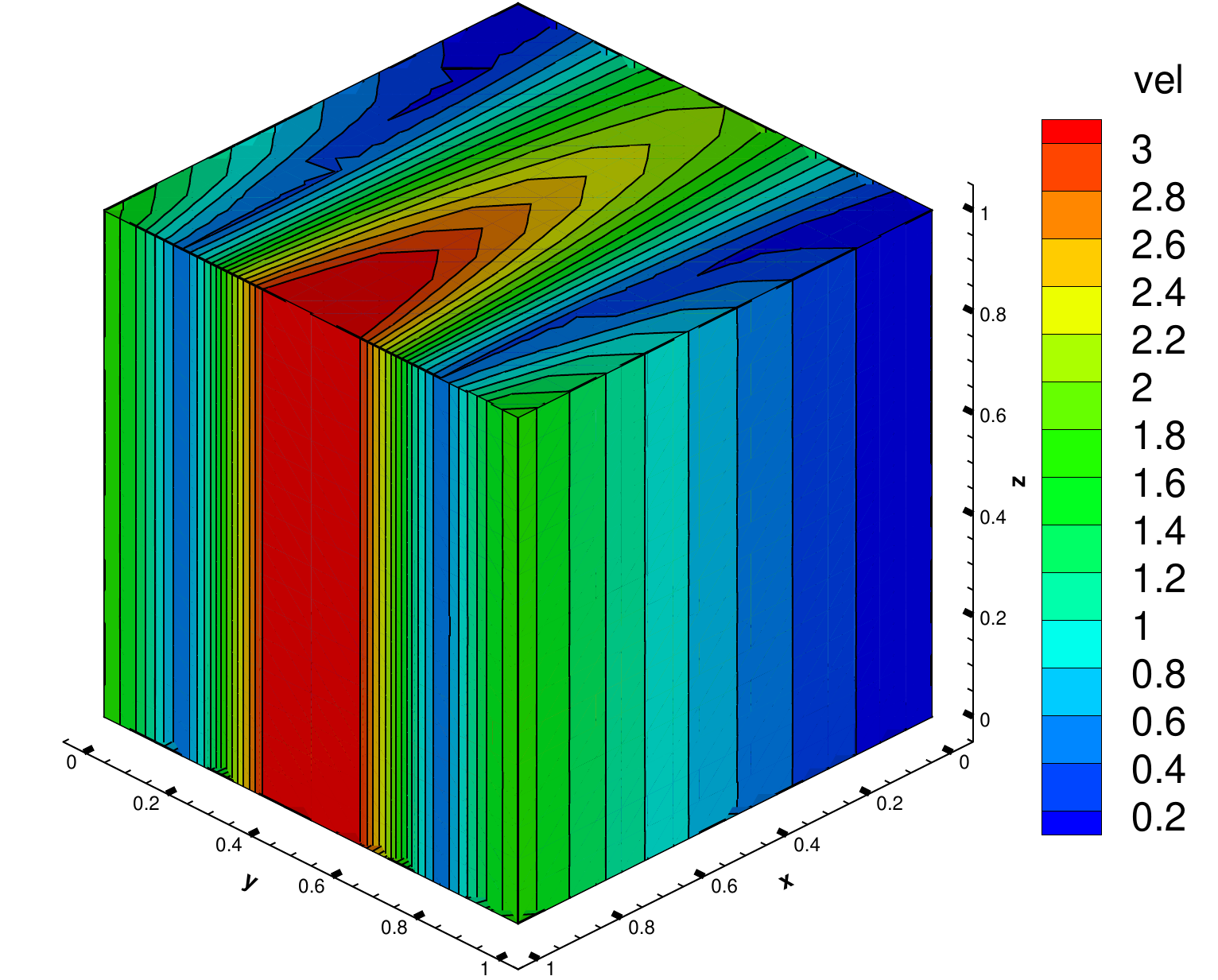}
  \caption{The velocity field (left) / the contour of $ | \bm{u}_h|$
  (right) for Example 4.}
  \label{fig_ex4slice}
\end{figure}

\begin{table}
  \centering
  \renewcommand\arraystretch{1.2}
  \scalebox{0.9}{
  \begin{tabular}{p{0.3cm} | c | c | c | c  |
     c |c | c}
    \hline\hline
    $m$ & $h$ &  $\Unorm{\bmr{U} - \bmr{U}_h} + \pnorm{p - p_h} $ &
    order & $\| \bmr{U} - \bmr{U}_h \|_{L^2(\Omega)}$ & order & $ \| p
    - p_h \|_{L^2(\Omega)}$ & order \\
    \hline
    \multirow{3}{*}{$1$} 
    & $1/4$ & 4.162e+1  & -  & 2.022e-0 & - & 3.421e-0 & - \\
    \cline{2-8}
    & $1/8$ & 2.446e+1 & 0.77  & 9.998e-1 & 1.02 & 1.743e-0 & 0.97 \\
    \cline{2-8}
    & $1/16$ & 1.206e+1 & 1.02  & 3.653e-1 & 1.46 & 7.143e-1 & 1.29 \\
    \hline
    \multirow{3}{*}{$2$} 
    & $1/4$ & 1.476e+1  & -  & 6.388e-1 & - & 1.053e-0 & - \\
    \cline{2-8}
    & $1/8$ & 4.213e-0 & 1.81  & 1.284e-1 & 2.31 & 3.425e-1 & 1.62 \\
    \cline{2-8}
    & $1/16$ & 1.167e-0 & 1.92  & 3.128e-2 & 2.03 & 7.752e-2 & 2.13 \\
    \hline
    \multirow{3}{*}{$3$} 
    & $1/4$ & 4.125e-0  & -  & 1.557e-1 & - & 1.409e-1 & - \\
    \cline{2-8}
    & $1/8$ & 4.913e-1 & 3.06  & 1.125e-2 & 3.79 & 1.010e-2 & 3.80 \\
    \cline{2-8}
    & $1/16$ & 5.431e-2 & 3.13  & 7.507e-4 & 3.91 & 6.931e-4 & 3.86 \\
    \hline\hline
  \end{tabular}}
  \caption{The numerical results of  $\Unorm{\bmr{U} - \bmr{U}_h} +
  \pnorm{p - p_h} $, $\| \bmr{U} - \bmr{U}_h
  \|_{L^2(\Omega)}$ and  $ \| p - p_h \|_{L^2(\Omega)}$ for
  Example 4.}
  \label{tab_ex4Up}
\end{table}

\begin{table}
  \centering
  \renewcommand\arraystretch{1.2}
  \begin{tabular}{ p{0.3cm} | p{0.8cm} | c | c|c  |
    c }
    \hline\hline 
    $m$ & $h$ &  $\unorm{\bm{u} - \bm{u}_h} $ & order & $\|\bm{u} -
    \bm{u}_h \|_{L^2(\Omega)}$ & order \\
    \hline  
    \multirow{3}{*}{$1$} 
    & $1/4 $  & 3.495e-0 & - & 2.865e-1 & - \\
    \cline{2-6}
    & $1/8 $  & 2.075e-0 & 0.75 & 1.278e-1 & 1.16 \\ 
    \cline{2-6}
    & $1/16$  & 1.062e-0 & 0.97 & 5.081e-2 & 1.33 \\ 
    \hline 
    \multirow{3}{*}{$2$} 
    & $1/4 $  & 2.395e-0 & -   & 1.509e-1 & - \\
    \cline{2-6}
    & $1/8 $  & 5.402e-1 & 2.15 & 3.259e-2 & 2.21 \\ 
    \cline{2-6}
    & $1/16$  & 1.367e-1 & 1.98 & 7.653e-3 & 2.09 \\ 
    \hline 
    \multirow{3}{*}{$3$} 
    & $1/4 $  & 4.640e-1 & -   & 2.307e-2 & - \\
    \cline{2-6}
    & $1/8 $  & 5.382e-2 & 3.10 & 1.747e-3 & 3.72 \\ 
    \cline{2-6}
    & $1/16$  & 6.649e-3 & 3.02 & 1.255e-4 & 3.81 \\ 
    \hline\hline
  \end{tabular}
  \caption{The numerical results of $\unorm{\bm{u} - \bm{u}_h} $ and
  $\|\bm{u} - \bm{u}_h \|_{L^2(\Omega)}$ for Example 4. }
  \label{tab_ex4u}
\end{table}

\paragraph{\textbf{Example 5}.}  In the last example, we solve another
three-dimensional test problem. The domain and the meshes are
selected the same as the previous example. For this test, the exact
solution is 
\begin{displaymath}
  \bm{u}(x, y, z) = \begin{bmatrix}
    \sin(\pi x) \cos(\pi y) e^{-2z} \\
    \cos(\pi x) \sin(\pi y) e^{-2z} \\ 
    \pi \cos(\pi x) \cos(\pi y) e^{-2z} \\
  \end{bmatrix}, \quad p(x, y, z) = x^2 + y^2 + z^2 - 1,
\end{displaymath}
and the source term and the boundary data are taken from $\bm{u}$ and
$p$. We list the numerical errors in approximation to the gradient of
the velocity and pressure in Tab.~\ref{tab_ex5Up} and the numerical
errors of the velocity are given in Tab.~\ref{tab_ex5u}. It can be
clearly seen that the convergence rates for all unknowns are optimal
under energy norms, and the old/even situation of $L^2$ errors is also
observed. In addition, we plot the numerical solution with the mesh
resolution $h = 1/16$ and the accuracy $m = 3$ to show the velocity
field and the contour of  $| \bm{u}_h |$ in Fig.~\ref{fig_ex5slice}.

\begin{figure}[htp]
  \centering
  \includegraphics[height=1.8in]{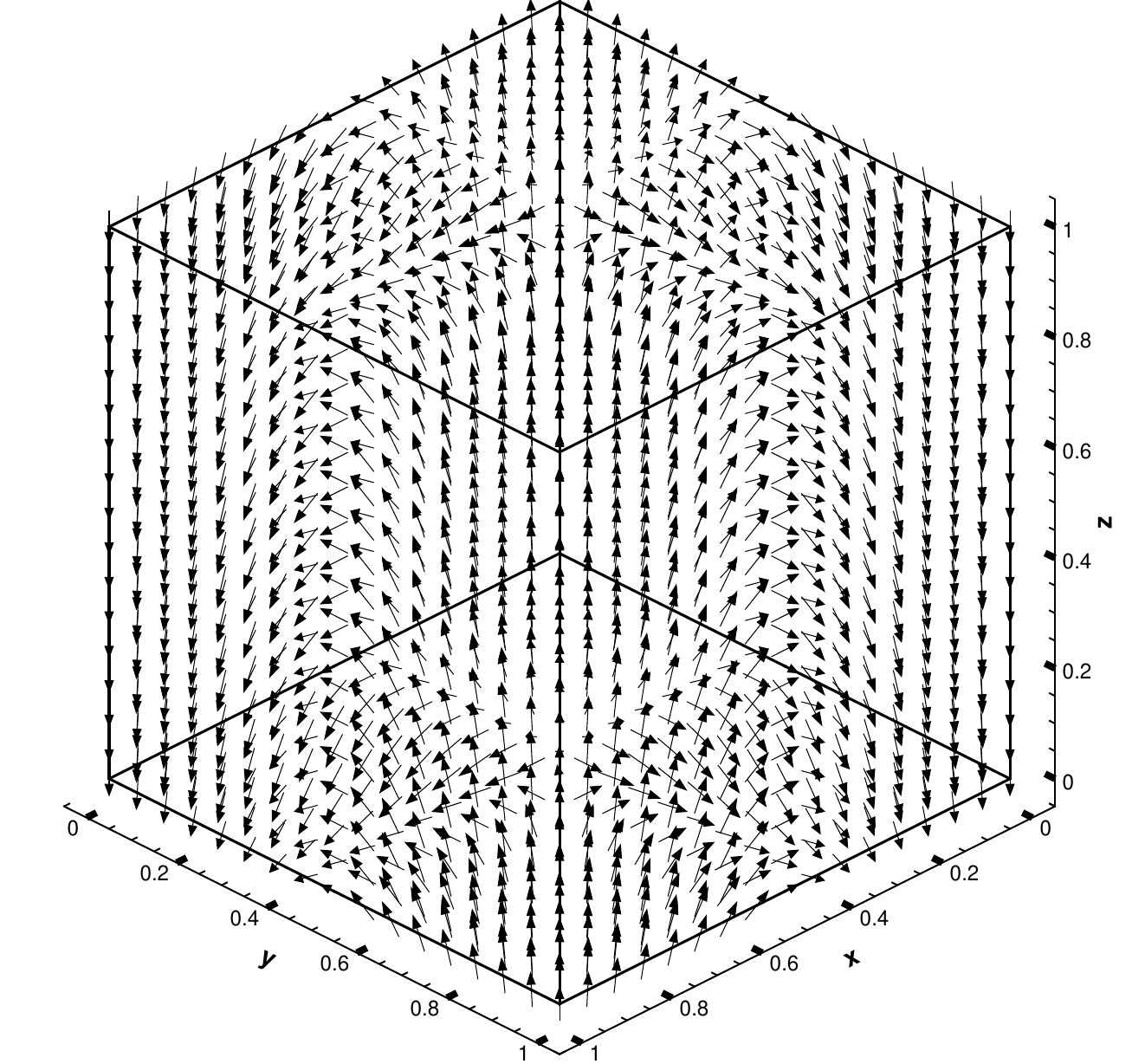}
  \hspace{35pt}
  \includegraphics[height=1.8in]{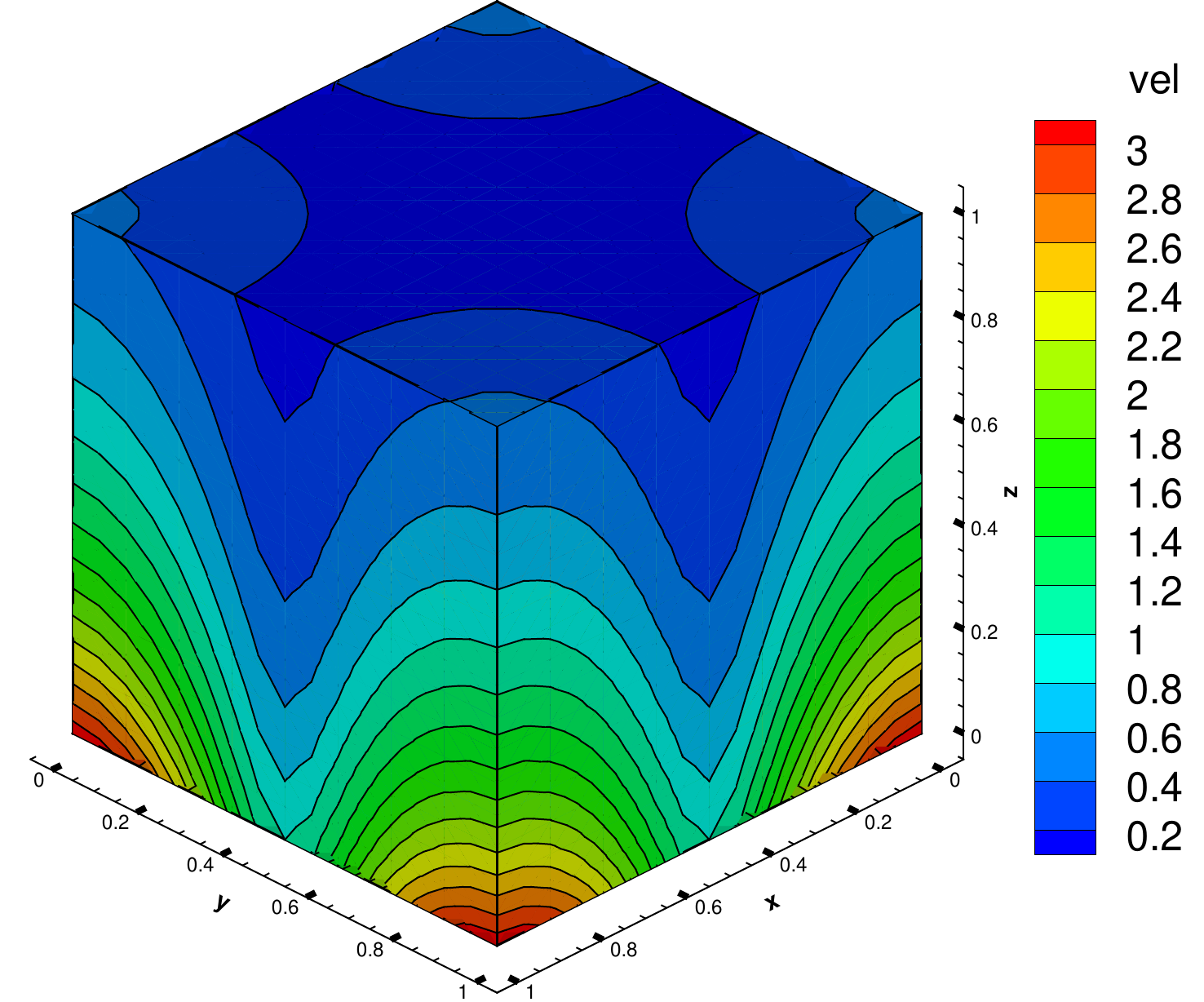}
  \caption{The velocity field (left) / the contour of $ | \bm{u}_h|$
  (right) for Example 5.}
  \label{fig_ex5slice}
\end{figure}

\begin{table}
  \centering
  \renewcommand\arraystretch{1.2}
  \scalebox{0.93}{
  \begin{tabular}{p{0.3cm} | c | c | c | c |
    c |c | c  }
    \hline\hline
    $m$ & $h$ &  $\Unorm{\bmr{U} - \bmr{U}_h} + \pnorm{p - p_h} $ &
    order & $\| \bmr{U} - \bmr{U}_h \|_{L^2(\Omega)}$ & order & $ \| p
    - p_h \|_{L^2(\Omega)}$ & order \\
    \hline
    \multirow{3}{*}{$1$} 
    & $1/4$ & 7.312e-0  & -  & 3.516e-1 & - & 1.932e-1 & - \\
    \cline{2-8}
    & $1/8$ & 3.691e-0 & 0.98  & 1.101e-1 & 1.67 & 9.032e-2 & 1.10 \\
    \cline{2-8}
    & $1/16$ & 1.831e-0 & 1.01  & 3.105e-2 & 1.82 & 3.132e-2 & 1.53 \\
    \hline
    \multirow{3}{*}{$2$} 
    & $1/4$ & 1.946e-0  & -  & 5.619e-2 & - & 7.300e-2 & - \\
    \cline{2-8}
    & $1/8$ & 5.082e-1 & 1.93  & 1.163e-2 & 2.27 & 1.921e-2 & 1.92 \\
    \cline{2-8}
    & $1/16$ & 1.287e-1 & 1.98  & 2.646e-3 & 2.12 & 4.932e-3 & 1.98 \\
    \hline
    \multirow{3}{*}{$3$} 
    & $1/4$ & 7.028e-1  & -  & 2.878e-2 & - & 2.609e-2 & - \\
    \cline{2-8}
    & $1/8$ & 7.992e-2 & 3.13  & 1.881e-3 & 3.93 & 1.848e-3 & 3.82 \\
    \cline{2-8}
    & $1/16$ & 9.250e-4 & 3.11  & 1.181e-4 & 4.01 & 1.093e-4 & 4.08 \\
    \hline\hline
  \end{tabular}}
  \caption{The numerical results of  $\Unorm{\bmr{U} - \bmr{U}_h} +
  \pnorm{p - p_h} $, $\| \bmr{U} - \bmr{U}_h
  \|_{L^2(\Omega)}$ and  $ \| p - p_h \|_{L^2(\Omega)}$ for
  Example 5.}
  \label{tab_ex5Up}
\end{table}

\begin{table}
  \centering
  \renewcommand\arraystretch{1.2}
  \begin{tabular}{ p{0.3cm} | c | c | c |c | c  }
    \hline\hline 
    $m$ & $h$ &  $\unorm{\bm{u} - \bm{u}_h} $ & order & $\|\bm{u} -
    \bm{u}_h \|_{L^2(\Omega)}$ & order \\
    \hline  
    \multirow{3}{*}{$1$} 
    & $1/4 $  & 1.217e-0 & - & 6.262e-2 & - \\
    \cline{2-6}
    & $1/8 $  & 6.123e-1 & 0.99 & 2.155e-2 & 1.53 \\ 
    \cline{2-6}
    & $1/16$  & 3.041e-1 & 1.00 & 6.256e-3 & 1.78 \\ 
    \hline 
    \multirow{3}{*}{$2$} 
    & $1/4 $  & 3.311e-1 & -   & 1.839e-2 & - \\
    \cline{2-6}
    & $1/8 $  & 8.045e-2 & 2.04 & 3.552e-3 & 2.37 \\ 
    \cline{2-6}
    & $1/16$  & 1.975e-2 & 2.02 & 7.529e-4 & 2.23 \\ 
    \hline 
    \multirow{3}{*}{$3$} 
    & $1/4 $  & 1.157e-1 & -   & 6.098e-3 & - \\
    \cline{2-6}
    & $1/8 $  & 1.336e-2 & 3.11 & 3.874e-4 & 3.97 \\ 
    \cline{2-6}
    & $1/16$  & 1.561e-3 & 3.09 & 2.357e-5 & 4.03 \\ 
    \hline\hline
  \end{tabular}
  \caption{The numerical results of $\unorm{\bm{u} - \bm{u}_h} $ and
  $\|\bm{u} - \bm{u}_h \|_{L^2(\Omega)}$ for Example 5. }
  \label{tab_ex5u}
\end{table}

% vim:spell:tw=70:fo+=Mn:cc=70
\section{Conclusion}
\label{sec_conclusion}
We constructed three types of approximation spaces by patch
reconstructions. These reconstructed discontinuous spaces allow us to
numerically solve the Stokes problem in two sequential steps. In the
three spaces, the gradient of velocity, the velocity and the pressure
are approximated, respectively. We first employed a reconstructed
space that consists of piecewise curl-free polynomials with zero trace
to approximate the gradient of the velocity and the pressure. Then we
obtained the approximation to the velocity in the reconstructed
piecewise divergence-free space. The convergence rates for all
unknowns under $L^2$ norms and energy norms are derived. We presented
a series of numerical tests in two and three dimensions to verify the
error estimates and illustrate the great flexibility of the method we
proposed. In addition, the computer program is able to handle
approximation spaces of any high order and the elements with various
geometry in a uniform manner.
%The program using polynomials with different degrees can be
%implemented in a uniform manner.

\section*{Acknowledgements}
This research was supported by the Science Challenge Project (No.
TZ2016002) and the National Natural Science Foundation in China (No.
11971041 and 11421101).

%%% Local Variables:
%%% mode: latex
%%% TeX-master: "article"
%%% End:

% vim:spell:tw=70:fo+=Mn:cc=70
\begin{appendix}
  \section{}
  \label{sec_appendix}
  In Appendix, we present the detailed computer implementation of
  constructing the approximation spaces introduced in Section
  \ref{sec_space}. The construction contains two steps that are
  constructing element patches and solving local least squares
  problems on every element. We first give the recursive algorithm to
  the construction of the element patch, see Alg.~\ref{alg_patch}.

  \begin{algorithm}[htb]
    \caption{Construction to Element Patch}
    \label{alg_patch}
    \begin{algorithmic}[1]
      \renewcommand{\algorithmicrequire}{\textbf{Input:}}
      \REQUIRE
      a partition $\MTh$ and a threshold $\# S$; \\
      \renewcommand{\algorithmicrequire}{\textbf{Output:}}
      \REQUIRE
      the element patch of each element in $\MTh$; \\
      \FOR{every $K \in \MTh$}
      \STATE{initialize $t=0$, $S_t(K) = \left\{ K \right\}$;}
      \WHILE{the cardinality of $S_t(K) < \# S$}
      \STATE{set $S_{t+1}(K) = S_t(K)$;}
      \FOR{every $\wt{K} \in S_t(K)$}
      \STATE{add all adjacent face-neighbouring elements of $\wt{K}$
      to $S_{t+1}(K)$;}
      \ENDFOR
      \STATE{let $t = t+1$;}
      \ENDWHILE
      \STATE{collect collocation points of all elements in $S_t(K)$ in
      $\mc{I}(K)$;}
      \STATE{sort the distances between points in $\mc{I}(K)$ and
      $\bm{x}_K$;}
      \STATE{select the $\# S$ smallest values and collect the
      corresponding elements to form $S(K)$;}
      \ENDFOR
    \end{algorithmic}
  \end{algorithm}

  Then we will explain how to solve the least squares problems
  \eqref{eq_scalarls}, \eqref{eq_vectorlsproblem} and
  \eqref{eq_tensorlsproblem}. The key of solving least squares
  problems is to construct a group of polynomial bases that satisfy
  the constraints in \eqref{eq_vectorlsproblem} and
  \eqref{eq_tensorlsproblem}. For \eqref{eq_vectorlsproblem}, we shall
  construct the bases of the divergence-free polynomial space, that is
  $\mb{P}_k(D)^d \cap \bmr{S}^0(D)$, here $D$ is a bounded domain and
  $k$ is a positive integer. In two dimensions, we can directly take
  the curl of the natural polynomial bases 
  \begin{displaymath}
    1, x, y, x^2, xy, y^2, x^3, x^2y, xy^2, y^3, \ldots
  \end{displaymath}
  to obtain the bases of divergence-free polynomials, as illustrated
  in \cite{baker1990piecewise}. For example, if the linear accuracy is
  considered we have that 
  \begin{displaymath}
    \mb{P}_1(D)^2 \cap \bmr{S}^0(D) = \text{span} \left\{ \vect{1}{0},
    \vect{0}{1}, \vect{0}{x}, \vect{y}{0}, \vect{x}{-y} \right\}.
  \end{displaymath}
  For the second-order case, it is easy to find that 
  \begin{displaymath}
    \mb{P}_2(D)^2 \cap \bmr{S}^0(D) = \text{span} \left\{
    \vect{1}{0}, \vect{0}{1}, \vect{0}{x}, \vect{y}{0}, \vect{x}{-y},
    \vect{0}{x^2}, \vect{x^2}{-2xy}, \vect{-2xy}{y^2}, \vect{y^2}{0}
    \right\}.
  \end{displaymath}

  To get a group of divergence-free polynomials bases in three
  dimensions is a bit more complicated and  we outline a method
  which is easy to implement. We construct two groups of polynomials
  $\wt{\bmr{S}}_k^1(D)$ and $\wt{\bmr{S}}_k^2(D)$ whose union actually
  forms a group of bases. The first group $\wt{\bmr{S}}_k^1(D)$
  consists of the vector-valued polynomials which only have one
  nonzero entry.  Specifically speaking, $\wt{\bmr{S}}_k^1(D)$ has
  three types of polynomials, that is $\wt{\bmr{S}}_k^1(D) =
  \bmr{Q}_k^1(D) \cup \bmr{Q}_k^2(D) \cup \bmr{Q}_k^3(D)$, where 
  \begin{equation}
    \begin{aligned}
      \bmr{Q}_k^1(D) := \left\{ \bm{p} = (p_1, 0, 0)^T \in
      \mb{P}_k(D)^3 \ |\  p_1(y, z) \in \mb{P}_k(y, z)\right\}, \\
      \bmr{Q}_k^2(D) := \left\{ \bm{p} = (0, p_2, 0)^T \in
      \mb{P}_k(D)^3 \ |\  p_2(x, z) \in \mb{P}_k(x, z)\right\}, \\
      \bmr{Q}_k^3(D) := \left\{ \bm{p} = (0, 0, p_3)^T \in
      \mb{P}_k(D)^3 \ |\  p_3(y, z) \in \mb{P}_k(x, y)\right\}, \\
%      \bmr{Q}_k^2(D) := \left\{ \bm{p} = (p_1, p_2, p_3)^T \in
      %\mb{P}_k(D)^3 \ |\  p_1 = p_3 = 0, \  p_2(x, z) \in \mb{P}_k(x,
      %z)\right\}, \\
      %\bmr{Q}_k^3(D) := \left\{ \bm{p} = (p_1, p_2, p_3)^T \in
      %\mb{P}_k(D)^3 \ |\  p_1 = p_2 = 0, \  p_3(x, y) \in \mb{P}_k(x,
      %y)\right\}, \\
    \end{aligned}
    \label{eq_S123}
  \end{equation}
  and $\mb{P}_k(a, b)$ denotes the polynomial space of degree $k$
  based on the coordinate $(a, b)$,
  \begin{displaymath}
    \mb{P}_k(a, b) = \text{span}\left\{ 1, a, b, a^2, ab, b^2, \ldots,
    a^k, a^{k - 1}b, \ldots, ab^{k - 1}, b^k \right\}.
  \end{displaymath}
  Then $\wt{\bmr{S}}_k^2(D)$ has two types of polynomials, that is
  $\wt{\bmr{S}}_k^2(D) = \bmr{R}_k^1(D) \cup \bmr{R}_k^2(D)$, where 
  \begin{equation}
    \begin{aligned}
      \bmr{R}_k^1(D) := \bigg\{  \bm{p} = (p_1, p_2, 0)^T \in
      \mb{P}_k(D)^3 \  \big|\   p_1  = x^t q, & \  p_2 = - \int q(y,
      z) dy, \\
      q(y, z) &\in \mb{P}_{k - t}(y, z), \ 1 \leq t \leq k \bigg\}, \\
      \bmr{R}_k^2(D) := \bigg\{  \bm{p} = (p_1, 0, p_3 )^T \in
      \mb{P}_k(D)^3 \  \big|\   p_1  = x^t q, & \  p_3 = - \int q(y,
      z) dz, \\
      q(y, z) &\in \mb{P}_{k - t}(y, z), \ 1 \leq t \leq k \bigg\}. \\
%      \bmr{R}_k^2(D) := \left\{  \bm{p} = (p_1, p_2, p_3)^T \in 
      %\mb{P}_k(D)^3 \ | \  p_2 = 0, \ p_1 = xq, \  q(y, z) \in
      %\mb{P}_{k - 1}(y, z), \ p_3 = - \int q(y, z) dz  \right\}. \\
    \end{aligned}
    \label{eq_R12}
  \end{equation}
  It is trivial to verify that the polynomials in
  $\wt{\bmr{S}}_k^1(D)$ and $\wt{\bmr{S}}_k^2(D)$ are divergence-free,
  and we state the following lemma. 
  \begin{lemma}
    The divergence-free polynomial space $\mb{P}_k(D)^3 \cap
    \bmr{S}^0(D)  $ satisfies that $\mb{P}_k(D)^3 \cap \bmr{S}^0(D) =
    \wt{\bmr{S}}_k^1(D) \cup \wt{\bmr{S}}_k^2(D)$. 
    \label{le_solenoidbasis}
  \end{lemma}
  \begin{proof}
    We let $\bm{q} \in \mb{P}_k(D)^3$ such that $\bm{q} =
    \wt{\bmr{S}}_k^1(D) \cap \wt{\bmr{S}}_k^2(D)$. By the definition
    \eqref{eq_S123}, the first entry of $\bm{q}$ only depends on $y$
    and $z$. From \eqref{eq_R12}, the first entry of $\bm{q}$ must
    rely on $x$, which gives $\bm{q} = \bm{0}$. Hence, we have that
    $\wt{\bmr{S}}_k^1(D) \cap \wt{\bmr{S}}_k^2(D) = \left\{ \bm{0}
    \right\}$. From \eqref{eq_S123} and \eqref{eq_R12}, we can know
    that $\text{dim}(\wt{\bmr{S}}_k^1(D)) = 3(k+2)(k+1)/2$ and
    $\text{dim}(\wt{\bmr{S}}_k^2(D)) = (k+2)(k+1)k/3$. By
    \cite{baker1990piecewise}, we have that $\text{dim}(\mb{P}_k(D)^3
    \cap \bmr{S}^0(D)) = 3C_{k+3}^3 - C_{k+2}^3$, which exactly
    implies $\text{dim}(\mb{P}_k(D)^3 \cap \bmr{S}^0(D)) =
    \text{dim}(\wt{\bmr{S}}_k^1(D)) +
    \text{dim}(\wt{\bmr{S}}_k^2(D))$. This fact gives us that $
    \mb{P}_k(D)^3 \cap \bmr{S}^0(D) = \wt{\bmr{S}}_k^1(D) \cup
    \wt{\bmr{S}}_k^2(D)$ and completes the proof. 
  \end{proof}
  Further, we give an example of the linear accuracy. In this case, we
  can obtain that 
  \begin{displaymath}
    \wt{\bmr{S}}_1^1(D) = \text{span} \left\{ \vech{1}{0}{0},
    \vech{0}{1}{0}, \vech{0}{0}{1}, \vech{y}{0}{0}, \vech{z}{0}{0},
    \vech{0}{x}{0}, \vech{0}{z}{0}, \vech{0}{0}{y}, \vech{0}{0}{z}
    \right\},
  \end{displaymath}
  and 
  \begin{displaymath}
    \wt{\bmr{S}}_2^2(D) = \text{span} \left\{ \vech{x}{-y}{0},
    \vech{x}{0}{-z} \right\}.
  \end{displaymath}
  Hence, $\mb{P}_1(D)^3 \cap \bmr{S}^0(D) =  \wt{\bmr{S}}_1^1(D) \cup
  \wt{\bmr{S}}_1^2(D)$.  
  
  Then we consider to solve the problem \eqref{eq_tensorlsproblem},
  which requires us to construct the polynomial space consists of the
  curl-free polynomials with zero trace, that is $\mb{P}_k(D)^{d
  \times d} \cap \bmr{I}^0(D) $. Actually after obtaining the bases of
  the divergence-free polynomial space, it is easy to get the bases of
  the polynomial space $\mb{P}_k(D)^{d \times d} \cap \bmr{I}^0(D)$.
  We can take the gradient of the divergence-free polynomial bases to
  get those bases. Again we take $k=1$ for an example and we can
  obtain that 
  \begin{displaymath}
    \mb{P}_1(D)^{d \times d} \cap \bmr{I}^0(D) = \text{span} \left\{
    \matt{0}{0}{1}{0}, \matt{0}{1}{0}{0}, \matt{1}{0}{0}{-1},
    \matt{0}{0}{x}{0}, \matt{x}{0}{-y}{-x}, \matt{-y}{-x}{0}{y},
    \matt{0}{y}{0}{0} \right\}.
  \end{displaymath}

  In the rest of Appendix, we present the details of the computer
  implementation to the reconstructed space. Let us construct the
  space $\bmr{I}_h^1$ in two dimensions as an illustration.  We
  consider the element $K_0$ and we let its element patch $S(K_0)$
  formed by its face-neighbouring elements, see
  Fig.~\ref{fig_Kexample}. 

  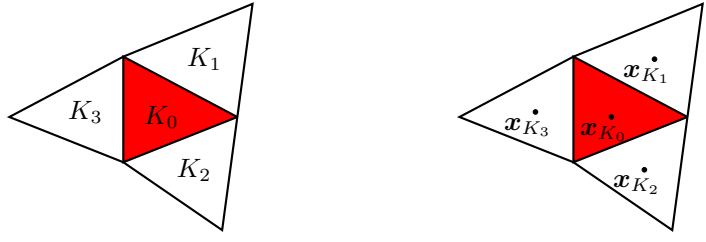
\begin{figure}[htp]
    \centering
    \begin{tikzpicture}[scale=1]
      \coordinate (A) at (1, 0); 
      \coordinate (B) at (-0.5, -0.6);
      \coordinate (C) at (-0.5, 0.8);
      \coordinate (D) at (1.2, 1.5);
      \coordinate (E) at (-2, 0);
      \coordinate (F) at (0.8, -1.5);
      \draw[fill, red] (A) -- (B) -- (C);
      \draw[thick, black] (A) -- (C) -- (B) -- (A);
      \draw[thick, black] (A) -- (D) -- (C);
      \draw[thick, black] (C) -- (E) -- (B);
      \draw[thick, black] (A) -- (F) -- (B);
      \node at(0, 0) {$K_0$}; \node at(1.7/3, 2.3/3) {$K_1$};
      \node at(1.3/3, -2.1/3) {$K_2$}; \node at(-1, 0.2/3) {$K_3$};
    \end{tikzpicture}
    \hspace{70pt}
    \begin{tikzpicture}[scale=1]
      \coordinate (A) at (1, 0); 
      \coordinate (B) at (-0.5, -0.6);
      \coordinate (C) at (-0.5, 0.8);
      \coordinate (D) at (1.2, 1.5);
      \coordinate (E) at (-2, 0);
      \coordinate (F) at (0.8, -1.5);
      \draw[fill, red] (A) -- (B) -- (C);
      \draw[thick, black] (A) -- (C) -- (B) -- (A);
      \draw[thick, black] (A) -- (D) -- (C);
      \draw[thick, black] (C) -- (E) -- (B);
      \draw[thick, black] (A) -- (F) -- (B);
      \draw[fill, black] (0, 0) circle [radius=0.03]; 
      \draw[fill, black] (1.7/3, 2.3/3)  circle [radius=0.03]; 
      \draw[fill, black] (1.3/3, -2.1/3) circle [radius=0.03];  
      \draw[fill, black] (-1, 0.2/3)     circle [radius=0.03];
      \node at(-0.1, -0.2) {$\bm{x}_{K_0}$}; 
      \node at(1.7/3-0.1, 2.3/3-0.2)  {$\bm{x}_{K_1}$};
      \node at(1.3/3-0.1, -2.1/3-0.2) {$\bm{x}_{K_2}$}; 
      \node at(-1-0.1, 0.2/3-0.2)     {$\bm{x}_{K_3}$};
    \end{tikzpicture}
    \caption{$K_0$ and its neighbours (left) / collocation points
    (right).}
    \label{fig_Kexample}
  \end{figure}

  For a tensor-valued function $\bm{g} = (g_{00}, g_{01}; g_{10},
  g_{11})^T \in \bmr{U}_h (g_{00} = -g_{11})$, the least squares
  problem \eqref{eq_vectorlsproblem} on $S(K_0)$ takes the form 
  \begin{equation}
    \begin{aligned}
      \wh{\mc{R}}_{K_0}^1 \bm{g} =  \mathop{\arg \min}_{\bm{q} \in
      \mb{P}_m(S(K_0))^{2 \times } \cap \bmr{I}^0(S(K_0))}
      \sum_{\bm{x} \in \mc{I}(K_0)} \| \bm{q}(\bm{x}) -
      \bm{g}(\bm{x}) \|^2_{l^2 \times l^2}, \quad 
      \text{s.t. } % \begin{cases}
        \bm{q}(\bm{x}_{K_0}) = \bm{g}(\bm{x}_{K_0}).
        %\nabla \times \bm{q} = 0,  \\
        %\tr{\bm{q}} = 0. \\
        %\end{cases}
    \end{aligned}
    \label{eq_lsproblemK0}
  \end{equation}
  From the bases of the polynomial space $\mb{P}_1(S(K_0))^{2}
  \cap \bmr{I}^0(S(K))$, the polynomial $\bm{p}(\bm{x})$ in
  \eqref{eq_lsproblemK0} has the form 
  \begin{displaymath}
    \begin{aligned}
      \bm{p}(\bm{x}) = a_0 \matt{1}{0}{0}{-1} +  a_1 \matt{0}{1}{0}{0}
      + & a_2 \matt{0}{0}{1}{0} \\ 
      + a_3 \matt{0}{0}{x}{0}  + & a_4 \matt{x}{0}{-y}{-x} + a_5
      \matt{-y}{-x}{0}{y} + a_6 \matt{0}{y}{0}{0}. \\
%      \bm{p}(\bm{x}) = \matt{g_{00}(\bm{x}_{K_0})
      %}{g_{01}(\bm{x}_{K_0})}{g_{10}(\bm{x}_{K_0})
      %}{g_{11}(\bm{x}_{K_0}) } + &
      %a_0 \matt{0}{0}{x - x_{K_0} }{0}  +  a_1 \matt{x - x_{K_0}
      %}{0}{-y + y_{K_0} }{-x + x_{K_0} } \\ 
      %+&  a_2 \matt{-y + y_{K_0}}{-x + x_{K_0} }{0}{y - y_{K_0} }
      %+ a_3 \matt{0}{y - y_{K_0} }{0}{0}, \\
    \end{aligned}
  \end{displaymath}
  By the constraint in \eqref{eq_lsproblemK0}, we can know the values
  of $a_0$, $a_1$ and $a_2$ and we rewrite the polynomial
  $\bm{p}(\bm{x})$ as 
  \begin{displaymath}
    \begin{aligned}
      \bm{p}(\bm{x}) = \matt{g_{00}(\bm{x}_{K_0})
      }{g_{10}(\bm{x}_{K_0})}{g_{01}(\bm{x}_{K_0})
      }{g_{11}(\bm{x}_{K_0}) } + &
      a_0 \matt{0}{0}{x - x_{K_0} }{0}  +  a_1 \matt{x - x_{K_0}
      }{0}{-y + y_{K_0} }{-x + x_{K_0} } \\ 
      +&  a_2 \matt{-y + y_{K_0}}{-x + x_{K_0} }{0}{y - y_{K_0} }
      + a_3 \matt{0}{y - y_{K_0} }{0}{0}, \\
    \end{aligned}
  \end{displaymath}
  where $\bm{x}_{K_i} = (x_{K_i}, y_{K_i})(0 \leq i \leq 3)$. Thus
  the problem \eqref{eq_lsproblemK0} is equivalent to 
  \begin{equation}
    \begin{aligned}
      \mathop{\arg \min}_{a_3, a_4, a_5, a_6 \in \mb{R}}  \sum_{i
      = 1}^3 & \bigg\|  a_3 \matt{0}{0}{x_{K_i} - x_{K_0} }{0}  +
      a_4 \matt{x_{K_i} - x_{K_0} }{0}{-y_{K_i} + y_{K_0} }{-x_{K_i} +
      x_{K_0} } + a_5 \matt{-y_{K_i} + y_{K_0}}{-x_{K_i} + x_{K_0}
      }{0}{y_{K_i} - y_{K_0} } \\ 
      + & a_6 \matt{0}{y_{K_i} - y_{K_0} }{0}{0} -
      \matt{g_{00}(\bm{x}_{K_i}) - g_{00}(\bm{x}_{K_0})
      }{g_{10}(\bm{x}_{K_i}) - g_{10}(\bm{x}_{K_0})}{
      g_{01}(\bm{x}_{K_i}) - g_{01}(\bm{x}_{K_0})
      }{g_{11}(\bm{x}_{K_i}) - g_{11}(\bm{x}_{K_0}) } \bigg\|^2_{l^2
      \times l^2}.
    \end{aligned}
    \label{eq_a0a3}
  \end{equation}
  The solution to \eqref{eq_a0a3} reads
  \begin{displaymath}
    \begin{bmatrix}
      a_3 \\ a_4 \\ a_5 \\ a_6 
    \end{bmatrix} = (A^TA)^{-1} A^T \begin{bmatrix}
      2(g_{00}(\bm{x}_{K_1}) - g_{00}(\bm{x}_{K_0})) \\ 
      g_{01}(\bm{x}_{K_1}) - g_{01}(\bm{x}_{K_0}) \\
      g_{10}(\bm{x}_{K_1}) - g_{10}(\bm{x}_{K_0}) \\
      %g_{11}(\bm{x}_{K_1}) - g_{11}(\bm{x}_{K_0}) \\
      \cdots \\ 
      2 (g_{00}(\bm{x}_{K_3}) - g_{00}(\bm{x}_{K_0})) \\ 
      g_{01}(\bm{x}_{K_3}) - g_{01}(\bm{x}_{K_0}) \\
      g_{10}(\bm{x}_{K_3}) - g_{10}(\bm{x}_{K_0}) \\
      %g_{11}(\bm{x}_{K_3}) - g_{11}(\bm{x}_{K_0}) \\
    \end{bmatrix}, 
  \end{displaymath}
  where 
  \begin{displaymath}
    A =  \begin{bmatrix}
      0 &  2(x_{K_1} - x_{K_0}) &  - 2(y_{K_1} - y_{K_0}) &  0 \\ 
      x_{K_1} - x_{K_0} &  -y_{K_1} + y_{K_0} & 0 & 0   \\
      0 & 0 &  -x_{K_1} + x_{K_0} &  y_{K_1} - y_{K_0}  \\ 
      %0 & -x_{K_1} + x_{K_0} &  y_{K_1} - y_{K_0}  & 0  \\
      \cdots & \cdots &\cdots &\cdots  \\ 
      0 &  2(x_{K_3} - x_{K_0}) &  - 2(y_{K_3} - y_{K_0}) &  0 \\ 
      x_{K_3} - x_{K_0} &  -y_{K_3} + y_{K_0} & 0 & 0   \\
      0 & 0 &  -x_{K_3} + x_{K_0} &  y_{K_3} - y_{K_0}  \\ 
      %0 & -x_{K_3} + x_{K_0} &  y_{K_3} - y_{K_0}  & 0  \\
    \end{bmatrix}.
  \end{displaymath}
  By rearrangement, we can obtain the solution to
  \eqref{eq_lsproblemK0}, which takes the form
  \begin{displaymath}
    \begin{bmatrix}
      a_0 \\ a_1 \\a_2 \\a_3 \\a_4 \\a_5 \\ a_6 
    \end{bmatrix} = \begin{bmatrix}
      I_{3 \times 3}  & 0 \\
      -M I_{9 \times 3}  & M \\ 
    \end{bmatrix} \begin{bmatrix}
      g_{00}(\bm{x}_{K_0}) \\
      g_{01}(\bm{x}_{K_0}) \\
      g_{10}(\bm{x}_{K_0}) \\
      \cdots \\
      g_{00}(\bm{x}_{K_3}) \\
      g_{01}(\bm{x}_{K_3}) \\
      g_{10}(\bm{x}_{K_3}) \\
    \end{bmatrix}, \quad M = (A^TA)^{-1}A^T \begin{bmatrix}
      2 & 0              & 0  & 0 & 0 & 0\\ 
      0 & I_{2 \times 2} & 0  & 0 & 0 & 0\\ 
      0 & 0              & 2  & 0 & 0 & 0 \\
      0 & 0 & 0  & I_{2 \times 2} & 0 & 0 \\
      0 & 0 & 0  & 0 & 2 & 0  \\
      0 & 0 & 0  & 0 & 0 & I_{2 \times 2}\\
    \end{bmatrix}, 
  \end{displaymath}
  where $I_{2 \times 2}$ and $I_{3 \times 3}$ are $2 \times 2$
  identity matrix and $3 \times 3$ identity matrix and $I_{9 \times 3}
  = (I_{3 \times 3}, I_{3 \times 3}, I_{3 \times 3})^T$. We note that
  the collocation points in $\mc{I}(K)$ totally determine the matrix
  $M$, and by the expansion \eqref{eq_tensorRmg}, the coefficient
  matrix 
  \begin{equation}
    \begin{bmatrix}
      I_{3 \times 3}  & 0 \\
      -M I_{9 \times 3}  & M \\ 
    \end{bmatrix} 
    \label{eq_coematrix}
  \end{equation}
  actually contains all information of the basis functions
  $\wh{\bm{\lambda}}_{K_i}^{j ,k}(0 \leq i \leq 3, 1 \leq j, k \leq 2,
  j + k < 4)$ on the element $K_0$. Then we can use the coefficient
  matrix \eqref{eq_coematrix} on each element to represent the
  reconstructed space $\bmr{I}^1_h$. For the spaces $U_h^1$ and
  $\bmr{S}_h^1$, their constructions are very similar.  In addition,
  such a computer implementation can be easily adapted to three
  dimensions and the case when higher-order accuracy is considered. 
\end{appendix}

\bibliographystyle{amsplain}
\bibliography{../ref}

\end{document}